\documentclass[11pt,reqno]{amsart}
\usepackage[utf8]{inputenc}
\usepackage[T1]{fontenc}
\usepackage{graphicx} 

\usepackage{csquotes,nicefrac}
\usepackage{thmtools}

\usepackage{todonotes}
\usepackage[normalem]{ulem}

\usepackage{xcolor,mathrsfs,amsmath,mathtools,amssymb,hyperref}
\def\E{\mathbb E}

\usepackage[capitalise,nameinlink]{cleveref} 
\crefname{enumi}{item}{items}
\crefname{equation}{}{}
\crefname{subsection}{Subsection}{Subsections}

\newcommand{\R}{\mathbb{R}}

\usepackage[backend=biber, style=alphabetic,maxnames=4,
  minnames=4, maxcitenames=4,
  mincitenames=4]{biblatex}
\usepackage{amsthm}
\usepackage{mathabx}
\usepackage{enumitem}
\usepackage{bbm}
\usepackage[english]{babel}
\usepackage{geometry}
 \allowdisplaybreaks
 \usepackage{appendix}
\usepackage{hyperref}
\hypersetup{
colorlinks   = true,
citecolor    = blue,
linkcolor=blue
}

\allowdisplaybreaks

 \usepackage[most]{tcolorbox}
\theoremstyle{definition}
\newtheorem{theorem}{Theorem}[section]
\AfterEndEnvironment{theorem}{\noindent\ignorespaces}
\theoremstyle{definition}
\newtheorem{definition}[theorem]{Definition}
\AfterEndEnvironment{definition}{\noindent\ignorespaces}

\theoremstyle{definition}
\newtheorem{example}[theorem]{Example}
\AfterEndEnvironment{Example}{\noindent\ignorespaces}

\def\l@subsection{\@tocline{2}{0pt}{2.5pc}{5pc}{}}

\theoremstyle{definition}
\newtheorem{proposition}[theorem]{Proposition}
\AfterEndEnvironment{proposition}{\noindent\ignorespaces}

\theoremstyle{definition}
\newtheorem{lemma}[theorem]{Lemma}
\AfterEndEnvironment{lemma}{\noindent\ignorespaces}

\theoremstyle{definition}
\newtheorem{corollary}[theorem]{Corollary}
\AfterEndEnvironment{corollary}{\noindent\ignorespaces}

\newtheorem{remark}[theorem]{Remark}
\AfterEndEnvironment{remark}{\noindent\ignorespaces}

\newtheorem{assumption}
[theorem]{Assumption}
\AfterEndEnvironment{assumption}{\noindent\ignorespaces}



\usepackage{accents}

\DeclareMathOperator*{\Lip}{\operatorname{Lip}}
\DeclareMathOperator*{\loc}{\operatorname{loc}}

\usepackage[english]{babel}
\usepackage{geometry}
 \def\bd{\boldsymbol{\delta}}
\title{Controlled fields, Rough stochastic calculus,\\ and It\^o--Wentzell--Alekseev--Gr\"obner identities}

\author{Jannis R. Dause}
\address{Institut für Mathematik, 
Technische Universität Berlin, Berlin, Germany}
\email{dause@math.tu-berlin.de}

\author{Peter K. Friz}
\address{Institut für Mathematik, 
Technische Universität Berlin and Weierstraß Institut, Berlin, Germany}
\email{friz@math.tu-berlin.de}

\author{Arnulf Jentzen}
\address{School of Data Science \& School of Artificial Intelligence, The Chinese University of Hong Kong, Shenzhen (CUHK-Shenzhen), Shenzhen, China and Institute for Analysis \& Numerics, University of M\"{u}nster, M\"{u}nster, Germany}
\email{ajentzen@cuhk.edu.cn, ajentzen@uni-muenster.de}

\author{Jian Song}
\address{Research Center for Mathematics and Interdisciplinary Sciences, Shandong University, Qingdao 266237, China and Frontiers Science Center for Nonlinear Expectations, Ministry of Education, Shandong University, Qingdao 266237, China}
\email{txjsong@sdu.edu.cn}

\date{\today}


\begin{document}
    
\begin{abstract}
We develop a calculus of space--time controlled fields for rough stochastic
systems. 
This approach provides a unified composition rule for evaluating random fields
along rough semimartingales and yields a rough stochastic It{\^o}--Wentzell
formula under natural and verifiable regularity assumptions.

Our motivation comes from works of Hudde et al.\ (2024) and, independently,
Del~Moral \& Singh (2022) where the authors established, respectively,
It{\^o}--Alekseev--Gr{\"o}bner, backward It{\^o}--Wentzell, and diffusion
interpolation formulas. 
\end{abstract}
\maketitle
\tableofcontents

\section{Introduction}

In his work~{\cite{wentzell65}} from 1965, A.\ D.\ Wentzell
introduced a fundamental extension of the \emph{It{\^o} formula}, today known as the
{{\em It{\^o}--Wentzell formula}}. At its core, it provides a composition
rule for evaluating a (sufficiently regular) random field along a random
trajectory. This composition principle has become a standard tool across
stochastic analysis and its applications; as a small selection we point to 
stochastic partial differential equation (SPDE) theory~{\cite{kunita_1997,rozovsky_lototsky_2018}}, fluid dynamics \cite{flandoli_random_2011}, filtering~{\cite{kunita_1997,moral_2021}}, mathematical
finance~{\cite{durrleman2010implied,bayer_qiu_yao22}}, stochastic optimal
control~{\cite{Peng92}}, and mean-field games and control with common
noise~{\cite{cardliaguet_souganidids22,TouziTalbi2025}}. 
Many authors extended Wentzell's original formula in various directions, including
anticipative 
and backward formulations~{\cite{Kunita81,OconePardoux1989,nualart_2006,DELMORAL2022197}}, 
analytic refinements ~{\cite{Krylov2011ito-wentzell}}, 
measure flow variants~{\cite{,GUO2023350,dosReisPlatonov2023}}, and last not least rough paths settings~{\cite{keller_zhang_2016, castrequini_2025}}. 

\medskip

In recent years, we have witnessed the emergence of {{\em rough
stochastic analysis}}, arguably initiated by the first intrinsic well-posedness theory for {{\em rough
stochastic differential equations}} (RSDEs) in the work~{\cite{friz_2021}}, as well as the unifying notion of {{\em rough semimartingales}} (RSM) due to~{\cite{fzk23}}, providing a unifying framework that
simultaneously generalizes Lyons' rough path theory~{\cite{Lyons1998}} and
classical It{\^o} calculus. Technically, for RSDEs, a key role is played by stochastic
sewing arguments~{\cite{Le2020}} (in mixed moment settings, cf.\ $L^{p, q}$
in~{\cite{friz_2021}}) together with the harmonic-analysis rooted 
multilevel Burkholder--Davis--Gundy (BDG)-type estimates for martingale transforms that underlies the RSM approach of~{\cite{fzk23}}. These tools have enabled a rapid expansion of applications, 
including stochastic filtering and robust/pathwise 
control~{\cite{flz24,friz_randomisation_2025-1}}, 
asset pricing and finance~{\cite{bbfp25}}, McKean--Vlasov equations 
with common noise~{\cite{friz_mckean-vlasov_2025,bugini_rough_2025,Bugini2025Nonlinear}},
extensions to jumps~{\cite{allan_rough_2025}}, and 
further links to nonlinear partial differential equations (PDEs)~{\cite{Bugini2025Nonlinear}}.

Given the pervasive role of It{\^o}--Wentzell in classical stochastic analysis
and the emergence of rough stochastic calculus, the need for a {{\em rough
stochastic It{\^o}--Wentzell formula}} has become apparent. However, the
appropriate notion of composition in this regime must (i) accommodate
{{\em rough temporal regularity}}, (ii) retain sufficient {{\em spatial
structure}} to permit evaluation along stochastic flows, and (iii) interact
well with {{\em forward-backward}} and {{\em anticipative}} constructions
which naturally arise in stochastic numerics and more generally the study of diffusion limits.
The present work develops a framework that addresses these requirements in a
unified manner. 
\addtocontents{toc}{\protect\setcounter{tocdepth}{1}}
\subsection{Space-time rough calculus via controlled fields (\cref{section:Introduction to controlled fields})} Motivated by controlled rough paths~{\cite{GUBINELLI2004}} and the
  jet-based view of rough It{\^o} formulas (see, e.g.,~{\cite{friz_2020}}) for {\em strongly controlled} rough paths (think ``second order''  controlled plus drift), as well as inspiration from Hairer's regularity structures \cite{Hairer2014}, we
  introduce in  \cref{section:Introduction to controlled fields} a
  notion of {{\em controlled fields}}: space-time objects which encode
  simultaneously (Lipschitz-type) spatial regularity and (rough) temporal
  regularity in a compact ``jet'' form. In our main theorem of this section, 
  \cref{cor:composition_rule}, we establish stability under composition and 
  a flexible chain rule for these controlled fields. As a
  consequence of this theorem, we obtain rough analytic tools for RDE flows and {\em rough partial differential equations} (RPDEs), and we
  recover a \emph{rough It{\^o}--Wentzell (rIW) formula} as an immediate corollary
  under natural, verifiable regularity assumptions; compare, for instance,
  with the rIW perspectives in~{\cite{keller_zhang_2016,castrequini_2025}}. Notably, the rIW identity emerges from algebraic composition properties without assuming uniqueness of the Gubinelli derivatives or a priori rough integral formulations.
  
\subsection{Moment-free rough stochastic calculus and It{\^o}--Wentzell formula (\cref{section:On Rough Stochastic Calculus})} We review elements of~{\cite{fzk23}}, notably the notion of rough
  semimartingales, adapted to the H\"older framework which is more appropriate for the purposes of this work, and then introduce a second order controlled variant of rough semimartingales,
  dubbed {\em strongly controlled rough semimartingales}, which can be viewed as ``semimartingales variants'' of {\em rough It\^{o} processes}, introduced in \cite{friz_2021}.
 A key insight is that such {\em strongly controlled rough semimartingales} are in non-trivial correspondence to adapted a.s.\ strongly controlled rough paths, provided that the (deterministic) reference rough path is replaced by 
 a semi-stochastic joint lift with suitable martingales. This comes with the significant advantage that ``rough stochastic identities'' can be obtained from ``rough identities'', best appreciated by 
 comparing the level of difficulty in deriving a rough It\^{o} formula, 
 as seen in~{\cite{friz_2020}, versus a rough stochastic It\^{o} formula, 
 as done in \cite{friz_2021}, in rather involved moment spaces of stochastic processes.
 As proof in case, we use this approach to elegantly derive a {{\em rough stochastic It{\^o}--Wentzell formula}} (\cref{theorem:total_RSIW_formula}), essentially as consequence of the rough It{\^o}--Wentzell formula of \cref{section:Introduction to controlled fields}, plus some Kunita-type considerations to allow for an extra $x$-dependent local martingale part $\int \beta(x) dW$.
  Specifically, this theorem provides a composition rule for strongly controlled rough semimartingales, that is, processes of the form 
  \begin{equation*}
      Y_{t}=Y_{0}+\int_{0}^{t} \dot{Y}_{s} ds + M_{t} + \int_{0}^{t} (\partial_{X} Y, \partial_{X}^{2} Y)_{s} d\mathbf{X}_{s} \, ,
  \end{equation*}
  with suitably regular (in space) rough stochastic fields of the form
  \begin{equation*}
      H_{t}(x)=\int_{0}^{t} \dot{F}_{s}(x) ds+ \int_{0}^{t} (F'_{s}(x), F''_{s}(x)) d\mathbf{X}_{s} + \int_{0}^{t} \beta_{s}(x) dW_{s},  
  \end{equation*}
  where $M$ is a martingale and $d\mathbf{X}$ denotes {\em rough stochastic integration} (as exposed in \cref{section:On Rough Stochastic Calculus}). Notably, no controlledness is assumed on any of the martingales. Then \cref{theorem:total_RSIW_formula} asserts the composition rule
  \begin{eqnarray*}
    H_{t}(Y_{t})&=& \int_{0}^{t} \beta_s (Y_s) dW_{s}+ \int_{0}^{t} (DH_{s}(Y_{s})  \dot{Y}_{s} + \dot{F}_{s}(Y_{s}) ) ds + \int_0^t DH_{s} (Y_{s}) 
    dM_{s}  \nonumber \\
    &&+\int_{0}^{t} ( F'_{s}(Y_{s}) + D H_{s}(Y_{s}) \partial_{X} Y_{s} )  d
    \mathbf{X}_{s} +\frac{1}{2} \int_{0}^{t} D^{2} H_{s}(Y_{s}) d\langle M \rangle_{s} \nonumber\\
    &&+\int_{0}^{t}  D F'_{s}(Y_{s})  \partial_{X}Y_{s} ds + \frac{1}{2} \int_{0}^{t}   D^2
    H_{s}(Y_{s}) (\partial_{X}Y_{s}, \partial_{X} Y_{s} )   d [\mathbf{X}]_{s} \nonumber\\
    &&+ \left \langle \int_{0}^{\cdot} D \beta_s (Y_s) dW_{s}, M \right \rangle_{0, t} \nonumber, 
\end{eqnarray*}which is
  tailored to the mixed rough/stochastic regime relevant for
  RSDEs~{\cite{friz_2021}} and rough semimartingale calculus~{\cite{fzk23}}. As an application we give a simple and intrinsic proof of the wellposedness-theory of RSDEs using flow-transforms generalizing results of \cite{CDFO13}.
  
\subsection{Forward-backward stochastic analysis and
  stochastic numerics (\cref{section:Applications to (Rough) Stochastic Analysis})} We discuss
  applications in the context of the {{\em It{\^o}--Alekseev--Gr{\"o}bner}}
  (IAG) formula of~{\cite{HHJM24}}, closely related to the 
  forward-backward stochastic analysis of \cite{DELMORAL2022197}. 
  Specifically, \cite{HHJM24} established a composition rule in Skorokhod integral form 
   \begin{equation}
   \label{eq:intro_IAG}
   \begin{aligned}
   F_{t}(Y_{t})-F_s(Y_{s})= & \int_s^t D F_r
    (Y_r) (b_r - \mu (Y_r)) dr
    + \int_s^t DF_r (Y_r)
    (\beta_r - \sigma (Y_r)) \diamond dW_{r}\\
    & + \frac{1}{2} \int_s^t D^{2}F_{r} (Y_r)\left( ( \beta_{r}, \beta_{r}) -(\sigma(Y_{r}), \sigma(Y_{r})) \right) dr, 
  \end{aligned}  
  \end{equation}
  where $Y$ denotes a generic Itô process $dY_{t} = b_{t} dt + \beta_{t} dW_{t}$ (remarkably, with no Malliavin smoothness assumptions on $\beta$) and 
  $F_{t}(x)\coloneqq f(X_{T}^{t,x})$, for suitable test fields $f$ and (well-posed) SDE solutions flow induced by
      \begin{equation*}
      dX_{s}^{t,x}=\mu(X^{t,x}_{s})ds + \sigma(X^{t,x}_{s}) dW_{s}; \quad X^{t,x}_{t}=x.
      \end{equation*}
  Taking a well-known RDE perspective on this SDE, which can be recovered anytime from Brownian randomization, but keep $Y$ as generic Itô process, we provide
  a first rough path view \eqref{eq:intro_IAG} in \cref{cor:RSAG_formula}, a full understanding of the resulting terms on the right-hand side, after randomization is then
  achieved in \cref{theorem:IAG}, embracing the fact that consistent randomization requires independence (used to much benefit in recent works like ~{\cite{flz24,friz_randomisation_2025-1,bugini_rough_2025}}), which is here only available in a very localized sense, making the analysis very subtle, but still possible.
  This leads us to a (much) simplified proof of the IAG formula \eqref{eq:intro_IAG}, which 
  only requires $\beta, \beta^{\perp}, b \in L^{1+}([0,T]\times \Omega; \operatorname{Leb}\times \mathbb{P})$, reducing integrability assumptions on the It{\^o}
  characteristics of the comparison process $Y$ in the IAG setting, thereby answering positively a conjecture made in~{\cite[Remark~3.2]{HHJM24}}.

As also noted in \cite{DELMORAL2022197}, the IAG formula sits somewhat ``in between'' forward-backward Itô-Wentzell- and stochastic interpolation formulas, discussed in \cite{DELMORAL2022197} and also \cref{sec:csn},
with many pointer to the literature, emphasizing the general use of such methods, notably to stochastic numerics and limit theorems where forward-backward error expansions can be most useful. 
  \par 
In \cref{subsection:stochastic_interpolation_formula} we briefly discuss how rough paths techniques relate to anticipating Stratonovich integration, making the link to \cite{Coutin2007}. In particular, 
we see that the Skorokhod stochastic interpolation formula that appears in \cite{DELMORAL2022197} can also be derived from this perspective. 
\subsection{Comments and comparisons (\cref{sec:comments})} We conclude the main body of this article with a detailed literature review highlighting further the relation to the present work.
\subsection{Supplementary materials (\cref{sec:besov-spaces} and \cref{appendix:section:On RDE-flows for non-autonomous vector fields})}
In \cref{sec:besov-spaces}, we present a generalization of the rough path Kolmogorov criterion \cite[Theorem 3.1]{friz_2020}, which  accommodates  higher order estimates and more intricate algebraic structures, and is applied throughout \cref{section:On Rough Stochastic Calculus}. In \cref{appendix:section:On RDE-flows for non-autonomous vector fields}, we establish results on RDEs along non-autonomous vector fields. In particular, motivated by the notion of stochastic controlled vector fields introduced in \cite{friz_2021} and the controlled fields of \cref{section:Introduction to controlled fields}, we introduce a natural class of controlled fields that serve as coefficient fields for RDEs.  We also provide a quantitative analysis of the regularity of the associated solution flows. Notably, our proofs follow a more functional-analytic approach, in contrast to the limiting arguments employed in \cite[Chapter 11]{Friz_Victoir_2010}.

\medskip 

\noindent 
\textbf{Acknowledgments:}
JRD acknowledges current support by the DFG - Project-ID \\
410208580 - IRTG2544 (“Stochastic Analysis in Interaction”) and former support by the Berlin Mathematical School through a 6-months PhD fellowship. PKF acknowledges funding by the Deutsche Forschungsgemeinschaft (DFG, German Research Foundation) – CRC/TRR 388 ``Rough Analysis, Stochastic Dynamics and Related Fields'' – Project ID 516748464. AJ has been partially supported by the National Science Foundation of China (NSFC) under grant number W2531010. AJ also gratefully acknowledges the Cluster of Excellence EXC 2044/2-390685587, Mathematics Münster: Dynamics-Geometry-Structure funded by the Deutsche Forschungsgemeinschaft (DFG, German Research Foundation). JS is partially supported by NSFC (No.\ 12471142) and the Fundamental Research Funds for the Central Universities.

\section{Notation}
\subsection{Linear Algebra:}
In the following we assume familiarity with direct sums and tensor products of vector spaces see \cite{bayer_introduction_2026}. 
In the following let $V, W, U$ be some finite-dim.\ real Banach spaces. We denote their direct sum by $V\oplus W$ and tensor product by $V \otimes W$. We denote by $\mathcal{L}(V; W)$ the space of linear maps $A: V \to W$ and by $\operatorname{Bil}(V \times W; U)$ the space of bilinear maps $A: V \times W\to U$.  For any $v\otimes w\in V\otimes W$ denote by $(v \otimes w)^{\top}\coloneqq w \otimes v$ and for any $A\in \mathcal{L}(V \otimes W; U)$ we denote by $A^{\top}$ the unique linear map $A^{\top}\in \mathcal{L}(W\otimes V; U)$ such that $A(v \otimes w)=A^{\top}(w \otimes v)$ for any $v \in V, w \in W$. For $v_{1}\otimes v_{2} \in V\otimes V$  denote $\operatorname{Sym}(v_{1}\otimes v_{2})\coloneqq \frac{1}{2}\left(v_{1}\otimes v_{2}+ (v_{1} \otimes v_{2})^{\top}\right)$. We denote by $\mathcal{S}(V \otimes V; W)\subset \mathcal{L}(V \otimes V; W)$ the subset of symmetric linear maps.\newline
Given vector spaces $V_{1}, V_{2}$ and $A_{i,j}\in \mathcal{L}(V_{i}\otimes V_{j}; W)$ for $1\leq i,j\leq 2$ we define the block-matrix operator $A \in \mathcal{L}((V_{1} \oplus V_{2})^{\otimes 2}; W)$ by 
\begin{equation*}
    A((v_{1}\oplus v_{2})^{\otimes 2})\coloneqq \left(\begin{array}{cc}
    A_{1, 1} & A_{1,2}\\
    A_{2, 1} & A_{2,2}
    \end{array}\right): ((v_{1}\oplus v_{2})^{\otimes 2})\coloneqq \sum_{1\leq i, j\leq 2} A_{i,j}(v_{i} \otimes v_{j}),
\end{equation*}
and for $A_{i}\in \mathcal{L}(V_{i}; W)$ the block-vector operator $ A\in \mathcal{L}(V_{1}\oplus V_{2}; W)$ by
\begin{equation*}
    A(v_{1} \oplus v_{2}) \coloneqq \left(\begin{array}{c}
    A_{1}\\
    A_{2}
    \end{array}
    \right)\cdot
    (v_{1} \oplus v_{2})\coloneqq A_{1} v{1} + A_{2} v_{2}.
\end{equation*}
One immediately checks that these notations are generalizations of the Fréchet- and scalar-product from classical linear algebra. It will however often be helpful to perform calculations in the more compact tensor notation even when working in $\mathbb{R}^{d}$. To this end recall that for $x \in \mathbb{R}^{d_{x}}, y \in \mathbb{R}^{d_{y}}$ the tensor-product is defined by $x \otimes y\coloneqq x y^{\top}$. 
\par 
\subsection{Probability Theory:} We consider a fixed time-horizon $T>0$ and filtered probability space $(\Omega, \mathfrak{F}, (\mathfrak{F}_{t})_{t \in [0,T]}, \mathbb{P})$ satisfying the ``usual assumptions''. 
By a localizing sequence $(\tau_{k})_{k 
\in \mathbb{N}}$ we mean a sequence of stopping times w.r.t.\ $(\mathfrak{F}_{t})$ such that $\tau_{k}< \infty $, $ \sup_{k \in \mathbb{N}} \tau_{k}=\infty $, 
and $ \tau_k \leq \tau_{ k + 1 } $ a.s.\ for any $k \in \mathbb{N}$. 
For any stochastic process $Y$ and stopping time $\tau$ we denote by 
$Y^{\tau}_{\cdot}\coloneqq Y_{\cdot \wedge \tau}$ the stopped process. The It\^{o}-integral of a suitable integrand $\beta$ w.r.t. a suitable integrator $M$ is sometimes denoted by $\beta \bullet M$.
\par 
\subsection{Malliavin Calculus:} 
Let $(W_t)_{t\in [0,T]}$ be a Brownian Motion w.r.t. $(\Omega, \mathfrak{F}, (\mathfrak{F}_{t})_{t \in [0,T]}, \mathbb{P})$. 
We use $\mathbf DF$ to denote the Malliavin derivative of a random variable $F\in \sigma\{W_s, s\in[0,T]\}$. Denote by $H:=L^2(0,T)$ the Hilbert space associated with the Brownian motion on $[0,T]$. Then, $\mathbf DF=\{\mathbf D_s F, s\in[0,T]\}\in H$ a.s.  For a fixed $p\ge1$, let $\mathcal D^{1,p}$ be the space of all random variables $F$ satisfying 
\[ \|F\|^p_{1,p}:= \mathbb E \left[|F|^p+ \|\mathbf DF\|^p_{H}\right]<\infty.\]
Similarly,  we use $\mathcal D^{1,p}(H)$ to denote the set of all $H$-valued random elements $u$ satisfying 
\[  \mathbb E \left[\|u\|_H^p+ \|\mathbf Du\|^p_{H^{\otimes 2}}\right]<\infty.\]
For $1<p,q<\infty$ satisfying $\frac1p+\frac1q=1$, we denote by $\mathrm{Dom}_p(\bd)$ the set of all elements $u\in L^p(\Omega;H)$ for which  there exists a unique $\bd(u)\in L^p(\Omega)$  such that 
\[ \E[F\bd(u)]=\E\langle \mathbf D F, u\rangle_H, \text{ for all } F\in \mathcal D^{1,q}.\]
We call $\bd$ the divergence operator and $\bd(u)$ a  Skorokhod integral.  Noting that $\mathcal D^{1,q}$ is dense in $L^q(\Omega)$, the divergence operator $\bd$ is closed, i.e., if $u_n\to u$ in $L^p(\Omega;H)$ and $\bd(u_n) \to G$ in $L^p(\Omega)$ as $n\to \infty$, then we have $u\in \mathrm{Dom}_p(\bd)$ and $\bd(u)=G$.   For $u\in \mathrm{Dom}_p(\bd)$, we also denote $\bd(u)$ by $\int_0^T u_s  \diamond dW_s$.  
\par
\subsection{Path spaces:}
Consider a fixed time-horizon $T>0$. Let
$
\Delta_T \coloneqq \{(s,t): 0 \le s < t \le T \}.
$
Let $\gamma \in (0, \infty)$. We say that $X \in \mathcal{C}^{\gamma}([0,T];V)$ if it
is $\lfloor \gamma \rfloor$ times continuously differentiable with $\lfloor
\gamma \rfloor$-th derivative H\"older continuous of exponent $\{ \gamma \} =
\gamma - \lfloor \gamma \rfloor \in (0, 1]$. (In particular, $X \in
\mathcal{C}^1$ means Lipschitz rather than continuously differentiable.) For $\gamma \in (0,1]$ we denote by $\Vert X \Vert_{\gamma; [0,T]}$ the $\gamma$-Hölder semi-norm on $[0,T]$. For $A: \Delta_{T}\to V$ and $\gamma \in (0, \infty)$ we say that $A\in \mathcal{C}^{\gamma}_{2}([0,T]; V)$ if $\sup_{(s,t) \in \Delta_{T}}\frac{|A_{s,t}|}{|t-s|^{\gamma}}< \infty$.\\
For a path $X: [0,T]\to V$ we denote by $\delta X_{s,t}\coloneqq X_{t}-X_{s}$ and for $A:\Delta_{T}\to V$, $\delta A_{s,u,t}\coloneqq A_{s,t}-A_{s,u}-A_{u,t}$.
\

\subsection{Rough Paths:}
We refer to \cite{friz_2020} for the (standard) notation of a H\"older rough path and its bracket, 
$$\mathbf{X}=(X, \mathbb{X})\in \mathscr{C}^{\alpha}([0,T]; V), \qquad 
[\mathbf{X}] := (\delta X) \otimes (\delta X) - 2 \operatorname{Sym}(\mathbb{X}).$$
Imposing $[\mathbf{X}]\equiv 0$ yields $\mathscr{C}_g^{\alpha}$, the space of weakly geometric rough paths, and by $\mathscr{C}^{0, \alpha}_{g}$ we denote the space of geometric rough paths. Also set
\[ \mathscr{C}^{\alpha ; \beta} \coloneqq \{ \mathbf{X} \in \mathscr{C}^{\alpha} :
   [\mathbf{X}] \in \mathcal{C}^{\beta} \}. \]
When $\beta \ge 1$, the bracket is (at least) Lipschitz continuous, and we write $\dot{[\mathbf{X}]}$  for its derivative. We denote by $\mathscr{C}^{0,\alpha,1}$ the (Polish) space of rough paths with continuously differentiable bracket.
The space of   
$X$-controlled (resp.\ strongly $\mathbf{X}$-controlled) rough paths is denoted by $\mathscr{D}^{2\alpha}_{X}([0,T]; W)$ (resp.\ $ \mathscr{D}^{3\alpha}_{\mathbf{X}}([0,T]; W)$), 
details left to \cref{section:Introduction to controlled fields}.



\section{Space-time controlled fields}
\addtocontents{toc}{\protect\setcounter{tocdepth}{2}}
\label{section:Introduction to controlled fields}
In the following, let $V, W, U$ denote finite-dim.\footnote{Extensions to infinite-dim.\ are possible but not of relevance in this work.}\ Banach spaces. Throughout this section fix $\alpha\in ( 0, 1]$. 
Consider a two-parameter in space and time process 
$A \colon \Delta_{T}\times V\times V \to W$. 
We define for $\mathfrak{K}\subset V$ and fixed $T>0$: 
\begin{equation}
\label{def:brackets}
\begin{aligned}
|A|_{\infty;T, \mathfrak{K}}\coloneqq \sup_{x, y \in \mathfrak{K}} \sup_{(s, t) \in \Delta_{T}} |A_{s,t}(x,y)|;& \qquad
|A|_{\infty, k;T,\mathfrak{K}}\coloneqq \sup_{x, y \in \mathfrak{K}; x \neq y} \sup_{(s, t) \in \Delta_{T}} \frac{ |A_{s,t}(x,y)|}{|x-y|^{k}} \\
|A|_{k, \infty; T,  \mathfrak{K}}\coloneqq \sup_{x, y \in \mathfrak{K}} \sup_{(s, t) \in \Delta_{T}} \frac{ |A_{s,t}(x,y)|}{| t-s|^{\alpha k}};& \qquad
    |A|_{k; T,  \mathfrak{K}}\coloneqq  \sup_{x, y \in \mathfrak{K}; x \neq y} \sup_{(s, t) \in \Delta_{T}} \frac{ |A_{s,t}(x,y)|}{| t-s; x-y|_{\mathfrak{s}}^{k}},
\end{aligned}
\end{equation}
where we consider the anisotropic scaling 
\begin{equation*}
| t; x |_{\mathfrak{s}} \coloneqq  |t|^{\alpha} \vee | x|. 
\end{equation*}
When $\mathfrak{K}=V$, we write $|A|_{\infty}=|A|_{\infty; V}$ and analogously for the other objects in \eqref{def:brackets}. We say that for two two-parameter processes with spatial arguments $A, B \colon \Delta_{T}\times V \times V\to W$ it holds
\begin{equation}\label{e:k=}
\begin{aligned}
A_{s,t}(x,y)&\stackrel{k}{=}B_{s,t}(x,y) \; \text{on} \; [0,T]\times \mathfrak{K} :\Leftrightarrow |A-B|_{ k; T, \mathfrak{K}}<\infty\\
A_{s,t}(x,y)&\stackrel{k}{=}B_{s,t}(x,y) :\Leftrightarrow A_{s,t}(x,y)\stackrel{k}{=}B_{s,t}(x,y) \; \text{on} \; [0,T]\times V.
\end{aligned}
\end{equation} 
We recall the classical notion of Lipschitz functions; see, e.g., \cite[Definition~1.21]{LyonsCaruanaLevy2007}.
\begin{definition}
\label{def:Stein_Lip_def}
Let $k\geq 0$ be an integer and let $\gamma\in (k,k+1]$.
Let $\mathfrak{K}\subset W$ be closed and $f\colon \mathfrak{K}\to U$ be a function. For each integer
$j=1,\dots,k$, let
\[
\partial^{j} f \colon \mathfrak{K} \longrightarrow \mathcal{L}\bigl(W^{\otimes j};U\bigr)
\]
take values in the space of \emph{symmetric} $j$-linear mappings from $W$ to $U$.
We write \\
$(f,\partial^{1}f,\dots,\partial^{k}f)\in \Lip^{\gamma}(\mathfrak{K}; U)$, if $f$ is bounded on $\mathfrak{K}$ and there exists a constant
$M>0$ such that, for each $j=1,\dots,k$,
\begin{equation}
\label{def:Stein_Lip1}
\sup_{x\in \mathfrak{K}} |\partial^{j} f(x)| \le M,
\end{equation}
and there exist maps
\[
R_j\colon W\times W \longrightarrow \mathcal{L}\bigl(W^{\otimes j}; U\bigr)
\]
such that, for all $x_0,x_1\in \mathfrak{K},v\in W^{\otimes j}$ and $j=0, \dots, k$ (setting $\partial^{0}f\equiv f$):
\begin{align}
\partial^{j} f(x_1)(v)
&=
\sum_{l=0}^{k-j}\frac{1}{l!}\,
\partial^{j+l} f(x_0)\bigl(v\otimes (x_1-x_0)^{\otimes l}\bigr)
\;+\;
R_j(x_0,x_1)(v), \label{eq:lipgamma-taylor}
\\
|R_j(x_0,x_1)|
&\le
M |x_{1}-x_{0}|^{\gamma-j}\label{eq:lipgamma-rem}
\end{align}
In our notation, this is equivalent to saying that on $\mathfrak{K}$ it holds
$$
\partial^{j} f(x_1)(\cdot) \
\stackrel{\gamma -j }{=} \ \ 
\sum_{l=0}^{k-j}\frac{1}{l!}\,
\partial^{j+l} f(x_0)\bigl((\cdot)\otimes (x_1-x_0)^{\otimes l}\bigr).
$$
We usually say that $f$ is $\Lip^{\gamma}(\mathfrak{K}; U)$ without mentioning explicitly
$\partial^{1}f,\dots,\partial^{k}f$. The smallest constant $M$ for which
\eqref{def:Stein_Lip1}--\eqref{eq:lipgamma-rem} holds is denoted by $[f]_{\Lip^{\gamma}; \mathfrak{K}}$ and $|f|_{\Lip^{\gamma}; \mathfrak{K}}\coloneqq \sup_{v \in \mathfrak{K}} |f(v)|+[f]_{\Lip^{\gamma}; \mathfrak{K}}$.
\end{definition}
\begin{example} 
\label{ex:multivariate_calculus}
It will be instructive to consider a multivariate calculus example, with $f = f (x, y)$, defined on $V
=\mathbb{R}^{d_{x}}\times \mathbb{R}^{d_{y}}$, of $\Lip^3$-regularity, and values in $W =\mathbb{R}$.
Then the Jacobian $\partial f (x, y) \in \mathcal{L} (V; \mathbb{R})$ which we represent as $(d_{x} +
d_{y})$-dimensional row vector,
\[ \partial f (x, y) = \left( \begin{array}{c|c}
      \partial_x f & \partial_{y} f
   \end{array} \right) \]
omitting the base point $(x, y)$ on the right-hand side. Similarly, the Hessian
$\partial^2 f (x, y) \in \mathcal{L} (V; \mathcal{L} (V;  \mathbb{R})) \cong \mathcal{L} (V \otimes V;
\mathbb{R})$ has (symmetric) block-matrix representation
\[ \partial^2 f = \left( \begin{array}{c|c}
     \partial^2_x f & \partial_x \partial_y f\\
     \hline
     \partial_y \partial_x f & \partial^2_y f
   \end{array} \right) = \left( \begin{array}{c|c}
     \partial^2_x f & (\partial_y \partial_x f)^{\top}\\
     \hline
     \partial_y \partial_x f & \partial^2_y f
   \end{array} \right) . \]
Note $\partial_y \partial_x f(x,y) \in \mathcal{L} (\mathbb{R}^{d_y} \otimes \mathbb{R}^{d_x}; 
\mathbb{R}) \cong \mathrm{Bil}(\mathbb{R}^{d_y} \times \mathbb{R}^{d_x};  
\mathbb{R})$. Then \begin{eqnarray*}
  f (x_1, y_{1}) & \overset{3}{=} & 
  f (x_{0}, y_{0}) + \partial_{x} f (x_{0}, y_{0}) \Delta x+
  \partial_y f (x_{0}, y_{0}) \Delta y \nonumber \\
  &  & + \frac12\partial_{x}^2 f(x_{0}, y_{0})
  (\Delta x)^{\otimes 2} + \partial_y \partial_x f(x_{0}, y_{0}) (\Delta y, \Delta x)  + \frac{1}{2}
  \partial_y^2 f(x_{0}, y_{0}) (\Delta y)^{\otimes 2}; \nonumber \\
  \partial_x f (x_{1}, y_{1}) & \overset{2}{=} & \partial_x f(x_{0}, y_{0}) +
  \partial_x^2 f(x_{0}, y_{0}) \Delta x + \partial_y \partial_x f
  (x_{0}, y_{0}) \Delta y; \nonumber \\
  \partial_y f (x_1, y_{1}) & \overset{2}{=} & \partial_y f(x_{0}, y_{0}) +
  \partial_x \partial_y f(x_{0}, y_{0})  \Delta x + \partial_y^2 f
  (t_0, x_0) \Delta y \nonumber \\
  & = & \partial_y f(x_{0}, y_{0}) +
  \left(\partial_{y} \partial_x f(x_{0}, y_{0}) \right)^{\top}  \Delta x + \partial_y^2 f
  (t_0, x_0) \Delta y, \nonumber
\end{eqnarray*} 
where we substituted $\Delta x= x_{t_{1}}-x_{t_{0}}, \Delta y=y_{t_{1}}- y_{t_{0}}$ for brevity.
\end{example}
\

We now discuss controlled rough paths. The following definition is classical (see \cite{GUBINELLI2004} or \cite[Def 4.6]{friz_2020}). 
\begin{definition}
    Let $X \in \mathcal{C}^{\alpha}([0,T]; V)$. Then we call a tuple $\mathcal{Y}=(Y, Y')$ of paths 
    \begin{equation*}
    \begin{aligned}
        \mathcal{Y}: [0,T]&\to W \times \mathcal{L}(V; W)\\
        \quad t &\mapsto (Y_{t}, Y'_{t}),
        \end{aligned}
    \end{equation*}
    an $X$-\textit{controlled rough path} if, on $[0,T]$,
    \begin{equation*}
    \begin{aligned}
        Y_{t_{1}}\stackrel{2}{=} Y_{t_{0}} + Y'_{t_{0}} \delta X_{t_{0}, t_{1}}; \quad Y'_{t_{1}}\stackrel{1}{=}Y'_{t_{0}}.
    \end{aligned}
    \end{equation*}
    Equivalently
    \begin{equation*}
        \left |R^{Y, Y'}\right |_{2}\coloneqq \sup_{(t_{0},t_{1}) \in \Delta_{T}} \frac{|\delta Y_{t_{0}, t_{1}} - Y'_{t_{0}} \delta X_{t_{0}, t_{1}}|}{|t_{1}-t_{0}|^{2\alpha}}< \infty; \quad [Y']_{1}\coloneqq \sup_{(t_{0}, t_{1})\in \Delta_{T}}\frac{|\delta Y'_{t_{0}, t_{1}}|}{|t_{1}-t_{0}|^{\alpha}}<\infty.
    \end{equation*}
    We denote the space of $W$-valued, $X$-controlled rough paths by $\mathscr{D}^{2\alpha}_{X}([0,T]; W)$ and further define the seminorm
    \begin{equation}\label{e:semi-norm}
        \left [Y, Y'\right]_{X; 2}\coloneqq \left | R^{Y, Y'} \right|_{2}+ [ Y' ]_{1}.
    \end{equation}
\end{definition}

Higher-order controlled paths are also well-known (see \cite{Trees} or \cite[Secticon 4.5]{friz_2020}), and  the relevant variant for our purposes is given by
the following notion. 
\begin{definition}
\label{def:HOcrp}
    Let $\mathbf{X}=(X, \mathbb{X})\in \mathscr{C}^{\alpha;1+\alpha}([0,T]; V)$. We call a $4$-tuple $\mathcal{Y}=(Y, Y', Y'', \dot{Y})$ of paths
   \begin{equation*}
   \begin{aligned}
       \mathcal{Y} : [0,T] &\to W \times
  \mathcal{L}  (V ; W) \times \mathcal{L} (V^{\otimes 2} ; W) \times W \\
  t &\mapsto  (Y_{t}, Y'_{t}, Y''_{t}, \dot{Y}_{t}),
  \end{aligned}
   \end{equation*}
  a \textit{strongly} $\mathbf{X}$-\textit{controlled rough path} if, on $[0,T]$,
  \[ \begin{array}{lll}
       Y_{t_{1}} & \stackrel{3 }{=} & Y_{t_{0}} + Y_{t_{0}}' \delta X_{t_0, t_1} + Y''_{t_{0}} 
       \mathbb{X}_{t_0, t_1} + \dot{Y}_{t_{0}} (t_1 - t_0),\\
       Y_{t_{1}}' & \stackrel{2 }{=} & Y'_{t_{0}} + Y''_{t_{0}} \delta X_{t_0, t_1},\\
       ( Y_{t_{1}}'', \dot{Y}_{t_{1}} )& \stackrel{1}{=} & ( Y_{t_{0}}'', \dot{Y}_{t_{0}} ).
     \end{array} \] 
     Or equivalently $\dot{Y}\in \mathcal{C}^{\alpha}([0,T]; W)$, $(Y', Y'')\in \mathscr{D}^{2\alpha}_{X}([0,T]; \mathcal L(V;W))$ and 
     \begin{eqnarray*}
     \left | R^{Y, Y', Y'', \dot{Y}} \right |_{\mathbf{X}; 3}&\coloneqq& \sup_{(t_{0}, t_{1}) \in \Delta_{T}} \frac{|\delta Y_{t_{0}, t_{1}}- Y_{t_{0}}' \delta X_{t_0, t_1} - Y''_{t_{0}} 
       \mathbb{X}_{t_0, t_1} - \dot{Y}_{t_{0}} (t_1 - t_0)|}{|t_{1}-t_{0}|^{3\alpha}}<\infty.
     \end{eqnarray*}
     We denote the space of such $W$-valued, strongly $\mathbf{X}$-controlled rough paths by $ \mathscr{D}^{3\alpha}_{\mathbf{X}}([0,T]; W)$. Further we define the seminorm
\begin{equation*}
    \left [ Y, Y', Y'', \dot{Y}\right]_{\mathbf{X}; 3}\coloneqq\left |R^{Y, Y', Y'', \dot{Y}} \right|_{\mathbf{X}; 3}+[Y', Y'']_{X; 2}+ [  \dot{Y}]_{1}. 
\end{equation*}
\end{definition}

\begin{remark}
\label{rem:rightpoint_expansion}
    Note that in the above we assumed that $t_{0}<t_{1}$, which corresponds to the case of left-point integration in Riemann-type integrals. However, we will encounter many examples with right-point (backward) expansions of the form
\begin{equation}
\label{rem:eq_rightpoint_expansion}
\left.
\begin{array}{lll}
       Y_{t_{0}} & \stackrel{3 }{=} & Y_{t_{1}} + Y_{t_{1}}' \delta X_{t_0, t_1} + Y''_{t_{1}} 
       \mathbb{X}_{t_0, t_1} + \dot{Y}_{t_{1}} (t_1 - t_0),\\
       Y_{t_{0}}' & \stackrel{2 }{=} & Y'_{t_{1}} + Y''_{t_{1}} \delta X_{t_0, t_1},\\
       ( Y_{t_{0}}'', \dot{Y}_{t_{0}} )& \stackrel{1}{=} & ( Y_{t_{1}}'', \dot{Y}_{t_{1}} ).
     \end{array} 
     \right\}
\end{equation}
for $t_{0}< t_{1}$. Under the additional assumption that
$(Y', -(Y'')^{\top})\in \mathscr{D}^{2\alpha}_{\mathbf{X}}$  by using Chen's relation and  noting $\mathbb{X}_{t_{1}, t_{0}}=\mathbb{X}_{t_{0}, t_{1}}^{\top}$  for (weakly) geometric $\mathbf{X}$, we can rewrite \cref{rem:eq_rightpoint_expansion} as 
\begin{equation*}
\begin{array}{lll}
       Y_{t_{1}} &\stackrel{3 }{=} &Y_{t_{0}} - Y_{t_{0}}' \delta X_{t_{0}, t_{1}} + (Y''_{t_{0}})^{\top} 
       \mathbb{X}_{t_{0}, t_{1}} - \dot{Y}_{t_{0}} (t_1 - t_0),\\
       Y_{t_{1}}' & \stackrel{2 }{=} & Y'_{t_{0}} - (Y'')^{\top}_{t_{0}} \delta X_{t_0, t_1},\\
       ( Y_{t_{1}}'', \dot{Y}_{t_{1}} )& \stackrel{1}{=} & ( Y_{t_{0}}'', \dot{Y}_{t_{0}} )
     \end{array} 
\end{equation*}
and so 
\begin{equation*}
(Y,-Y', (Y'')^{\top}, -\dot Y) \in \mathscr{D}^{3 \alpha}_{\mathbf{X}}, 
\end{equation*}
which is analogous to Remarks~5.13 and~12.4 in \cite{friz_2020}. 
\end{remark}

\

We now define an analogous notion for fields with space-time dependence.
\begin{definition}
  \label{def:C3_jet} Let $\mathfrak{K}\subset W$ be closed. We call the $7$-tuple $\mathcal{F}=(F, F', \partial F, F'',  \partial{F}', \partial^{2} F, \dot{F})$ of functions
  \begin{equation}  \label{def:dim_space_time_controlled}
  \begin{aligned}
      \mathcal{F}: [0,T] \times \mathfrak{K} &\to U \times \mathcal{L} (V ; U)
     \times \mathcal{L} (W ; U) \times \mathcal{L} (V^{\otimes 2} ; U) \times \mathcal{L} ( W\otimes V ; U) \times \mathcal{S} (W^{\otimes 2} ; U) \times U\\
     (t,x)  &\mapsto (F_{t}(x) , F'_{t}(x), \partial F_{t}(x), F''_{t}(x), \partial F'_{t}(x), \partial^2
     F_{t}(x), \dot{F}_{t}(x))  
  \end{aligned}
  \end{equation}
  a {\em strongly} $\mathbf{X}$-{\em controlled} 
  $\Lip^3$ 
  {\em field} (in the following referred to as a \textit{controlled field}), on $[0,T]\times \mathfrak{K}$, if it holds on $[0,T]\times \mathfrak{K}$,
  \begin{equation}
  \left. \label{def:C3_jets_cascade}
  \begin{array}{rll}
    F_{t_{1}}(x_{1}) & \stackrel{3}{=} & F_{t_{0}}(x_{0}) + F'_{t_{0}}(x_{0}) \delta X_{t_0, t_1} + \partial F_{t_{0}}(x_{0}) (x_1 -
    x_0)+ F''_{t_{0}}(x_{0})  \mathbb{X}_{t_0, t_1}
    \\ 
    && + \partial F'_{t_{0}}(x_{0}) \left( x_1 - x_0, \delta X_{t_{0}, t_{1}}\right) + \frac{1}{2} \partial^2 F_{t_{0}}(x_{0})(x_{1}-x_{0})^{\otimes 2} \\
    &&+
    \dot{F}_{t_{0}}(x_{0}) (t_1 - t_0),\\
    && \\
    F'_{t_{1}}(x_{1})& \stackrel{2}{=} & F'_{t_{0}}(x_{0})+ F''_{t_{0}}(x_{0})\delta X_{t_0, t_1} + \partial F'_{t_{0}}(x_{0}) (x_{1}-x_{0}), \\
    \partial F_{t_{1}}(x_{1})& \stackrel{2}{=} & \partial F_{t_{0}}(x_{0})+ (\partial F'_{t_{0}}(x_{0}))^{\top} \delta X_{t_0, t_1}
    + \partial^2 F_{t_{0}}( x_{0}) (x_1 - x_0),\\
    && \\
     G_{t_{1}}(x_{1}) &\stackrel{1}{=} &  G_{t_{0}}(x_{0}), \; \; \; \;  G \in \{ F'', \partial F', \partial^{2} F, \dot{F} \}.
  \end{array}
  \right  \}
  \end{equation}
   We denote the space of such $\mathbf{X}$-controlled fields  by $\mathscr{D}^{3 \alpha}_{\mathbf{X}}\Lip^{3}_{x}(\mathfrak{K}; U)$. If $\mathcal{F}\in \mathscr{D}^{3\alpha}\Lip^{3}_{x}(\mathfrak{K}; U)$ for any compact subset $\mathfrak{K}\subset W$, we say $\mathcal F\in \mathscr{D}^{3 \alpha}_{\mathbf{X}}\Lip^{3}_{x, \loc}(W; U)$\footnote{This of course implies that $\mathcal{F}$ needs to be well-defined on any compact subset of $W$ e.g. as the restriction of a field defined on $W$.}.
\end{definition}
\begin{remark} We insist that the ``sewing condition'' $\alpha > 1/3$ plays no role in the definition of 
strongly controlled paths or controlled fields. That said, if $\alpha > 1/3$, then $Y$ of \cref{def:HOcrp} 
(resp.\ $ F $ in \cref{def:C3_jet}) 
can be written as a rough integral, 
\[ Y_{t_1} - Y_{t_0} = \int_{t_0}^{t_1} (Y_s', Y_s'') d \mathbf{X}_s +
   \int_{t_0}^{t_1} \dot{Y_s} d s \]
and, respectively, as
\[ F_{t_1} (x) - F_{t_0} (x) = \int_{t_0}^{t_1} (F_s', F_s'') (x) d
   \mathbf{X}_s + \int_{t_0}^{t_1} \dot{F_s} (x) d s. \]
In this sense, (rough) It\^o formulas are essentially composition rules for expansions, rather than integral identities; the latter are merely consequences when suitable sewing conditions are met. 
\end{remark}
\begin{remark}
\label{rem:partial vs. Frechet}
By \cref{def:Stein_Lip_def}, we see immediately that for $\mathcal{F}\in \mathscr{D}^{3\alpha}_{\mathbf{X}}\Lip^{3}_{x, \loc}(W; U)$, it follows $(F_t, \partial F_t, \partial^{2} F_t) \in  \Lip^{3}(\mathfrak{K}; U)$ for any compact $\mathfrak{K}\subset W$ and each fixed $t\in[0,T]$. In particular, following Remark~1.23 in \cite{LyonsCaruanaLevy2007}, $\partial F$ and  $\partial^{2} F$ agree respectively with the Fréchet derivatives $DF$ and $D^{2}F$ on the interior of $\mathfrak{K}$ (which however may be empty)\footnote{Throughout this work $DF$ and $D^{2}F$ will always denote elements of $\mathcal{L}(W; U)$ resp. $\mathcal{L}(W^{\otimes 2}; U)$ instead of their (equivalent) representation in terms of vectors (resp. matrices). }. If $\mathcal{F}\in \mathscr{D}^{3\alpha}_{\mathbf{X}}\Lip^{3}_{x}$ then this agreement holds true on all of $W$. 
\end{remark}
\begin{remark}
\label{rem:commute}
Note that in \cref{def:C3_jet}, we explicitly set the Gubinelli derivative of $ \partial F $ to be $(\partial F)'\coloneqq (\partial F')^{\top}$ in \cref{def:C3_jets_cascade}, where we insist that $\partial F'= \partial (F')$ (which agrees with $D (F')$ on open sets). This choice is rooted in calculus, similar to the familiar chain rule for Gubinelli derivatives ($f(Y)'=Df (Y) Y'$, for, e.g., $f\in \Lip^2$ say); cf.\ \cref{ex:multivariate_calculus_ccontd} below. We note that the relation $(\partial F)' = (\partial F')^{\top}$  also appears in \cite{castrequini_2025}, and in fact much earlier \cite{keller_zhang_2016}, there however stated as a result, rather than a choice, in a ``truly rough'' setting \cite[Section~6.2]{friz_2020} when the Gubinelli derivative can be shown to be unique. (This assumption has been crucial in H\"ormander theory via rough paths, \cite[Section~11.3]{friz_2020} and references therein, but is not suitable to study stability and approximation results.)

\end{remark}

\begin{example} \label{ex:multivariate_calculus_ccontd}
Again consider $f\in \Lip^{3}(\mathbb{R}^{d_{x}}\oplus \mathbb{R}^{d_{y}}; \mathbb{R})$ as in \cref{ex:multivariate_calculus} and 
a controlled rough path 
$(Y, Y')\in \mathscr{D}^{2\alpha}_{X}([0,T]; \mathbb{R}^{d_{y}})$ w.r.t.\ $\mathbf{X}\in \mathscr{C}^{\alpha}([0,T]; \mathbb{R}^{d_{X}})$. Define $F_{t}(x)\coloneqq f(x, Y_{t})$. We know (\cite[Chapter 7]{friz_2020}) that for any $x\in \mathbb{R}^{d_{x}}$ we have $(F(x), F'(x))\in \mathscr{D}^{2\alpha}_{X}([0,T]; \mathbb{R})$ with $F'_{t}(x)\coloneqq \partial_{y} f(x, Y_{t}) Y'_{t}\in \mathcal{L}(\mathbb{R}^{d_{X}}; \mathbb{R})$. Then, 
$\partial_{x} F'_{t}(x)=\partial_{x} \partial_{y} f(x, Y_{t}) Y'_{t} \in \mathcal{L}(\mathbb{R}^{d_{x}}; \mathcal{L}(\mathbb{R}^{d_{X}}; \mathbb{R}))\cong\mathcal{L}(\mathbb{R}^{d_{x}} \otimes \mathbb{R}^{d_{X}} ; \mathbb{R})$ or explicitly
\begin{equation*}
    \partial_{x} F'_{t}(x)(w \otimes v)= \partial_{x} \partial_{y} f(x, Y_{t}) (w\otimes  (Y'_{t} v)), ~ \text{ for } w \otimes v \in \mathbb{R}^{d_{x}} \otimes \mathbb{R}^{d_{X}}.
\end{equation*} 
On the other hand, we also know  $(\partial_{x}F(x), (\partial_{x}F(x))')\in \mathscr{D}^{2\alpha}_{X}([0,T]; \mathcal{L}(\mathbb{R}^{d_{x}}; \mathbb{R}))$ for any $x \in \mathbb{R}^{d_{x}}$,  where $(\partial_{x}F)'_{t}(x)=\partial_{y} \partial_{x} f(x,Y_{t}) Y'_{t} \in \mathcal{L}(\mathbb{R}^{d_{X}}; \mathcal{L}(\mathbb{R}^{d_{x}}; \mathbb{R}))\cong\mathcal{L}(\mathbb{R}^{d_{X}}\otimes \mathbb{R}^{d_{x}}; \mathbb{R})$ or explicitly
\begin{equation*}
    (\partial_{x} F)'_{t}(x) (v \otimes w)= \partial_{y} \partial_{x} f(x, Y_{t}) ( (Y'_{t}v) \otimes w), ~ \text{ for } v \otimes \omega \in \mathbb{R}^{d_{X}}\otimes \mathbb{R}^{d_{x}}.
\end{equation*}
Since we know that $(\partial_{x} \partial_{y} f)^{\top}=\partial_{y} \partial_{x} f$ it follows 
\begin{equation*}
    \partial_{x} F'_{t}(x) (w \otimes v) = (\partial_{x} F)'_{t}(x) ( v \otimes w),
\end{equation*}
which means $( \partial_{x} F'_{t})^{\top}=(\partial_{x} F)'$, being the relation which we explicitly impose in \cref{def:C3_jet}.
\end{example}

\

Before giving examples of \cref{def:C3_jet}, we provide a simple criterion for verifying whether a given $7$-tuple of fields is $\mathbf{X}$-controlled.
\begin{lemma}
\label{lemma:jet_criterion}
     Let $\mathfrak{K}\subset W$ closed and $\mathcal{F}=(F, F', \partial F, F'', \partial F', \partial^{2} F, \dot{F})$ be a 7-tuple of functions on $[0,T]\times \mathfrak{K}$ of suitable dimension as in \eqref{def:dim_space_time_controlled}. 
    Let us define 
    \begin{equation*}
    \begin{aligned}
        \left [\mathcal{F} \right]_{x; \mathfrak{K}} &\coloneqq  \sup_{t\in [0,T]} \left(\left [F_{t}, \partial F_{t}, \partial^{2} F_{t} \right]_{\Lip^{3}; \mathfrak{K}}+ \left [F'_{t} , \partial F'_{t} \right ]_{\Lip^{2}; \mathfrak{K}}
        +[F''_{t
        }]_{\Lip^{1}; \mathfrak{K}}+ [\dot{F}_{t}]_{\Lip^{1};\mathfrak{K}}\right) ;\\
        \left [\mathcal{F} \right]_{t; \mathfrak{K}} &\coloneqq \sup_{x \in \mathfrak{K}} \left(\left [F(x), F'(x), F''(x), \dot{F}(x) \right]_{\mathbf{X}; 3}+ \left [\partial F(x), (\partial F')^{\top}(x) \right]_{X; 2}+[ \partial^{2} F(x)]_{1}\right).
    \end{aligned}
    \end{equation*}
    Then $\mathcal{F}\in \mathscr{D}^{3 \alpha}_{\mathbf{X}} \Lip^3_{x}(\mathfrak{K}; U)$ iff $[\mathcal{F}]_{\mathfrak{K}}\coloneqq [\mathcal{F}]_{x; \mathfrak{K}} + [\mathcal{F}]_{t; \mathfrak{K}} < \infty$. Analogously, $\mathcal{F}\in \mathscr{D}^{3 \alpha}_{\mathbf{X}} \Lip^3_{x, \loc}$ iff 
    $[\mathcal{F}]_{\mathfrak{K}}< \infty$ for any $\mathfrak{K}\subset W$ compact. In particular it holds that $[\mathcal{F}]_{\mathfrak{K}}$  bounds all expansions in  \eqref{def:C3_jets_cascade} on $[0,T]\times \mathfrak{K}$.
\end{lemma}
\begin{proof}
We only consider $\mathfrak{K}=W$ for simplicity, as the argument is mostly identical. 
For $\mathcal{F}\in \mathscr{D}^{3\alpha}\Lip^{3}_{x}$ note that we can immediately obtain $[\mathcal{F}]_{x; W}< \infty$ and $[\mathcal{F}]_{t; W}< \infty$ by setting $t_{1}=t_{0}$ 
resp.\ $x_{1}=x_{0}$ in \cref{def:C3_jets_cascade} and noting that the estimates by definition hold uniformly in both variables. Now for the other direction note that $[\mathcal{F}]$ immediately bounds the first-order estimates in \eqref{def:C3_jets_cascade} as e.g.\ for $F''$ it holds 
\begin{equation*}
\begin{aligned}
    \big| F''_{t_{1}}(x_{1})-F''_{t_{0}}(x_{0})\big| &\leq \sup_{t \in [0,T]} \big| F''_{t}(x_{1})-F''_{t}(x_{0}) \big| + \sup_{x \in W} \big| F''_{t_{1}}(x)-F''_{t_{0}}(x) \big|\\
    &\leq \left([\mathcal{F}]_{x; W} + [\mathcal{F}]_{t; W} \right)| t-s; x_{1}-x_{0}|_{\mathfrak s},
\end{aligned}
    \end{equation*}
and analogously for $\partial F' $, $ \partial^{2} F$ etc. \\The 2nd-order estimates in \cref{def:C3_jets_cascade} follow similarly by noting that 
\begin{equation*}
\begin{aligned}
    & \left. \right. &&\big |F'_{t_{1}}(x_{1})-  F'_{t_{0}}(x_{0})- F''_{t_{0}}(x_{0}) \delta X_{t_0, t_1} - \partial F'_{t_{0}}(x_{0}) (x_1 -
    x_0)\big| \nonumber\\
    &\leq&& \sup_{x \in W} \big| F'_{t_{1}}(x)- F'_{t_{0}}(x)- F''_{t_{0}}(x)  \delta X_{t_0, t_1}\big| + |(\partial F'_{t_{1}}(x_{0})-\partial F_{t_{0}}(x_{0}))(x_{1}-x_{0}) \big| \nonumber\\
    &\left. \right.&&+\sup_{t \in [0,T]} \big| F'_{t}(x_{1})-F'_{t}(x_{0})- \partial F'_{t}(x_{0}) (x_{1}-x_{0}) \big| \nonumber\\
    &\leq && \sup_{x\in W} \left(\left [ F'(x), F''(x)\right]_{X; 2}|t_{1}-t_{0}|^{2}+ [\partial F'(x)]_{1} |x_{1}-x_{0}||t-s|^{\alpha}\right) \nonumber\\
    &\left.\right.&&+\sup_{t\in [0,T]} [F'_{t}, \partial F'_{t}]_{\Lip^{2}; W} |x_{1}-x_{0}|^{2} \nonumber\\
    &\leq && \left( [\mathcal{F} ]_{x; W}+ [\mathcal{F}]_{t; W} \right) | t-s; x_{1}-x_{0} |_{\mathfrak{s}}^{2}, \nonumber
\end{aligned}
\end{equation*}
and analogously for $\partial F$. 
We are left to consider the 3rd-order expansion in \cref{def:C3_jets_cascade}. First note that by \cref{rem:partial vs. Frechet} we see that $[\mathcal{F}]_{x; W}< \infty$ implies
\begin{equation*}
   \sup_{t \in [0,T]} \Big| F_{t}(x)- F_{t}(y)- \partial F_{t}(y)(x-y)-\frac{1}{2} \partial^{2} F_{t}(y)(x-y)^{\otimes 2} \Big| \leq \sup_{t \in [0,T]} \left [ F_{t} , \partial F_{t}, \partial^{2} F_{t} \right]_{\Lip^{3}; W} |x-y|^3. 
\end{equation*}
Now we get
\begin{equation*}
\begin{array}{lll}
    & &\Big |F_{t_{1}}(x_{1})-F_{t_{0}}(x_{0}) - F'_{t_{0}}(x_{0}) \delta X_{t_0, t_1} - \partial F_{t_{0}}(x_{0}) (x_1 -
    x_0)- F''_{t_{0}}(x_{0})  \mathbb{X}_{t_0, t_1} \\
    &&- \partial F'_{t_{0}}(x_{0}) (x_1 - x_0) \delta
    X_{t_0, t_1} - \frac{1}{2} \partial^2 F_{t_{0}}(x_{0})(x_1 - x_0)^{\otimes 2} 
    -
    \dot{F}_{t_{0}}(x_{0}) (t_1 - t_0)\Big| \\
    && \\
    & \leq & \sup_{t
    \in [0,T]} [F_{t}, \partial F_{t} \partial^{2} F_{t}]_{\Lip^{3}; W} |x_1-x_0|^3 \\
    &&+ \Big|F_{t_{1}}(x_{0})+ \partial F_{t_{1}}(x_{0})(x_{1}-x_{0}) + 
    \frac{1}{2} \partial^{2} F_{t_{1}}(x_{0})(x_{1}-x_{0})^{\otimes 2}\\
    
    &&-F_{t_{0}}(x_{0}) - F'_{t_{0}}(x_{0}) \delta X_{t_0 t_1} - \partial F_{t_{0}}(x_{0}) (x_1 -
    x_0)- F''_{t_{0}}(x_{0})  \mathbb{X}_{t_0, t_1}\\
    && - \partial F'_{t_{0}}(x_{0}) (x_1 - x_0) \delta
    X_{t_0, t_1} - \frac{1}{2} \partial^2 F_{t_{0}}(x_{0})(x_1 - x_0)^{\otimes 2} 
    -
    \dot{F}_{t_{0}}(x_{0}) (t_1 - t_0)\Big|\\
    && \\
    & \leq& \sup_{t
    \in [0,T]} [F_{t}, \partial F_{t} \partial^{2} F_{t}]_{\Lip^{3}; W} |x_1-x_0|^3 \\
    &&+ \sup_{x \in W}\left [F(x), F'(x), F''(x), \dot{F}(x) \right]_{\mathbf{X}; 3} |t_{1}-t_{0}|^{3\alpha}\\
    && + \big|\partial F_{t_{1}}(x_{0})(x_{1}-x_{0}) + 
    \frac{1}{2} \partial^{2} F_{t_{1}}(x_{0})(x_{1}-x_{0})^{\otimes 2}- \partial F_{t_{0}}(x_{0}) (x_1 -
    x_0)\\
    &&- \partial F'_{t_{0}}(x_{0}) (x_1 - x_0) \delta
    X_{t_0, t_1}-  \frac{1}{2} \partial^2 F_{t_{0}}(x_{0})(x_1 - x_0)^{\otimes 2} \big|\\
    && \\
    & \leq& \sup_{t
    \in [0,T]} [F_{t}, \partial F_{t}, \partial^{2} F_{t}]_{\Lip^{3}; W} |x_1-x_0|^3 \\
    &&+ \sup_{x \in W}\left [F(x), F'(x), F''(x), \dot{F}(x) \right]_{\mathbf{X}; 3} |t_{1}-t_{0}|^{3\alpha}\\
    &&+ \sup_{x \in W}\left [ \partial F(x), (\partial F')^{\top}(x)\right]_{X; 2} |x_{1} -x_{0}| |t_{1}-t_{0}|^{2 \alpha}\\
    &&+ \frac{1}{2}\sup_{x\in W}[\partial^{2}F(x)]_{1}
    |x_{1}-x_{0}|^{2} |t_{1}-t_{0}|^{\alpha}\\
    && \\
    & \leq & [\mathcal{F}]_{W} |t_{1}-t_{0}; x_{1}-x_{0} |_\mathfrak{s}^{3},
\end{array}
\end{equation*}
yielding the claim.
\end{proof}
\cref{cor:composition_rule} below gives a composition rule for controlled fields. Let us emphasize again, that it is all about consistent expansions, the sewing condition $\alpha > 1/3$ plays no role here.
\begin{theorem}
\label{cor:composition_rule}
Let $\mathbf{X}=(X,\mathbb{X})\in \mathscr{C}^{\alpha; 1+\alpha}([0,T]; V)$ and 
$ U_{i} $, $ i=1, 2, 3 $, arbitrary (finite-dim.) Banach spaces. 
For $\mathcal{F}_{i}=(F_{i}, F'_{i}, \partial F_{i}, F''_{i}, \partial F'_{i}, \partial^{2} F_{i}, \dot{F}_{i}) \in \mathscr{D}^{3 \alpha}_{\mathbf{X}}\Lip^{3}_{x, \loc}(U_{i},U_{i+1})$ we define the composition
\begin{equation*}
\begin{aligned}
\mathcal{F}_{2} \circ \mathcal{F}_{1}\coloneqq &\left(F_{2} \circ F_{1}, (F_{2} \circ F_{1})', \partial (F_{2} \circ F_{1}),(F_{2} \circ F_{1})'' , \partial (F_{2} \circ F_{1})', \partial^{2}(F_{2} \circ F_{1}), (F_{2} \circ F_{1})^{\bullet} \right)
\end{aligned}
\end{equation*}
 with components
 \begin{equation*}
\begin{array}{rll}
(F_2\circ F_1)(t,x)&\coloneqq &F_2(t, F_1(t,x)),\\
(F_{2} \circ F_{1} )'&\coloneqq& F'_{2} \circ F_{1}+\left( \partial F_{2} \circ F_{1} \right) F'_{1},\\
    \partial(F_{2} \circ F_{1}) &\coloneqq & \left(\partial F_{2} \circ F_{1} \right) \partial F_{1},\\
    (F_{2} \circ F_{1})''&\coloneqq& F''_{2} \circ F_{1} + \left(\partial F'_{2} \circ F_{1} \right) F'_{1}+ \left( (\partial F_{2}' \circ F_{1} ) F'_{1}\right)^{\top}+ \left( \partial^{2}F_{2}\circ F_{1} \right) (F'_{1}, F'_{1})\\
    &&+ \left( \partial F_{2} \circ F_{1} \right) F''_{1},\\
\partial (F_{2} \circ F_{1})'&\coloneqq& \left( \partial F'_{2} \circ F_{1}\right) \partial F_{1} + \left( \partial^{2} F_{2} \circ F_{1} \right) (\partial F_{1}, F'_{1})+\left( \partial F_{2} \circ F_{1} \right) \partial F'_{1},  \\
\partial^{2} (F_{2} \circ F_{1} )&\coloneqq& \left(\partial^{2} F_{2} \circ F_{1} \right) (\partial F_{1}, \partial F_{1})+ \left(\partial  F_{2} \circ F_{1} \right) \partial ^{2}F_{1},\\     
(F_{2} \circ F_{1})^{\bullet} &\coloneqq& \dot{F}_{2} \circ F_{1} + \left(\partial{F}_{2} \circ F_{1}\right) \dot{F}_{1} + \Big(\frac{1}{2} (\partial ^{2} F_{2} \circ F_{1}) (F'_{1}, F'_{1}) + (\partial F_{2}'\circ F_{1}) F'_{1} \Big)\dot{[\mathbf{X}]}_{s}.
        \end{array}
    \end{equation*}
Then $\mathcal{F}_{2} \circ  \mathcal{F}_{1} \in \mathscr{D}^{3\alpha}_{\mathbf{X}}\Lip^{3}_{x, \loc}(U_{1}; U_{3})$. Additionally, if $\mathcal{F}_{i} \in \mathscr{D}^{3 \alpha}_{\mathbf{X}}\Lip^{3}_{x}(U_{i},U_{i+1})$ for $i=1,2$, then $\mathcal{F}_{2} \circ \mathcal{F}_{1} \in \mathscr{D}^{3 \alpha}_{\mathbf{X}} \Lip^{3}_{x}(U_{1}; U_{3})$. In particular, the spaces  $
    \mathscr{D}^{3 \alpha}_\mathbf{X} \Lip^{3}_{x, \loc} (U,U)$ and  $\mathscr{D}^{3 \alpha}_\mathbf{X} \Lip^{3}_{x} (U,U)$ form an associative algebra under the composition product.
\end{theorem}
\begin{proof}
Again we only check the case $\mathfrak{K}=U_{1}$. 
By \cref{lemma:jet_criterion} we know that we only need to check that 
    \begin{equation*}
        [\mathcal{F}
    _{2} \circ \mathcal{F}_{1} ]_{x; U_{1}}+ [\mathcal{F}
    _{2} \circ \mathcal{F}_{1} ]_{t;U_{1}} < \infty.
    \end{equation*}
    Starting off note that 
    \begin{equation*}
        \left[\mathcal{F}_{2} \circ \mathcal{F}_{1} \right]_{x; U_{1}}< \infty
    \end{equation*}
    by the chain-rule from classical calculus. Regarding time-regularity note that for any bounded functions $f: [0,T] \times U_{1} \to U_{2}$ and $g: [0,T]\times U_{2} \to U_{3}$ it holds
    \begin{equation*}
    \begin{aligned}
        \sup_{x \in U_{1}}|g_{t}(f_{t}(x))- g_{s}(f_{s}(x))| &\leq \sup_{x \in U_{1}}|(g_{t}-g_{s})(f_{t}(x))| +| g_{s}(f_{t}(x))- g_{s}(f_{s}(x))|\\
        &\leq \sup_{y \in U_{2}} |g_{t}(y) -g_{s}(y)|+ [g_{s}]_{\Lip^{1}; U_{2}} \sup_{x \in U_{1}} |f_{t}(x)-f_{s}(x)|, 
    \end{aligned}
    \end{equation*}
    which yields immediately that 
    \begin{equation*}
        \sup_{x \in U_{1}} \left([ \partial^{2}( F_{2} \circ F_{1})(x)]_{1} + [(F_{2} \circ F_{1})''(x)]_{1}+ [\partial  (F_{2} \circ F_{1})'(x)]_{1}+ [ (F_{2} \circ F_{1} )^{\bullet}(x)]_{1}\right)< \infty. 
    \end{equation*}
    Further by substituting $g= \partial F_{2}, f=F_{1}$ we see
    \begin{equation*}
    \begin{array}{rll}
    g_{t}(f_{t}) &\stackrel{2}{=}& g_{s}(f_{t}) + g'_{s}(f_{t}) \delta X_{s,t}\\
    &\stackrel{2}{=}& g_{s}(f_{s}) + \partial g_{s}(f_{s}) \delta f_{s,t}+ g'_{s}(f_{s}) \delta X_{s,t}\\
    &\stackrel{2}{=}& g_{s}(f_{s}) + (\partial g_{s}(f_{s}) f'_{s} + g'_{s}(f_{s})) \delta X_{s,t} 
    \end{array}
    \end{equation*} 
    uniformly over $x\in U_{1}$, 
    which implies by additionally using the Leibniz-rule for Gubinelli derivatives: 
    \begin{equation*}
    \sup_{x \in U_{1}} \left [ \partial (F_{2} \circ F_{1})(x), \left( \partial (F_{2} \circ F_{1})'\right)^{\top}(x)\right ]_{X; 2}< \infty
    \end{equation*}
    At last the 3rd order expansion follows by substituting $g=F_{2}, f=F_{1}$ for simplicity and calculating
    \begin{equation*}
    \label{cor:composition_estimate}
    \begin{array}{rll}
        g_{t}(f_{t})&\stackrel{3}{=}& g_{s} (f_{t}) + g'_{s}(f_{t}) \delta X_{s,t} + g''_{s}(f_{t}) \mathbb{X}_{s,t}+\dot{g}_{s}(f_{t}) (t-s)\\
        && \\
        &\stackrel{3}{=}& g_{s}(f_{s}) + \partial g_{s}(f_{s}) \delta f_{s,t} + \frac{1}{2} \partial^{2} g_{s}(f_{s}) (\delta f_{s,t})^{\otimes 2}+ g'_{s}(f_{s}) \delta X_{s,t}\\
        &&+\partial g'_{s}(f_{s}) \delta f_{s,t} \delta X_{s,t} +g''_{s}(f_{s}) \mathbb{X}_{s,t}+ \dot{g}_{s}(f_{s}) (t-s)\\
        && \\
        & \stackrel{3}{=}&g_{s}(f_{s}) + \partial g_{s}(f_{s}) f'_{s} \delta X_{s,t}+ \partial g_{s}(f_{s}) f''_{s}\mathbb{X}_{s,t} + \frac{1}{2} \partial^{2} g_{s}(f_{s}) (f'_{s}, f'_{s}) (\delta X_{s,t})^{\otimes 2}\\
        &&+ g'_{s}(f_{s}) \delta X_{s,t}+\partial g'_{s}(f_{s}) f'_{s} (\delta X_{s,t})^{\otimes 2}+ g''_{s}(f_{s}) \mathbb{X}_{s,t}+ \dot{g}_{s}(f_{s}) (t-s)+ \partial g_{s}(f_{s}) \dot{f}_{s} (t-s)\\
        && \\ 
        &\stackrel{3}{=}&g_{s}(f_{s}) + \partial g_{s}(f_{s}) f'_{s} \delta X_{s,t}+ \partial g_{s}(f_{s}) f''_{s}\mathbb{X}_{s,t} + \partial^{2} g_{s}(f_{s}) (f'_{s}, f'_{s})\mathbb{X}_{s,t}\\
        &&+ g'_{s}(f_{s}) \delta X_{s,t}+\partial g'_{s}(f_{s}) f'_{s} \mathbb{X}_{s,t}+(\partial g'_{s}(f_{s}) f'_{s})^{\top}\mathbb{X}_{s,t}+ g''_{s}(f_{s}) \mathbb{X}_{s,t}\\&&+ \dot{g}_{s}(f_{s}) (t-s)
        + \partial g_{s}(f_{s}) \dot{f}_{s} (t-s)+(\frac{1}{2} \partial^{2} g_{s}(f_{s}) (f'_{s}, f'_{s}) + \partial g'_{s}(f_{s}) f'_{s} )\dot{[\mathbf{X}]}_{s} (t-s),
    \end{array}
    \end{equation*}
    and thus $[\mathcal{F}_{2} \circ \mathcal{F}_{1}]_{t; U_{1}} < \infty$, yielding the claim by \cref{lemma:jet_criterion}. 
\end{proof}
Since every strongly controlled rough path is trivially in $\mathscr{D}^{3 \alpha}_{
\mathbf{X}} \Lip^{3}_{x}$, just with no $x$-dependence, \cref{cor:composition_rule} immediately implies a composition rule for evaluating controlled fields at strongly controlled rough paths. Upon considering \cref{ex:RDE_flows}, this will turn out to be the rough Itô-Wentzell (rIW) formula shown in \cite{keller_zhang_2016} and \cite{castrequini_2025}.
\begin{corollary}
\label{cor:RIW_formula} Let $\mathbf{X}=(X,\mathbb{X})\in \mathscr{C}^{\alpha; 1+\alpha}([0,T]; V)$. Let $\mathcal{Y}=(Y, Y', Y'', \dot{Y})\in \mathscr{D}^{3\alpha}_{\mathbf{X}}([0,T];W)$  and $\mathcal{F}=(F, F', F'', \partial F, \partial{F}', \partial^{2} F, \dot{F})
\in \mathscr{D}^{3\alpha}_{\mathbf{X}}\Lip^{3}_{x, \loc}(W; U)$. We then define:
  \[ \begin{array}{lll}
       Z_t & := & F_{t}(Y_t) \\
       Z'_t & := & F'_t(Y_{t}) + \partial F_t(Y_{t}) {Y'_t} \\
       Z''_t & := &  \partial F_t(Y_{t})   {Y_t''}  + F''_{t}(Y_{t})  +
       \partial F'_t(Y_{t}) Y_t'  + (\partial F'_t(Y_{t}) Y_t')^T + \partial^2 F_t(Y_{t}) (Y_t',
       Y_t') \\
       \dot{Z}_t & := & \partial F_t(Y_{t}) \dot{Y_t} + \dot{F_t}(Y_{t}) + \left( \partial
       F_t'(Y_{t})  Y_t' + \tfrac{1}{2} \partial^2 F_t(Y_{t}) (Y_t', Y_t^{'})  \right)
       \dot{[\mathbf{X}]}_t
     \end{array} \]
  Then $\mathcal{Z}= (Z, Z', Z'', \dot{Z})\in \mathscr{D}^{3 \alpha}_{\mathbf{X}}([0,T]; U)$. Here the operator $\partial F'_{t}(Y_{t})Y'_{t} \in \mathcal{L}(W \otimes V; U)$ means  $(\partial F'_{t}(Y_{t})Y'_{t})(w\otimes v)\coloneqq \partial F'_{t}(Y_{t})(w\otimes (Y'_{t}v))$ for any $w\otimes v \in W \otimes V$. 
     \end{corollary}
     \begin{proof}
      The claim follows as a direct application of \cref{cor:composition_rule} by noting that
     \begin{equation*}
         \mathcal{Y} \coloneqq (Y, 0, Y', Y'', 0, 0,\dot{Y})\in \mathscr{D}^{3 \alpha}_{\mathbf{X}} \operatorname{Lip}^{3}_{x}.
     \end{equation*}
     \end{proof}
\begin{remark}
The conditions $\dot{[\mathbf{X}]}, \dot{Y}\in \mathcal C^\alpha$ imposed in \cref{cor:composition_rule} and \cref{cor:RIW_formula} are only to ensure that $(F_{2} \circ F_{1})^{\bullet}$ is $\alpha$-H\"older continuous in time, as required in \cref{def:C3_jet}. Nevertheless, the composition rules stated in those corollaries remain valid when $\dot{Y}\in L^{1}([0,T]; W)$,  $\mathbf{X} \in \mathscr{C}^{\alpha}([0,T]; V)$ and the mapping $[0,T] \ni t \mapsto [\mathbf{X}]_t$ is continuous and of finite variation, with $\dot{Y}dt$ replaced by $dV_{t}$ for $V= \int_{0}^{\cdot} \dot{Y}_{s} ds$ and $\dot{[\mathbf{X}]}dt$ replaced by $d[\mathbf{X}]_t$. In the context of \cref{cor:RIW_formula}, one can then show that in case  $\alpha > \frac{1}{3}$ it holds 
\begin{equation}
\begin{split}
\label{rem:eq:weak_RIW}
    \delta Z_{s,t}=& \int_{s}^{t} (Z', Z'')_{r} d\mathbf{X}_{r} + \int_{s}^{t}  \dot{F_r}(Y_{r}) dr+ \int_{s}^{t} \partial F_r(Y_{r}) dV_{r} \\
    &+ \int_{s}^{t}  \partial
       F_r'(Y_{r})  Y_r' + \tfrac{1}{2} \partial^2 F_r(Y_{r}) (Y_r', Y_r^{'}) d[\mathbf{X}]_{r}.
\end{split}
\end{equation}
Here the latter two integrals are taken in Riemann-Stieltjes sense. Note that, however, now we can not obtain a local description (in terms of local cascades of type \eqref{def:C3_jets_cascade}) of $Z$. In the proof of \cref{theorem:RSIW_Theorem_1} we will encounter a situation where this global version will play a role. 
\end{remark}
\begin{remark}
\label{remark:p_variation}
    Let us also remark that a version of \eqref{rem:eq:weak_RIW} remains valid under considering  a strongly $\mathbf{X}$-controlled $ p $-variation c\`{a}dl\`{a}g rough path $(Y, Y', Y'', \dot{Y})$ for $p \in [2, 3)$ as introduced in \cite{FRIZ20186226}. The $3$rd order estimate in \cref{cor:RIW_formula} (see proof of \cref{cor:composition_rule}) can then be seen to be dominated by a control $w(s, t-)^{ 3 / p }$, when suitably correcting for jumps. 
\end{remark}
\

\begin{example}
\label{ex:RDE_flows}
Let $\mathbf{Z}=(Z, \mathbb{Z})\in \mathscr{C}^{\alpha}_{g}([0,T]; \mathbb{R}^{m})$ with $\alpha \in (\nicefrac{1}{3}, \nicefrac{1}{2}]$ and consider $(X^{r,x; \mathbf{Z}}_{s})_{s\in [r, T]}$ the solution to the RDE given by 
    \begin{equation}
    \label{eq:RDE_sol}
        dX^{r,x; \mathbf{Z}}_{s}= \mu(X^{r,x; \mathbf{Z}})ds + \sigma(X^{r,x; \mathbf{Z}}_{s})d\mathbf{Z}_{s}; \; \; \; X^{r,x; \mathbf{Z}}_{r}=x,
    \end{equation}
    with $\mu \in \Lip^{3}(\mathbb{R}^{d}; \mathbb{R}^{d})$ and $\sigma \in \Lip^{5}(\mathbb{R}^{d}; 
    \mathbb{R}^{d\times m})$, where $\Gamma(\cdot)\coloneqq D(\cdot) \sigma$. Then by classical estimates for RDEs (see \cite[Chapters 10--12]{Friz_Victoir_2010}) and using \cref{lemma:jet_criterion}, we see that   the forward flow $[r, T] \ni t \mapsto \phi(r, t; x)\coloneqq X^{r,x; \mathbf{Z}}_{t}$ induces a strongly $\mathbf{Z}$-controlled field as 
    \begin{equation}
    \label{ex:forward_flow_jet}
    \Phi\coloneqq (\phi, \sigma \circ \phi , D \phi, (\Gamma \sigma) \circ \phi , D(\sigma \circ \phi), D^{2} \phi, \mu \circ \phi) \in \mathscr{D}^{3 \alpha}_{\mathbf{Z}} \operatorname{Lip}^{3}_{x}(\mathbb{R}^d; \mathbb{R}^{d}). 
    \end{equation}
    Analogously the backward RDE flow $\overleftarrow{\phi}(t;x)\coloneqq \phi(t, T; x)$ satisfies 
    \begin{eqnarray*}
        \overleftarrow{\phi}(s;x)= \overleftarrow{\phi}(t; \phi(s, t; x))&\stackrel{3}{=}& \overleftarrow{\phi}(t;x) + D \overleftarrow{\phi}(t;x) ( \delta \phi (s, \cdot; x)_{s,t} )\nonumber \\
        &&+\frac{1}{2} D^{2} \overleftarrow{\phi}(t; x)(\delta \phi (s, \cdot; x)_{s,t} )^{\otimes 2}\nonumber \\
        &\stackrel{3}{=} &\overleftarrow{\phi}(t;x) + D \overleftarrow{\phi}(t;x)  \sigma(x) \delta Z_{s,t} + D\overleftarrow{\phi}(t;x)(\Gamma \sigma(x)) \mathbb{Z}_{s,t}\nonumber  \\
        && + D^{2} \overleftarrow{\phi}(t; x)(\sigma(x), \sigma(x)) \mathbb{Z}_{s,t}+D \overleftarrow{\phi}(t;x)  \mu(x) (t-s)\nonumber \\
        &=& \overleftarrow{\phi}(t;x) + \Gamma \overleftarrow{\phi}(t;x) \delta Z_{s,t} + \Gamma^{2} \overleftarrow{\phi}(t; x) \mathbb{Z}_{s,t}+D \overleftarrow{\phi}(t;x) \mu(x) (t-s)\nonumber.
    \end{eqnarray*}
    Now by \cref{rem:rightpoint_expansion} and upon noting that $(\Gamma \overleftarrow{\phi}(\cdot, x), -(\Gamma^{2} \overleftarrow{\phi})^{\top}(\cdot, x))\in \mathscr{D}^{2 \alpha}_{\mathbf{Z}}$, we verify that $(\overleftarrow{\phi}(\cdot, x), -\Gamma \overleftarrow{\phi}(\cdot, x), (\Gamma^{2} \overleftarrow{\phi})^{\top}(\cdot, x))\in \mathscr{D}^{3\alpha}_{\mathbf{Z}}$ uniformly in $x \in \mathbb{R}^{d}$. Now the remaining conditions on the spatial and temporal residuals from \cref{lemma:jet_criterion} are checked analogously and so 
    \begin{equation*}
        \left(\overleftarrow{\phi},  - \Gamma \overleftarrow{\phi}, D\overleftarrow{\phi}, (\Gamma^{2} 
        \overleftarrow{\phi})^{\top}, -D(\Gamma \overleftarrow{\phi}), D^{2} \overleftarrow{\phi},  -(D \overleftarrow{\phi})  \mu\right)\in \mathscr{D}^{3 \alpha}_{\mathbf{Z}} \operatorname{Lip}^{3}_{x}(\mathbb{R}^{d}; \mathbb{R}^{d}).
    \end{equation*}
\end{example}
\begin{remark}
\label{rem:no_explosion}

Note that, in order for $\mathcal{F}\in \mathscr{D}^{3\alpha}_{\mathbf{Z}}\Lip^{3}_x$ to hold, the conditions imposed on $\sigma$ and $\mu$ in \cref{ex:RDE_flows} can be relaxed, provided that no explosion occurs. Sufficient conditions ensuring non-explosion can be found in \cite{Riedel2017}; see also \cite{XM25}.
\end{remark}
\

Having established the regularity of the forward and backward RDE-flows, the composition rule in \cref{cor:composition_rule} allows us to easily derive results in the theory of RPDEs, one of the most immediate ones being for rough transport equations.  
\begin{example}
\label{example:RDE_flow_transport}
    In \cite[Section 12.1.1]{friz_2020},  the authors consider rough transport equations of the form 
\begin{equation}\label{ex:rough_transport}
    \begin{cases}
       -d u_{t}(x)&= \displaystyle Du_t(x) \mu(x) dt +   \Gamma u_{t}(x)d\mathbf Z_t\; \; \; \; \text{on} \; \; [0,T]\times\mathbb{R}^{d},\\
        u_{T}(x)&= \displaystyle g(x) \; \; \; \; \text{on} \; \; \mathbb{R}^{d}.
    \end{cases}
    \end{equation}    
    Making the same assumption on $\alpha, Z, \mu$ and $\sigma$ as in \cref{ex:RDE_flows}, they identify for $g \in \Lip^{3}(\mathbb{R}^{d}; \mathbb{R})$,  $u_{t}(x)\coloneqq(g \circ \overleftarrow{\phi})_{t}(x)$ as the unique solution out of a suitably regular class of solutions. By \cref{ex:RDE_flows},  this motivates the following definition of ``regular" solutions to \eqref{ex:rough_transport}.
    \begin{definition}
\label{ex:regular_solutions_transport}
  We say a jet $\mathcal{U}=(u, u', \partial u, u'', \partial u', \partial^{2} u, \dot{u})\in \mathscr{D}^{3 \alpha}_{\mathbf{X}}\Lip^{3}_{x}(\mathbb{R}^{d}; \mathbb{R})$ is a \textit{regular solution} to \eqref{ex:rough_transport}, if $u_T(x) = g(x)$ and it admits the expansion  
    \begin{equation}\label{def:space_time_reg_solution}
        \mathcal{U}=(u, -\Gamma u,  Du, (\Gamma^{2} u)^{\top}, -D(\Gamma u), D^{2} u, - (Du) \mu ).
     \end{equation}
    \end{definition}

 One immediately verifies that this is a special case of the regular solutions considered in \cite{friz_2020} by taking constant dependence in the space-variables in \cref{def:C3_jet}. The following theorem shows however that the  unique regular solution considered in \cite{friz_2020} is always the unique regular solution in the sense of \cref{ex:regular_solutions_transport}.

 \begin{theorem}    \label{theorem:rough_transport_existence_uniqueness}
 Suppose $g \in \Lip^{3}(\mathbb{R}^{d}; \mathbb{R})$ and $\mu \in \Lip^{3}(\mathbb{R}^{d}; \mathbb{R}^{d})$, $\sigma \in \Lip^{5}(\mathbb{R}^{d}; \mathbb{R}^{d\times m})$ and $X^{t, x,\mathbf{Z}}$ denoting the solution to \eqref{eq:RDE_sol}. Define $u_{t}(x)\coloneqq g(X^{t,x, \mathbf{Z}}_{T})$ and accordingly $\mathcal{U}$ as in \eqref{def:space_time_reg_solution}. Then, $\mathcal{U}$ is the unique regular solution of \eqref{ex:rough_transport}. 
\end{theorem}

\begin{proof}
 \cref{ex:RDE_flows} and \cref{cor:composition_rule} yield that $\mathcal{U}\in \mathscr{D}^{3 \alpha}_{\mathbf{Z}}\Lip^{3}_{x}$, and hence $\mathcal U$ solves \eqref{ex:rough_transport}. To see the uniqueness, consider any space-time regular solution  $\mathcal{U}$ of the form \eqref{def:space_time_reg_solution}, whose first component is given by a function $u$. Furthermore, for fixed $s\in [t,T]$ the forward RDE flow $[s, T]\times \mathbb{R}^{d} \ni (t, x) \mapsto \phi(s, t; x)$ in \cref{ex:RDE_flows} gives rise to a jet $\Phi \in \mathscr{D}^{3 \alpha}_{\mathbf{Z}} \Lip^{3}_{x}$ given by \eqref{ex:forward_flow_jet}. Now by \cref{cor:composition_rule}, we see  that 
\begin{eqnarray*}
 (u \circ \phi)'&=& - \Gamma u \circ \phi + (D u \circ \phi) \phi' \equiv 0; \nonumber\\
(u \circ \phi)'' &=& (\Gamma^{2}
u)^{\top} \circ \phi -(D \Gamma u \circ \phi)\sigma \circ \phi- ((D \Gamma u \circ \phi) \sigma \circ \phi)^{\top} + (D^{2} u \circ \phi)( \sigma \circ \phi , \sigma \circ \phi) \nonumber \\
 && + (Du \circ \phi) \Gamma \sigma \circ \phi \equiv 0;\nonumber\\
(u\circ \phi)^{\bullet} &=&- (Du \mu) \circ \phi + (Du \circ \phi) \mu \circ \phi\equiv0. \nonumber
\end{eqnarray*}

Since $\alpha \in (\nicefrac{1}{3}, \nicefrac{1}{2}]$, $u \circ \phi$ is thus constant in time. Therefore,
\begin{equation*}
u_{s}(x)=u\left (s, \phi(s,s; x)\right)= u\left(T, \phi(s, T; x)\right)=g\left( X^{s,x; \mathbf{Z}}_{T}\right).
\end{equation*}
  Since by definition all other components of $\mathcal{U}$ are defined in terms of $u$, the solution is unique.
\end{proof}
\end{example}
Further we can directly derive a rough version of the Alekseev-Gröbner (AG) formula \cite{alekseev_1961,grobner_1960}.
\begin{lemma}
\label{lemma:rAG}
    Let $\mathbf{Z}\in \mathscr{C}^{\alpha}_{g}([0, T]; \mathbb{R}^{d_{Z}})$ with $\alpha \in (\nicefrac{1}{3}, \nicefrac{1}{2}]$ and $(Y, Y', Y'', \dot{Y})\in \mathscr{D}^{3\alpha}_{\mathbf{Z}}([0, T]; \mathbb{R}^{d_{X}})$. Further let $\mu\in \Lip^{1}(\mathbb{R}^{d_{X}}; \mathbb{R}^{d_{X}}), \sigma \in \Lip^{5}(\mathbb{R}^{d_{X}}; \mathbb{R}^{d_{X} \times d_{Z}})$ and $X^{t,x; \mathbf{Z}}$ denote the solution to the RDE \eqref{eq:RDE_sol}. Then for any $f \in C^{3}(\mathbb{R}^{d_{X}}; \mathbb{R}^{d_{X}})$ it holds with $F_{t}^{\mathbf{Z}}(x)\coloneqq f(X_{T}^{t, x; \mathbf{Z}})$, that
    \begin{eqnarray*}
  F^{\mathbf{Z}}_t (Y_t) - F^{\mathbf{Z}}_s (Y_s) & = & \int_s^t D
  F^{\mathbf{Z}}_r (Y_r)  (\dot{Y_r} - \mu (Y_r))  \hspace{0.17em} dr 
  + \int_s^t D F^{\mathbf{Z}}_r (Y_r)  \left( \hspace{0.17em} Y'_r -
  \sigma \left( Y_r \right) \right) \circ d \mathbf{Z}_r, 
\end{eqnarray*}
where $\circ$ indicates the integration is with respect to a geometric rough path. 
\end{lemma}
\begin{proof}
     This is a direct consequence of \cref{cor:RIW_formula} and \cref{ex:RDE_flows}. 
\end{proof}
\section{Rough stochastic calculus}
\label{section:On Rough Stochastic Calculus}
\subsection{Elements of rough semimartingales }
\label{subsection:Rough_Semimartingales}
For the reader's convenience, we briefly adapt the theory of rough
semimartingales introduced in \cite{fzk23} to our Hölder setting.

We assume an underlying
filtered probability space $(\Omega, \mathfrak{F}, (\mathfrak{F}_t)_{t \in [0,T]},
\mathbb{P})$ satisfying the ``usual assumptions'' and finite-dimensional Banach spaces $V, W, U$\footnote{The following proofs rely crucially on the Burkholder-Davis-Gundy inequality, which does not hold on general infinite-dim. Banach spaces. On Hilbert-spaces however it holds.}. Let $\alpha \in (\frac{1}{3}, \frac{1}{2})$ be fixed.  We say that a $V$-valued process $A: \Delta_{T} \times \Omega \to V$ is adapted,  if $A_{s,t} \in \mathfrak{F}_{t}$ for any $(s,t) \in \Delta_{T}$. For two processes $A, B: \Delta_{T} \times \Omega \to V$ and $\beta \in (0, \infty)$,  we say $A_{s,t}\stackrel{\beta}{=} B_{s,t}$ a.s. if
\begin{equation*}
    \sup_{(s,t) \in \Delta_{T}} \frac{|A_{s,t}- B_{s,t}|}{|t-s|^{\beta}}< \infty, \quad \mathbb{P}\text{-almost surely}.
\end{equation*}
Let us emphasize that here ``$\stackrel{\beta}{=}$'' is not related to the notation ``$\stackrel{k}{=}$'' in \cref{section:Introduction to controlled fields}. For a process $Y: [0,T]\times \Omega \to V$,  we write $Y_{t} \stackrel{\beta}{=} Y_{s} + B_{s,t}$, if $\delta Y_{s,t} \stackrel{\beta}{=}B_{s,t}$ almost surely. For a controlled rough path $(Y, Y') \in \mathscr{D}^{2 \alpha}_{X}([0,T]; W)$ and any $t \in [0,T]$, we define $$X^{t}_{\cdot}:=X_{\cdot \wedge t}, \quad  (Y_{\cdot}, Y_{\cdot}')^{t}:=(Y_{\cdot \wedge t}, Y'_{\cdot \wedge t}).$$
Then, $(Y, Y')^{t} \in \mathscr{D}^{2\alpha}_{X^{t}}([0,T]; W)$.

\begin{definition}
  We denote by $\mathcal{M}^{c,\loc}(W)$ the space of $W$-valued, almost surely continuous local martingales. Furthermore, we write $\mathcal{M}^{c, \loc,1}(W) \subset \mathcal{M}^{c, \loc}(W)$ for the subclass consisting of those martingales $M$ whose quadratic variation (bracket) process satisfies $\langle M \rangle \in \Lip^{1}([0,T]; W^{\otimes 2})$ almost surely. We denote its weak derivative by $\dot{\langle M \rangle}$ which belongs to $L^{\infty}([0,T]; W^{\otimes 2})$ almost surely.
\end{definition}

We will focus on the class $\mathcal{M}^{c, \loc,1}(W)$ of martingales, as we can easily obtain Hölder-regularity of sample paths by standard localization arguments, which we recall here for the convenience of the reader. 
\begin{lemma}
\label{lemma:path_regularity_Martingale}
    Let $M \in \mathcal{M}^{c, \loc, 1}(W)$. Then $M \in \mathcal{C}^{\frac{1}{2}-}([0,T]; W)$ almost surely. 
\end{lemma}
\begin{proof}
    For any $\beta \in (0, 1)$ we define 
    \begin{equation*}
    \tau_k \coloneqq \inf \{ t\geq 0 | \Vert\langle M \rangle \Vert_{\beta; [0, t]}\geq k \}.
    \end{equation*}
By Theorem~5.31 in \cite{Friz_Victoir_2010} this is a sequence of stopping times and naturally even a localizing sequence. Now by the Burkholder–Davis–Gundy inequality, it holds for any $q \in [1, \infty)$, 
\begin{equation*}
    \left \Vert \delta M_{s,t}^{\tau_{k}} \right \Vert_{L^{q}}\lesssim \sqrt{ \left \Vert \delta \langle M \rangle_{s,t}^{\tau_{k}} \right \Vert_{L^{\frac{q}{2}}}}\leq \sqrt{k |t-s|^{\beta}}.
\end{equation*}
Thus, applying the Kolmogorov continuity criterion implies $M^{\tau_{k}}\in \mathcal{C}^{\beta-}$ a.s. for any $k \in \mathbb{N}$. Since $M=M^{\tau_{k}}$ up to indistinguishability on the event $\{T \leq \tau_{k} \}$ and $\beta \in (0,1)$ is arbitrary, the claim follows. 
\end{proof}
Analogously to  Definition~1.5  in \cite{fzk23},  which is formulated in the $p$-variation setting, we introduce the notion of rough semimartingales in the $\alpha$-Hölder scale.
\begin{definition}
  For $X \in \mathcal{C}^{\alpha}([0,T]; V)$, we call a $3$-tuple $(Y, \partial_{X}Y; M)$  an $X$-controlled, $W$-valued, $\alpha$-Hölder \textit{rough semimartingale}, if $M
  \in \mathcal{M}^{c,\loc, 1}(W)$ and 
  \begin{equation*}
  (Y, \partial_{X} Y) :  [0, T] \times \Omega  \rightarrow W \times \mathcal{L}(V; W)
  \end{equation*}
  is a continuous adapted process such that $(Y - M, \partial_{X} Y) \in \mathscr{D}^{2 \alpha}_X([0,T]; W)$ almost surely. 
\end{definition}

We now recall and extend the integration theory of rough semimartingales established in \cite{fzk23}. 
\begin{lemma}
    \label{lemma:ibp}
    Let $X\in \mathcal{C}^{\alpha}([0,T]; V)$ and $M \in \mathcal{M}^{c, \loc, 1}(W)$. Then the Itô integrals
    \begin{equation*}
     \mathbb{M}_{s,t}\coloneqq \int_{s}^{t} \delta M_{s,r} \otimes dM_{r} ~\text{ and  } ~ \Pi(X; M)_{s,t}\coloneqq \int_{s}^{t} \delta X_{s,r} \otimes dM_{r}
    \end{equation*}
    are well-defined Itô integrals. We define the following integrals via integration by parts (referred to as the IBP-integrals):
    \begin{equation*}
    \begin{aligned}
        \Pi(M; X)_{s,t} &\coloneqq \int_{s}^{t} \delta M_{s,r} \otimes dX_{r} :=  (\delta M)_{s,t} \otimes (\delta X)_{s,t} - \Pi(X; M)^{\top}; \\
        \int_{s}^{t} M_{r} \otimes dX_{r} &\coloneqq \Pi(M; X)_{s,t}+ M_{s} \otimes \delta X_{s,t}.
    \end{aligned}
    \end{equation*}
    Then, the following statements hold:
\begin{enumerate}
    \item \label{lemma:ibp_claim_1} The iterated integrals satisfy the Chen-type relations: for any $s,u,t \in [0,T]$ with $s< u < t$, it holds a.s. \begin{equation*}
    \begin{aligned}
    \delta \mathbb{M}_{s,u,t}&=\delta M_{s,u} \otimes \delta M_{u,t}; \\
    \delta \Pi(X; M)_{s,u,t}&= \delta X_{s,u} \otimes \delta M_{u,t}; \\
    \delta \Pi(M; X)_{s,u,t}&= \delta M_{s,u} \otimes \delta X_{u,t}.  
    \end{aligned}
    \end{equation*}
    \item \label{lemma:ibp_claim_2}$\mathbb{M}\in \mathcal{C}^{1-}_{2}([0,T]; W^{\otimes 2})$ almost surely.
    \item \label{lemma:ibp_claim_3}  $\Pi(X; M)\in \mathcal{C}_{2}^{(\frac{1}{2}+\alpha)-}([0,T]; V \otimes W)$ and $
    \Pi(M; X) \in \mathcal{C}_2^{(\frac{1}{2}+\alpha)-}([0,T]; W\otimes V)$ almost surely.
    \item \label{lemma:ibp_claim_4} For any sequence $(\mathcal{P}^{n})_{n\in \mathbb{N}}$  of partitions of $[s,t]$ with vanishing mesh-size, it holds 
    \begin{equation*}
        \sum_{[u,v]\in \mathcal{P}^{n}} \delta M_{s,u} \otimes \delta X_{u,v} \stackrel{\mathbb{P}}{\longrightarrow} \int_{s}^{t} \delta M_{s,r} \otimes dX_{r}, \quad \text{ as } n\to \infty. 
    \end{equation*}
\end{enumerate}
\end{lemma}

  \begin{proof}
Claim \eqref{lemma:ibp_claim_1} follows by simple calculations.  Now we prove Claim \eqref{lemma:ibp_claim_2}.  Let $\beta \in (0, 1)$ be fixed and $(\tau_{k})_{k \in \mathbb{N}}$ be the localizing sequence as in the proof of \cref{lemma:path_regularity_Martingale}. Then,  by the Burkholder–Davis–Gundy inequality, it holds for any $q \in [2, \infty)$,
  \begin{equation*}
  \begin{aligned}
      \left \Vert \mathbb{M}^{\tau_{k}}_{s,t}  \right\Vert_{L^{\frac{q}{2}}}= \left \Vert \int_{s}^{t} \delta M_{s,r}^{\tau_{k}} \otimes dM^{\tau_{k}}_{r} \right \Vert_{L^{\frac{q}{2}}} &\lesssim \left \Vert \sqrt{\int_{s}^{t} |\delta M_{s,r}^{\tau_{k}}|^{2} d\langle M \rangle^{\tau_{k}}_{r}}\right \Vert_{L^{\frac{q}{2}}}\\
      &\leq \sqrt{k} |t-s|^{\frac{\beta}{2}} \left \Vert \sup_{r \in [s,t]} |\delta M^{\tau_{k}}_{s,r}| \right \Vert_{L^{\frac{q}{2}}} \leq k |t-s|^{ \beta}.
  \end{aligned}
  \end{equation*}
Applying the Kolmogorov criterion for random rough paths (Theorem 3.1 in \cite{friz_2020}) gives $\mathbb{M}^{\tau_{k}}\in \mathcal{C}^{\beta-}_{2}([0,T]; V^{\otimes 2})$. Claim \eqref{lemma:ibp_claim_2} then follows  by noting that  $\mathbb{M}=\mathbb{M}^{\tau_{k}}$  on  $\{T \leq \tau_{k}\}$ and  that $(\tau_{k})_{k \in \mathbb{N}}$ is  a localizing sequence.

For Claim \eqref{lemma:ibp_claim_3},  applying again the Burkholder–Davis–Gundy inequality, we get  for any $q \in [2, \infty)$, 
  \begin{equation*}
      \Vert \Pi(X; M)^{\tau_{k}}_{s,t} \Vert_{L^{\frac{q}{2}}} \lesssim |t-s|^{\frac{\beta}{2}+\alpha}.
  \end{equation*}
  Applying the generalized rough path Kolmogorov criterion (\cref{appendix:theorem:Kolmogorov}) to $\Pi(X; M)^{\tau_{k}}$ (setting $A^{2}\equiv 0, A^{1}=\Pi(X; M), A^{1,1}=X, A^{1,2}=M$ in the context of the theorem) yields $\Pi(X;M)^{\tau_{k}}\in \mathcal{C}^{(\frac{\beta}{2}+\alpha)-}([0,T]; V \otimes W)$ a.s. and so $\Pi(X;M) \in \mathcal{C}^{(\frac{1}{2}+\alpha)-}([0,T]; V \otimes W)$ almost surely.  Then, it follows directly by the definition of $\Pi(M; X)$ that 
    $\Pi(M; X) \in \mathcal{C}_2^{(\frac{1}{2}+\alpha)-}([0,T]; W\otimes V)$. This completes the proof of  Claim \eqref{lemma:ibp_claim_3} .  
    
    Finally, we prove Claim \eqref{lemma:ibp_claim_4}.  Note that for any $(s,t)\in \Delta_{T}$ and any partition $\mathcal{P}=(t_{i})_{i=0}^{N}$ of $[s,t]$ it holds
  \begin{equation*}
      \delta X_{s,t} \otimes \delta M_{s,t}= \sum_{i=0}^{N-1} \delta X_{t_{i}, t_{i+1}} \otimes \delta M_{s,t_{i}}+ \sum_{i=0}^{N-1} \delta X_{s, t_{i}} \otimes \delta M_{t_{i}, t_{i+1}} + \sum_{i=0}^{N} \delta X_{t_{i}, t_{i+1}} \otimes \delta M_{t_{i}, t_{i+1}}.
  \end{equation*}
As $|\mathcal{P}|\downarrow 0$ , the 2nd sum on the right-hand side   converges to the Itô integral $\Pi(X; M)$, while the 3rd sum converges to $0$ in probability noting that  $X$ is deterministic (see \cite{fzk23}).  This together with the definition of $\Pi(M;X)$ implies Claim \eqref{lemma:ibp_claim_4}.   \end{proof}
    
  In \cite{fzk23}, the authors show that rough semimartingales are in one-to-one correspondence with a.s. controlled rough paths. For the reader’s convenience, we briefly reprove this fact in our setting.
\begin{corollary}
  \label{cor:CRPtoRSM} Consider a continuous, adapted process
  \begin{equation*}
      (Y, \partial_{X} Y , \partial_{M}Y): [0,T] \times \Omega \to W \times \mathcal{L}(V; W) \times \mathcal{L}(U; W)
  \end{equation*} and $M \in \mathcal{M}^{c, \loc, 1}(U)$ such that a.s.
    \begin{equation*}
        Y_{t} \stackrel{2\alpha}{=} Y_{s} + (\partial_{X} Y)_{s} \delta X_{s,t}+ (\partial_{M} Y)_{s} \delta M_{s,t}; \; 
        \; \; \; (\partial_{X} Y, \partial_{M} Y)_{t} \stackrel{\alpha}{=} (\partial_{X} Y, \partial_{M} Y)_{s}
    \end{equation*}
    Then a.s.
    \begin{equation*}
  \delta (Y - N)_{s,t} \stackrel{2 \alpha}{=} (\partial_X Y)_{s} \delta X_{s,t}, 
  \end{equation*}
  where
  \begin{equation*}
     N_{t} \coloneqq \int_0^{t} (\partial_M Y)_r d M_r \in \mathcal{M}^{c,\loc, 1}(W).
     \end{equation*}
  In words, an adapted $ (X, M)$-controlled $\alpha$-Hölder rough path, with $M \in \mathcal{M}^{c, \loc, 1}$, induces an $X$-controlled, $\alpha$-Hölder, rough 
  semimartingale $(Y, \partial_X Y; N)$. Conversely any $X$-controlled, $\alpha$-Hölder rough semimartingale $(Y, \partial_{X} Y; M)$ induces an a.s. $(X, M)$-controlled, $\alpha$-Hölder, rough path $(Y, Y') \in \mathscr{D}^{2 \alpha}_{(X, M)}$ with $Y'\coloneqq (\partial_{X} Y, \operatorname{Id}_{W})^{\top}$.
\end{corollary}
\begin{proof}
The first claim follows immediately,  once we
verify  that  it holds almost surely
\begin{equation}
  \label{eq:CRPtoRSMestimate} \delta N_{s,t} \stackrel{2 \alpha}{=}(\partial_M Y)_{s} \delta M_{s,t}.
\end{equation}
 Since a.s. $Y\in \mathcal{C}^{\alpha}([0,T]; W)$, \eqref{eq:CRPtoRSMestimate} follows by an analogous localization argument to \cref{lemma:ibp} claim \eqref{lemma:ibp_claim_3} by considering for any $\beta \in (0, \alpha)$ and $\gamma \in (0, 1)$
\begin{equation*}
    \tau_{k}\coloneqq \inf \left \{ t\geq 0\big| \left(\Vert \partial_{M}Y \Vert_{\beta; [0,t]}\vee \Vert \langle M \rangle \Vert_{\gamma; [0,t]}\right) \geq k \right \}. 
\end{equation*}
Again by Theorem 5.31 in \cite{Friz_Victoir_2010}, this is a sequence of stopping times and naturally a localizing sequence. The converse claim follows trivially.
\end{proof}
\begin{definition}[Rough stochastic integral] \label{def:RSI}
\label{rough_stoch_integral} Let $\mathbf{X}=(X, \mathbb{X})\in \mathscr{C}^{\alpha}([0,T]; V)$ be a (deterministic) rough path and  $(Y,\partial_{X}Y; M)$ be an $X$-controlled, $\mathcal{L}(V; W)$-valued
  rough semimartingale. Then we define the
  \textit{rough stochastic integral}
  \begin{equation*}
   \int_{0}^{t} (Y, \partial_X Y)_{r} d \mathbf{X}_{r} : = \int_{0}^{t} (Y - M,
     \partial_X Y)_{r} d \mathbf{X}_{r} + \int_{0}^{t} M_{r} d X_{r}, 
    \end{equation*}
  where the integrals on the right-hand side exist respectively
  as a (classical)  rough integral and an IBP-integral in the sense of \cref{lemma:ibp}. 
\end{definition} 
We note that
\[ \int_{0}^{t} (Y, \partial_X Y)_{r} d \mathbf{X}_{r} \sim \sum_{[u, v] \in \pi} (Y_u \delta X_{u, v} + (\partial_X Y)_u
   \mathbb{X}_{u, v}) \]
in the sense of convergence in probability, as $| \pi |\downarrow 0$.

As demonstrated in the proofs above, the localization technique will play a central role in what follows. Combined with (higher-order) Kolmogorov continuity criteria (see \cref{appendix:theorem:Kolmogorov}), it enables us to derive a.s. pathwise properties from  $L^{p}$-bounds.  The following  result ensures the compatibility of these localization procedures with  rough stochastic integrals, analogously to \cite{fzk23}.
\begin{lemma}
  \label{lemma:Stopped_RSI}Let $(Y, \partial_{X} Y; M)$ be an $X$-controlled, $\mathcal{L}(V; W)$-valued rough
  semimartingale and $\tau$ be a stopping time.
  Then it holds for any $t\in [0,T]$, that a.s.
  \begin{eqnarray*}
    \left( \int_{0}^{t} (Y, \partial_{X} Y)_{r}  d \mathbf{X}_{r}
    \right)^{\tau} & = & \left( \int_{0}^{t} (Y-M, \partial_{X} Y)_{r}  d
    \mathbf{X}_{r} \right)^{\tau} + \delta(M X)^{\tau}_{0,t} + \left(  \int_{0}^{t} X_{r} dM_{r} \right)^{\tau} \nonumber \\
    & = & \int_{0}^{t} (Y-M, \partial_{X} Y)^{\tau}_{r} d
    \mathbf{X}^{\tau}_{r} + \delta (MX)_{0,t}^{\tau} +\int_{0}^{t} X^{\tau}_{r} dM_{r}^{\tau}.  \nonumber
  \end{eqnarray*}
  In particular, we have the following consistency
  \begin{equation*}
      \left( \int_{0}^{t} (Y, \partial_{X} Y)_{r} d\mathbf{X}_{r} \right)\mathbbm{1}_{t \in [0, \tau]} = \left(\int_{0}^{t} (Y-M, \partial_{X} Y)^{\tau}_{r} d
    \mathbf{X}^{\tau}_{r} + \delta (MX)_{0,t}^{\tau} + \int_{0}^{t} X_{r}^{\tau} dM^{\tau}_{r} \right) \mathbbm{1}_{t \in [0, \tau]}
  \end{equation*}
  almost surely.
  \begin{proof}It follows immediately by the construction of the
   integral $\Pi(M;X)=\int M dX$ through the integration by parts identity in  \cref{lemma:ibp}  that
  \begin{equation*}
      \left( \int_0^t M_{r} dX_{r} \right)^{\tau} = \delta
     (\ensuremath{MX})^{\tau}_{0, t} + \left( \int_0^t X_{r} 
     dM_{r} \right)^{\tau} = \delta
     (MX)^{\tau}_{0, t} + \int_0^t X^{\tau}_{r}
     dM_{r}^\tau , 
     \end{equation*}
  where the second equality is due to a standard property of Itô integrals. 
  The identity
  \begin{equation*}
  \left( \int_{0}^{t} (Y-M, \partial_X Y)_{r}  d \mathbf{X}_{r}
     \right)^{\tau} = \int_{0}^{t} (Y-M, \partial_X Y)_{r}^{\tau} d
     \mathbf{X}^{\tau}_{r}
    \end{equation*}
  follows by the application of the sewing lemma (see, e.g.,  \cite[Lemma 4.2]{friz_2020}), noting $(Y-M, \partial_{X}Y)^{\tau} \in \mathscr{D}^{2 \alpha}_{X^{\tau}}$ almost surely. The consistency then follows trivially.  
  \end{proof}
\end{lemma}
We close off this expository section by mentioning that rough semimartingales are a priori distinct from the notion of stochastic controlled rough paths introduced in \cite{friz_2021}.

\subsection{Strongly controlled rough semimartingales and
paths}\label{Rough_It\^{o}_proc}
As expected by the classical rough Itô formula, see e.g. \cite{friz_2020} Section~7.5 or our rough Itô-Wentzell formula \cref{cor:RIW_formula}, a rough stochastic Itô(-Wentzell) formula  needs to involve
some notion of ``2nd order controlledness'' w.r.t. a suitable reference rough path as well as suitable probabilistic components. In \cref{cor:CRPtoRSM}, we have already seen the connection between rough semimartingales and a.s. controlled rough paths. In this section, we aim to  extend this idea to ``2nd order rough semimartingales'' and a.s. strongly controlled rough paths. To define the latter w.r.t. $(X, M)$ we need to define a ``joint lift'' of $(X, M)$. 
\begin{lemma}
\label{lemma:joint_lift}
    Let $\mathbf{X}=(X, \mathbb{X})\in \mathscr{C}^{\alpha}([0,T]; \mathbb{R}^{d_{X}})$ and $M \in \mathcal{M}^{c, \loc, 1}(\mathbb{R}^{d_{M}})$. Then we define     \begin{equation*}
        (\mathbf{X}; M)\coloneqq \left((X;M), (\mathbb{X}; \mathbb{M})\right)\coloneqq \left( \left(\begin{array}{c}
        X\\
        M
        \end{array}
        \right), 
        \left(\begin{array}{cc}
        \mathbb{X} & \Pi(X; M)\\
        \Pi(M; X) & \mathbb{M}
        \end{array} \right)\right),
    \end{equation*}
    where $\mathbb{M}, \Pi(X; M)$ and  $\Pi(M; X)$ 
    are the iterated integrals
    given in \cref{lemma:ibp}. 
    Then $(\mathbf{X}; M) \in \mathscr{C}^{\alpha}([0,T]; \mathbb{R}^{d_{X}}\oplus \mathbb{R}^{d_{M}})$ 
    almost surely. Further we have a.s.
    \begin{equation*}
        [(\mathbf{X}; M)]=\left(\begin{array}{cc}
        [\mathbf{X}] & 0\\
        0 & \langle M \rangle
        \end{array} \right).
    \end{equation*}
If $\mathbf{X}\in \mathscr{C}^{\alpha,1}([0,T]; \mathbb{R}^{d{X}})$, then $(\mathbf{X}; M)\in \mathscr{C}^{\alpha, 1}([0,T]; \mathbb{R}^{d_{X}}\oplus \mathbb{R}^{d_{M}})$.
\end{lemma}

\begin{proof}
It follows by \cref{lemma:ibp}  that  $(\mathbb{X}; \mathbb{M})\in \mathcal{C}^{2\alpha}_{2}([0,T];(\mathbb{R}^{d_{X}} \oplus \mathbb{R}^{d_{Y}})^{\otimes 2}) $ almost surely. The Chen's relation is a  direct consequence of  the construction, noting that  all of the blocks of $(\mathbb{X}; \mathbb{M})$ satisfy Chen-type relations \eqref{lemma:ibp_claim_1} in  \cref{lemma:ibp} .

    Regarding the bracket process, note that we have
    \begin{equation*}
    \operatorname{Sym}(\mathbb{X}; \mathbb{M}) = \frac{1}{2}
     \left(\begin{array}{cc}
       2\ensuremath{\operatorname{Sym}} (\mathbb{X}) &  \Pi (X ;
       M) + \Pi(M; X)^{\top}\\
       \Pi (M ; X) + \Pi (X ; M)^{\top} & 2\ensuremath{\operatorname{Sym}} (\mathbb{M})
     \end{array}\right) 
     \end{equation*}
  and
  \begin{equation*}
  (\delta (X; M)_{s, t})^{\otimes 2} = \left(\begin{array}{cc}
       (\delta X_{s, t})^{\otimes 2} & \delta X_{s, t} \otimes \delta M_{s,
       t}\\
       \delta M_{s, t} \otimes \delta X_{s,t} & (\delta M_{s, t})^{\otimes 2}
     \end{array}\right). 
  \end{equation*}   
  Now note that for any $k, l \in \{ 1, \ldots, d_M \} $ we have a.s.
  \begin{equation*}
      \delta M^k_{s, t} \delta M_{s, t}^l = \int_s^t M^{k }_{s, r}
    dM^l_r + \int_s^t M^{l }_{s, r}
    dM^k_r + \langle M^k, M^l\rangle_{s, t}, 
  \end{equation*}   
  and for the off-diagonal terms,  we have for any $k \in \{ 1, \ldots, d_X \},
  l \in \{ 1, \ldots, d_M \}$ that,  by the definition of $\Pi (M ; X)$ ,
  \[ \delta X_{s, t}^k \delta M_{s, t}^l = \int_s^t X^k_{s, r}
    dM_r^l + \int_s^t M^l_{s, r}
  dX_r^k = \Pi (X^k ; M^l)_{s, t} + \Pi (M^l ;
     X^k)_{s, t},\]
which yields the claim.
\end{proof}

\begin{remark}
The notation introduced in \cref{lemma:joint_lift} can be easily extended to any finite number of local martingales. To this end,  let $\mathbf{X}, M$ be as above and $N \in \mathcal{M}^{c, \loc, 1}(\mathbb{R}^{d_{N}})$. Then we define (noting that $\langle M, N \rangle \in \Lip^{1}$ a.s. due to polarization ) 
\begin{equation*}
    (\mathbf{X}; M; N)\coloneqq (\mathbf{X}; L); \quad L\coloneqq \left(\begin{array}{c}
    M \\
    N 
    \end{array}\right)
    \in \mathcal{M}^{c, \loc, 1}(\mathbb{R}^{d_{M}} \oplus \mathbb{R}^{d_{N}}).
\end{equation*}
Explicitly, we define it to be 
\begin{equation*}
    (\mathbf{X}; M; N)=\left(\left(\begin{array}{c}
    X\\
    M \\
    N
    \end{array}\right), \left(\begin{array}{ccc}
      \mathbb{X} & \Pi (X; M) & \Pi (X; N)\\
      \Pi (M; X) & \mathbb{M} & \Pi(M; N) \\
      \Pi (N ; X) & \Pi (N ; M) & \mathbb{N}
    \end{array}\right)\right), 
\end{equation*}
where all terms involving $X$ are given by the corresponding IBP-integrals in \cref{lemma:ibp} and the terms involving only $M, N$ are the corresponding Itô integrals. Of course $(\mathbf{X}; M; N)\in \mathscr{C}^{\alpha}([0,T]; \mathbb{R}^{d_{X}} \oplus \mathbb{R}^{d_{M}} \oplus \mathbb{R}^{d_{N}})$.
\end{remark}
\begin{proposition}
  \label{prop:SCRP_to_RIP} Let $\mathbf{X} \in \mathscr{C}^{\alpha}([0,T]; \mathbb{R}^{d_{X}})$ and $M \in \mathcal{M}^{c, \loc, 1}(\mathbb{R}^{d_{Y}})$. Consider $Y, Y', Y''$ continuous, adapted processes, where $Y$ is $\mathbb{R}^{d_{Y}}$-valued and 
  \begin{equation*}
  \begin{aligned}
      Y'_{t}&\eqqcolon \left(\left( 
    \begin{array}{c}
    \partial_1 Y_{t}\\
    \partial_2 Y_{t}
    \end{array} \right)\cdot \right) \in \mathcal{L}(\mathbb{R}^{d_{X}} \oplus \mathbb{R}^{d_{Y}}; \mathbb{R}^{d_{Y}}) ;\\
    Y''_{t}&\eqqcolon \left(\left( \begin{array}{ll}
    \partial_{1,1} Y_{t} & \partial_{1,2} Y_{t}\\
    \partial_{2,1} Y_{t} & \partial_{2,2} Y_{t}
  \end{array} \right) : \right) \in \mathcal{L}((\mathbb{R}^{d_{X}} \oplus \mathbb{R}^{d_{Y}})^{\otimes 2}; \mathbb{R}^{d_{Y}}).
  \end{aligned}
  \end{equation*}
  Suppose $(Y, Y', Y'', 0) \in \mathscr{D}^{3 \alpha}_{(\mathbf{X}; M)}$ almost surely. Then for $i = 1, 2$, $(\partial_i Y,
\partial_{1, i} Y; N_{i})$ is an $X$-controlled, $\alpha$-Hölder, rough semimartingale with $N_i
:= \int_0 (\partial_{2, i} Y) d M \in \mathcal{M}^{c, \loc, 1}$. If further $3 \alpha>1$, then almost surely
\begin{equation}
  \label{prop:eq:SCRPtoRIP} Y = N + \int (\partial_{X} Y, \partial_{X}^{2}Y) d \mathbf{X}, \quad \text{ where } N  :=
  \int_0 \partial_{M}Y d M \in \mathcal{M}^{c,\loc,1},  
\end{equation}
where $\partial_{X}Y\coloneqq\partial_{1}Y, \partial^{2}_{X}Y \coloneqq\partial_{1,1} Y$ and $\partial_{M}Y\coloneqq \partial_{2}Y$. 
Here, the $d \mathbf{X}$-integral in \cref{prop:eq:SCRPtoRIP} is a rough stochastic integral in the sense of \cref{def:RSI}.
\end{proposition}


\begin{proof}  
Note first that the operator $Y'' (X; M) \in \mathcal{L}((\mathbb{R}^{d_{X}}\oplus \mathbb{R}^{ d_{Y}}); \mathbb{R}^{d_{Y}})$ is given by
\begin{equation*}
\begin{aligned}
    (Y''(X; M))(z)= Y''\left( (X; M) \otimes z\right)=Y''\left(\begin{array}{cc}
    X \otimes z_{1} & X \otimes z_{2} \\
    M \otimes z_{1} & M \otimes z_{2} \end{array}\right)
\end{aligned}    
\end{equation*} for any $z=z_{1} \oplus z_{2}\in \mathbb{R}^{d_X}\oplus \mathbb{R}^{d_{Y}}$. Therefore $(\partial_i Y, \partial_{1, i} Y)$ being an 
$X$-controlled rough semimartingales with local martingale part $N_{i}\in \mathcal{M}^{c, \loc, 1}$ is a
direct consequence of \cref{cor:CRPtoRSM} and $(Y', Y'')\in \mathscr{D}^{2 \alpha}_{(X; M)}$ almost surely. Thus the only claim left to check
is the identity \eqref{prop:eq:SCRPtoRIP}.

Fix $k>0$ and assume for now that for any $\beta \in (0, \alpha), \gamma \in (0, 1)$ it holds almost surely
\begin{equation}
\label{prop:eq:assumption} 
\Vert \partial_{2} Y \Vert_{[0,T]; \beta} \vee \Vert \partial_{2,1} Y \Vert_{[0,T]; \beta} \vee 
    \Vert \langle M \rangle \Vert_{[0,T]; \gamma} \vee \Vert \langle N_{1
    } \rangle\Vert_{[0,T]; \gamma} \leq k. 
\end{equation} Denote for any $(s, t) \in \Delta_T$  the
remainder
\begin{equation*}
R^{N_1}_{s, t} := \delta \left( {N_1}  \right)_{s, t} -\partial_{2, 1}
   Y_s \delta M_{s, t} = \int_{s}^{t} \partial_{2,1} Y_{r} dM_{r}- \partial_{2,1} Y_{s} \delta M_{s,t}= \Pi(\partial_{2,1} Y; M)_{s,t}. 
\end{equation*}
Note that by an analogous argument to \cref{lemma:ibp} \eqref{lemma:ibp_claim_3} we see that 
$ R^{ N_1 } \in \mathcal{C}^{\frac{1}{2}+\alpha-}_2 $ almost surely. Further, for any fixed $s\in [0,T]$,  it holds  that $[s, T]\ni t\mapsto R^{N_{1}}_{s,t}$ is a continuous local martingale w.r.t. the filtration $(\mathfrak{F}_{t})_{t\in [s, T]}$ and so the iterated integral
\begin{equation*}
    \Pi(R^{N_{1}}; X)_{s,t}\coloneqq \int_{s}^{t} R^{N_{1}}_{s,r} dX_{r}
\end{equation*}
is well defined as an IBP-Integral in \cref{lemma:ibp}. Further analogously to \cref{Iterated_Martingale_Inequality} we verify the algebraic identity for any
$s,u,t \in [0,T]$ with $s< u< t$
\begin{equation*}
\delta \Pi (R^{N_1} ; X)_{s, u, t} = -
   \delta \left( N_1\right)_{s, u}   \Pi (M ; X)_{u, t} + R^{N_1}_{s,
   u} \delta X_{u, t} 
\end{equation*}
and so we see that \cref{appendix:theorem:Kolmogorov} is applicable. Further note that by uniqueness of limits in probability it holds for any $(s,t)\in \Delta_{T}$ a.s.:
\begin{equation}
\label{eq:It\^{o}_multiplication}
   \int_s^t \partial_{2,1} Y_s (\delta
  M)_{s, r} \otimes d X_{r} =\partial_{2,1} Y_s \int_s^t (\delta
  M)_{s, r} \otimes d X_{r}  
\end{equation}
and since both expressions are continuous in $(s,t)$ this holds up to indistinguishability. Thus it holds a.s for any $(s, t) \in \Delta_T$, by \cref{appendix:theorem:Kolmogorov},  
\begin{equation}\label{e:Pi-R-N}
   \int_s^t N_1 (r) dX_{r} - \left(
  N_1 (s)  \delta X_{s, t} + \partial_{2,1} Y_s \int_s^t (\delta
  M)_{s, r} \otimes d X_{r} \right)=  \Pi (R^{N_1} ; X)_{s, t} 
  \stackrel{3 \alpha}{=} 0.
\end{equation}
Now note that again by \cref{appendix:theorem:Kolmogorov} it holds a.s. 
\begin{equation}\label{e:delta-N}
\begin{aligned}
  \delta N_{s, t} - (\partial_2 Y)_s \delta M_{s, t} & \stackrel{3\alpha}{=} 
  \int_s^t (\partial_{1, 2} Y_s \delta X_{s, r} + \partial_{2, 2} Y_s \delta
  M_{s, r}) dM_r \\
  & =  \partial_{1, 2} Y_s \Pi ( X ; M)_{s, t} + \partial_{2, 2} Y_s
  \mathbb{M}_{s, t}, 
\end{aligned}
\end{equation}
where again we used an analogous argument to \eqref{eq:It\^{o}_multiplication} to deduce the second equality. Recall that by \cref{rough_stoch_integral} we have
\begin{equation}\label{e:--dX}
\begin{aligned}
   \int_{s}^{t} (\partial_1 Y , \partial_{1, 1} Y)_{r} d
  \mathbf{X}_{r}  & \stackrel{3 \alpha}{=} (\partial_1 Y -
  N_1)_{s} \delta X_{s,t} + \partial_{1, 1} Y_{s} \mathbb{X}_{s,t} + \int_{s}^{t} (N_{1})_{r} dX_{r}\\
  & =  \partial_1 Y_{s} \delta X_{s,t} + \partial_{1, 1} Y_{s} \mathbb{X}_{s,t} + \Pi (N_1 ; X)_{s,t}. 
  \end{aligned}
\end{equation}
Recalling that  $(Y, Y', Y'', 0) \in \mathscr{D}^{3 \alpha}_{(\mathbf{X}; M)}$,  we get 
\begin{eqnarray*}
    \delta Y_{s,t} &\stackrel{3\alpha}{=}& Y'_{s} \delta (X; M)_{s,t} + Y''_{s} (\mathbb{X}; \mathbb{M})_{s,t} \nonumber \\
    &=&\partial_1 Y_{s} \delta X_{s,t}+ \partial_{2} Y_{s} \delta M_{s,t} + \partial_{1, 1} Y_{s} \mathbb{X}_{s,t}+ \partial_{1,2} Y_{s} \Pi(X; M)_{s,t} \nonumber \\
    &&+ \partial_{2,1} Y_{s} \Pi(M; X)_{s,t}+ \partial_{2,2} Y_{s} \mathbb{M}_{s,t} \nonumber \\
    & \stackrel{3\alpha}{=}& \int_{s}^{t}(\partial_{1} Y, \partial_{1,1} Y)_{r} d\mathbf{X}_{r}+  \delta N_{s,t}-\Pi(N_1;X)_{s,t}+(\partial_{2,1} Y)_{s} \Pi(M; X)_{s,t} \nonumber \\
    &  \stackrel{3\alpha}{=}& \int_{s}^{t}(\partial_{1} Y, \partial_{1,1} Y)_{r} d\mathbf{X}_{r}+ \delta N_{s,t}\nonumber ,
\end{eqnarray*}
where the third step follows from \eqref{e:delta-N} and \eqref{e:--dX}, and the last step follows from the fact $\Pi(N_1;X)_{s,t}-(\partial_{2,1} Y)_{s} \Pi(M; X)_{s,t}=\Pi(R^{N_1} ; X)_{s, t}$ and the estimate \eqref{e:Pi-R-N}. This yields the desired result. Now to get rid of the assumption \eqref{prop:eq:assumption} in the general case, one performs an analogous localization argument as in the proof of \cref{lemma:ibp} using \cref{lemma:Stopped_RSI} and Theorem~5.31 in \cite{Friz_Victoir_2010}.
\end{proof}

 \cref{prop:SCRP_to_RIP} suggests an analogous relationship to \cref{cor:CRPtoRSM} in the setting of strongly $(\mathbf{X}; M)$-controlled rough paths and processes of the form \eqref{prop:eq:SCRPtoRIP}. As these processes will play a central role in what follows, we now introduce specific terminology for them.

\begin{definition}
  \label{def:Rough_Ito_process}
  
  Consider a $6$-tuple $\mathcal{Y}=(Y, \partial_{X} Y, \partial^{2}_{X}Y, \dot{Y}; M , N)$ of continuous, adapted processes of suitable dimensions:
\begin{equation*}
\mathcal{Y}: [0, T] \times \Omega \longrightarrow \mathbb{R}^{d_{Y}} \times
     \mathcal{L} (\mathbb{R}^{d_{X}}; \mathbb{R}^{d_{Y}}) \times \mathcal{L} (\mathbb{R}^{d_{X}} \otimes \mathbb{R}^{d_{X}}; \mathbb{R}^{d_{Y}}) \times \mathbb{R}^{d_{Y}} \times \mathbb{R}^{d_{Y}} \times \mathcal{L} (\mathbb{R}^{d_{X}}; \mathbb{R}^{d_{Y}}), 
\end{equation*}
  where $(\partial_{X} Y, \partial^{2}_{X} Y; N)$ is an $X$-controlled, $\alpha$-Hölder rough semimartingale. Further assume $M \in \mathcal{M }^{c, \loc,1}(\mathbb{R}^{d_{Y}})$, $\dot{Y}\in L^{\infty}([0,T]; \mathbb{R}^{d_{Y}})$ and 
  \begin{equation*}
  Y_{t}=Y_{0} + \int_{0}^{t} \dot{Y}_{s} ds + M_{t} + \int_{0}^{t} (\partial_{X} Y, \partial^{2}_{X} Y)_s d \mathbf{X}_s
  \end{equation*}
  for any $t \in [0,T]$, almost surely. 
  Then  the process $\mathcal{Y}$ is called a \textit{strongly} $\mathbf{X}$-\textit{controlled} \textit{rough semimartingale} (short: $\mathbf{X}$-scRSM). 
  \end{definition}
\begin{remark}\label{remark:Lipschitz_characteristics}  Setting $\partial_{X}Y, \partial_{X}^{2}Y\equiv 0$ in \cref{def:Rough_Ito_process} yields the class of \textit{continuous semimartingales with Lipschitz characteristics}, which is a natural class of processes to consider when working in the Hölder scale. 
\end{remark}

We have thus seen that every strongly $(\mathbf{X}; M)$-controlled rough path induces an $\mathbf X$-scRSM. It turns out that, in analogy with \cref{cor:CRPtoRSM}, the converse implication holds as well.

\begin{proposition}
  \label{prop:RIP_to_SCRP}Let $(Y, \partial_{X} Y, \partial^{2}_{X} Y, \dot{Y}; M , N)$ be an $\mathbf{X}$-scRSM. 
  Then there are continuous adapted processes $Y', Y''$ of the form
  \begin{equation}
  \label{prop:eq:RIP_to_SCRP}
  \begin{aligned}
      Y'_{t}&\coloneqq \left( \left( \begin{array}{c}
      \partial_{X} Y_{t}\\
      \operatorname{Id}\\
      0
      \end{array}\right) \cdot \right)\in \mathcal{L}(\mathbb{R}^{d_{X}}\oplus \mathbb{R}^{d_{Y}} \oplus \mathbb{R}^{d_{Y}}; \mathbb{R}^{d_{Y}}) \\
      Y''_{t} &\coloneqq \left(\left(\begin{array}{ccc}
          \partial_{X}^{2} Y_{t} & 0 &0\\
          0&0&0 \\
          \operatorname{Id} &0&0
      \end{array}\right):\right)\in \mathcal{L}((\mathbb{R}^{d_{X}}\oplus \mathbb{R}^{d_{Y}} \oplus \mathbb{R}^{d_{Y}})^{\otimes 2}; \mathbb{R}^{d_{Y}})
  \end{aligned}    
  \end{equation}
  such that $(Y - V, Y', Y'', 0)\in \mathscr{D}^{3 \alpha}_{(\mathbf{X};M; N)}([0,T]; \mathbb{R}^{d_{Y}})$ a.s., where $V=\int_{0}^{\cdot} \dot{Y}_{s} ds$.
\end{proposition}
\begin{proof}
By the construction of the rough stochastic integral
in \cref{rough_stoch_integral} we see immediately that almost-surely
\begin{equation*}
\begin{aligned}
\int_{s}^{t} (\partial_{X} Y, \partial_{X}^{2} Y)_r d \mathbf{X}_r 
   &\stackrel{3 \alpha}{=} (\partial_{X}Y - N)_{s} \delta X_{s,t} + (\partial_{X}^{2} Y)_{s} \mathbb{X}_{s,t} + \int_{s}^{t}
   N_{r} dX_{r} \\
   &= (\partial_{X}Y)_{s} \delta X_{s,t} + (\partial_{X}^{2} Y)_{s} \mathbb{X}_{s,t} + \Pi(N; X)_{s,t}
\end{aligned}
\end{equation*}
The rest follows trivially. 
\end{proof}
\begin{remark}
  Note that \cref{prop:SCRP_to_RIP} and \cref{prop:RIP_to_SCRP} are consistent with each other. If we start with an $\mathbf{X}$-scRSM $(Y, \partial_{X} Y, \partial^{2}_{X} Y, \dot{Y}; M , N)$, then by applying \cref{prop:RIP_to_SCRP} we get $(Y- V, Y', Y'')\in \mathscr{D}^{3 \alpha}_{(\mathbf{X}; M; N)}$ with $(Y', Y'')$ as in \eqref{prop:eq:RIP_to_SCRP}. On the other hand,  by defining $(Y,Y',Y'',   \dot Y_s )\in \mathcal D^{3\alpha}_{(\mathbf X;M;N)}$ with
  \[Y':=\left(\begin{array}{c}
       \partial_1 Y\\
       \hline 
       \partial_2 Y
     \end{array}\right) \coloneqq \left(\begin{array}{c}
       \partial_{X}Y\\
       \hline
       \operatorname{Id}\\
       0
     \end{array}\right),\quad  Y'':= \left(\begin{array}{c|c}
       \partial_{1, 1} Y & \partial_{1, 2} Y\\
       \hline 
       \partial_{2, 1} Y & \partial_{2, 2} Y
     \end{array}\right) \coloneqq \left(\begin{array}{c|cc}
       \partial_{X}^{2} Y & 0 & 0\\
       \hline
       0 & 0 & 0\\
       \operatorname{Id} & 0 & 0
     \end{array}\right), \]and applying \cref{prop:SCRP_to_RIP} w.r.t. the rough path $(\mathbf{X};L)$ with $L\coloneqq (M, N)^{\top}$, we obtain again the original scRSM $(Y, \partial_{X} Y, \partial^{2}_{X} Y, \dot{Y}; M , N)$. 
\end{remark}
\subsection{The rough stochastic Itô-Wentzell formula}
\label{subsection:The Rough Stochastic It\^{o} Wentzell Formula}
In this section, we establish a composition rule for evaluating suitably parameter-dependent scRSM at other scRSMs.  
In \cref{subsubsection:RSIW_1} we obtain the first part of this \textit{rough stochastic It{\^o}-Wentzell }(rsIW) formula by
utilizing the joint lift techniques for rough semimartingales developed in
\cref{Rough_It\^{o}_proc} in combination with the rough calculus of \cref{section:Introduction to controlled fields}.  In \cref{subsubsection:RSIW_2}, we obtain the second part of the rsIW-formula, by following a more classically  probabilistic approach. 
\subsubsection{The rough stochastic Itô--Wentzell formula with controlled fields}
\label{subsubsection:RSIW_1}
Let us first give the following extension of \cref{def:C3_jet}. 
\begin{definition}
  \label{def:random_C3_jet}
  We say a random field 
  \begin{equation*}
      G: [0,T]\times \Omega \times W \to U
  \end{equation*} 
  is $(\mathfrak{F}_{t})$-adapted, if for any $x\in W$, the process $(G_{t}(x))_{t\in [0,T]}$ is $(\mathfrak{F}_{t})$-adapted. We call the $7$-tuple $\mathcal{F}=(F, F', \partial F, F'',  \partial{F}', \partial^{2} F, \dot{F})$ of adapted random fields 
  \begin{equation*}  \begin{aligned}
      \mathcal{F}: [0,T] \times \Omega \times W &\to U \times \mathcal{L} (V ; U)
     \times \mathcal{L} (W ; U) \times \mathcal{L} (V^{\otimes 2} ; U) \times \mathcal{L} (W \otimes V ; U) \times
     \mathcal{S} (W^{\otimes 2} ; U) \times U\\
     (t,\omega, x)  &\mapsto (F_{t}(\omega,x) , F'_{t}(\omega, x), \partial F_{t}(\omega, x), F''_{t}(\omega,x), \partial F'_{t}(\omega, x), \partial^2
     F_{t}(\omega, x), \dot{F}_{t}(\omega, x))  
  \end{aligned}
  \end{equation*}
  an \textit{adapted} $\mathbf{X}$-\textit{controlled field}, if $\mathcal{F}\in \mathscr{D}^{3\alpha}_{\mathbf{X}}\Lip^{3}_{x, \loc}(W; U)$ almost surely. We denote the space of such random fields by $L^0_{\operatorname{ad}}\left (\Omega;  
\mathscr{D}^{3\alpha}_{\mathbf{X}}{\Lip}^{3}_{x, \loc} (W; U) \right)$. Analogously, one defines $L^0_{\operatorname{ad}}\left 
(\Omega;\mathscr{D}^{3\alpha}_{\mathbf{X}}{\Lip}^{3}_{x} (W; U)\right)$. 
\end{definition}

\begin{theorem}\label{theorem:RSIW_Theorem_1}Let $\mathbf{X} \in \mathscr{C}^{\alpha, 1}([0,T]; \mathbb{R}^{d_{X}})$and  $\mathcal{Y}=(Y, \partial_{X} Y, \partial_{X}^{2} Y, \dot{Y}; M,N)$ be an
  $\mathbf{X}$-scRSM.
  Let $\mathcal{F} = (F , F', \partial F, F'', \partial
  F', \partial^2 F, \dot{F })\in L^0_{\operatorname{ad}}\left (\Omega;  
\mathscr{D}^{3\alpha}_{\mathbf{X}}{\Lip}^{3}_{x, \loc}(\mathbb{R}^{d_{Y}}; \mathbb{R}^{d_{F}}) \right)$. Then we define the $6$-tuple
  \begin{equation*}
  \mathcal{F} \circ \mathcal{Y}\coloneqq ( Z, \partial_{X} Z, \partial^{2}_{X} Z, \dot{Z}; \tilde{M}, \tilde{N})
  \end{equation*}
  componentwisely by
  \begin{align*}
    Z_t & :=  F_{t}(Y_t); \nonumber  \\
    \partial_{X} Z_t & :=  F'_t(Y_{t}) + \partial F_t(Y_{t}) {\partial_{X} Y_t}; \nonumber  \\
    \partial_{X}^{2} Z_t & :=  \left( \partial F_{t}(Y_{t})  \right) {\partial^{2}_{X} Y_{t}}  + F''_t(Y_{t})  +
    \partial F'_t(Y_{t}) \partial_{X} Y_{t}  + (\partial F'_t(Y_{t}) \partial_{X} Y_{t})^T + \partial^2 F_t(Y_{t}) (\partial_{X} Y_{t},
    \partial_{X} Y_{t}); \nonumber \\
    \dot{Z}_{t}& := \partial F_t(Y_{t}) \dot{Y}_{t} + \dot{F}_{t}(Y_{t})+ \left({\partial F_t'} (Y_{t}) \partial_{X} Y_{t} +
    \tfrac{1}{2} \partial^2 F_t (\partial_{X} Y_{t}, 
    \partial_{X} Y_{t}) \right)\dot{[\mathbf{X}]}_{t} + \frac{1}{2} \partial^2 F_t(Y_{t}) \dot{\langle M\rangle}_{t};\nonumber\\
    \tilde{M}_t & :=  \int_0^t \partial F_s(Y_{s})
    dM_{s}; \nonumber \\
    \tilde{N}_{t} & :=  \int_0^t \partial F_{s}(Y_{s}) dN_{s} +
    \int_0^t \partial F'_{s}(Y_{s}) dM_{s} + \int_{0}^{t} \partial^2 F_s(Y_{s})
    (\operatorname{Id}, \partial_{X} Y_{s}) dM_{s}.\nonumber
    \end{align*}
  Then, $\mathcal{F} \circ 
  \mathcal{Y}$ is an $\mathbf{X}$-scRSM. In particular it holds 
  \begin{eqnarray*}
    Z_t & = & \int_0^t \Big(F'_{s}(Y_{s}) + \partial F_{s}(Y_{s}) {\partial_{X}Y_s}, \partial_{X}^{2} Z_{s}\Big)  d
    \mathbf{X}_{s} + \int_0^t \partial F_{s}(Y_{s})
    dM_{s} + \int_0^t \partial F_{s}(Y_{s}) \dot{Y}_{s} + \dot{F}_{s}(Y_{s}) ds \nonumber\\
    &  & + \int_0^t \left( \partial F'_{s}(Y_{s})  \partial_{X}Y_{s} + \frac{1}{2} \partial^2
    F_{s}(Y_{s}) (\partial_{X}Y_{s}, \partial_{X}Y_{s} )  \right) d [\mathbf{X}]_{s} +
    \frac{1}{2} \int_0^t \partial^2 F_{s}(Y_{s}) d \langle M\rangle_{s}.  \nonumber
  \end{eqnarray*}
\end{theorem}
\begin{proof}
In the following, $\Omega^{i}$ for any $i \in \mathbb{N}$ will denote a set $\Omega^{i}\in \mathfrak{F}$ such that $\mathbb{P}(\Omega^{i})=1$. Let $\Omega^{1}$ be such that for any $\omega \in \Omega^{1}$ it holds $\mathcal{F}(\omega)\in \mathscr{D}^{3\alpha}_{\mathbf{X}}\Lip^{3}_{x, \loc}$, and we consider $(\mathbf{X}; M; N)(\omega)\in \mathscr{C}^{\alpha, 1}([0,T]; \mathbb{R}^{d_{X}}\oplus \mathbb{R}^{d_{Y}}\oplus (\mathbb{R}^{d_{Y}}\otimes \mathbb{R}^{d_{X}}))$ for each $\omega \in \Omega^{2}$. Now, for any $\omega \in \Omega^{3}\coloneqq \Omega^{1}\cap \Omega^{2}$, it holds 
\begin{equation*}
    \tilde{\mathcal{F}}=(F, \tilde{F}',  \partial F, \tilde{F}'', \partial \tilde{F}', \partial^{2} F , \dot{F})(\omega) \in \mathscr{D}^{3 \alpha}_{(\mathbf{X}; M; N)(\omega)}\operatorname{Lip}^{3}_{x, \loc}
\end{equation*}
with
\begin{equation*}
       \tilde{F}' \coloneqq \left(\left(\begin{array}{c}
         F'\\
         0\\
         0
       \end{array}\right)\cdot \right); \quad
       \partial \tilde{F}' \coloneqq \left(\left(\begin{array}{c}
         \partial F'\\
         0\\
         0
       \end{array}\right)\cdot\right); \quad 
        \tilde{F}'' \coloneqq \left(\left(\begin{array}{ccc}
         F'' & 0 & 0\\
         0 & 0 & 0\\
         0 & 0 & 0
       \end{array}\right): \right). 
     \end{equation*}

  Further apply \cref{prop:RIP_to_SCRP} to $\mathcal{Y}$ to see that there is a set $\Omega^{4}$ such that for any $\omega \in \Omega^{4}$ it holds $(Y-V, Y', Y'', 0)(\omega) \in \mathscr{D}^{3 \alpha}_{(\mathbf{X}; M; N)(\omega)}$ with $(Y' , Y'')$ given by \eqref{prop:eq:RIP_to_SCRP} and $V= \int_{0}^{\cdot} \dot{Y_{r}} dr$.

  Now we compute some of the terms appearing in \cref{cor:RIW_formula}. For any $y=(y_{1}\oplus y_{2} \oplus y_{3}), z= (z_{1} \oplus z_{2} \oplus z_{3})\in \mathbb{R}^{d_{X}}\oplus \mathbb{R}^{d_{Y}}\oplus (\mathbb{R}^{d_{Y}} \otimes \mathbb{R}^{d_{X}})$ and $x \in \mathbb{R}^{d_{Y}}$, it holds by \eqref{prop:eq:RIP_to_SCRP},
  \begin{equation*}
  \begin{aligned}
      \partial \tilde{F}_t'(x) Y_t'(y \otimes z) &= \partial \tilde{F}_t'(x)(Y_t'(y)\otimes z)=\partial \tilde{F}_t'(x) ((\partial_{X} Y_t (y_{1})+ y_{2} )\otimes z)\\
      &=\partial F_t'(x) ((\partial_{X} Y_t(y_{1})+ y_{2} )\otimes z_{1}),
  \end{aligned}    
  \end{equation*}
  and so we rewrite it in block-notation as
  \begin{equation*}
      \partial \tilde{F}_t'(x)Y_t'= \left( \left(\begin{array}{ccc} 
      \partial F'_t(x) \partial_{X} Y_t & 0 & 0 \\
      \partial F_t'(x) & 0 & 0 \\
      0& 0 &0 
      \end{array} \right): \right).
  \end{equation*}
  Similarly,  we also get 
  \begin{equation*}
  \begin{aligned}
      \partial^{2} F_{t}(x) (Y'_{t}, Y'_{t})(y\otimes z)&=\partial^{2} F_{t}(x)(Y'_{t}(y) \otimes Y'_{t}(z))\\
      &=\partial^{2} F_{t}(x)((\partial_{X} Y_{t} (y_{1})+ y_{2})\otimes (\partial_{X} Y_{t} (z_{1})+ z_{2} )),
      \end{aligned}
      \end{equation*}
      which in block-notation is
    \begin{equation*}
    \partial^{2} F_{t}(x) (Y'_{t}, Y'_{t})=\left(\left(\begin{array}{ccc} 
      \partial^{2} F_{t}(x)(\partial_{X} Y_{t}, \partial_{X} Y_{t}) & \partial^{2} F_{t}(x)(\partial_{X} Y_{t}, \operatorname{Id}_{d_{Y}}) & 0 \\
      \partial^{2} F_{t}(x)(\operatorname{Id}_{d_{Y}}, \partial_{X} Y_{t}) & \partial^{2} F_{t}(x)(\operatorname{Id}_{d_{Y}}, \operatorname{Id}_{d_{Y}}) & 0 \\
      0& 0 &0 
      \end{array} \right): \right).
  \end{equation*}
  So in total, by applying \cref{cor:RIW_formula} specifically version \eqref{rem:eq:weak_RIW}, we see for any $\omega \in \Omega^{5}\coloneqq \Omega^{3} \cap \Omega^{4}$ that $(Z- \tilde{V}, Z', Z'', 0)(\omega) \in \mathscr{D}^{3 \alpha}_{(\mathbf{X}; M; N)(\omega)}$, where 
  \begin{align*}
    Z_t & =  F_{t}(Y_t),\nonumber\\
    \tilde{V}_{t} & =  \int_{0}^{t} \dot{F}_{r}(Y_{r}) + \partial F_r(Y_{r}) \dot{Y}_{r} dr + \int_{0}^{t}  \partial F'_{r}(Y_{r}) \partial_{X} Y_{r} +\frac{1}{2} \partial^{2} F_{r}(Y_{r}) (\partial_{X} Y_{r}, \partial_{X} Y_{r}) d[\mathbf{X}]_{r} \nonumber \\
    &\quad + \frac{1}{2}\int_{0}^{t} \partial^{2} F_{r}(Y_{r}) d\langle M \rangle_{r},  \nonumber\\
    Z'_{t} &= \left(\left( \begin{array}{c}
    F'(Y_{t})+ \partial F_{t}(Y_{t}) \partial_{X} Y_{t} \\
    \partial F_{t}(Y_{t}) \\
    0
    \end{array}\right) \cdot \right), \nonumber \\
    Z''_t & =  \left(\left(\begin{array}{ccc}
      \partial F_t (Y_{t}) \partial_{X}^{2} Y_{t} + F''_{t} (Y_{t})+ \partial F'_{t}(Y_{t}) \partial_{X} Y_{t}+ (\partial F'_{t}(Y_{t}) \partial_{X} Y_{t})^{\top} & (\partial F'_{t}(Y_{t}))^{\top}& 0 \\
      \partial F'_{t}(Y_{t})& 0 & 0 \\
      \partial F_{t}(Y_{t})& 0 & 0
      \end{array}\right): \right)\nonumber\\
      &\quad + \left(\left(\begin{array}{ccc}
       \partial^{2} F_{t}(Y_{t}) (\partial_{X} Y_{t}, \partial_{X} Y_{t}) & (\partial^{2} F_{t}(Y_{t}) (\operatorname{Id}_{d_{Y}}, \partial_{X} Y_{t}))^{\top} & 0\\
        \partial^{2} F_{t}(Y_{t}) (\operatorname{Id}_{d_{Y}}, \partial_{X} Y_{t}) & \partial^{2} F_{t}(Y_{t}) (\operatorname{Id}_{d_{Y}}, \operatorname{Id}_{d_{Y}}) & 0 \\
      0 & 0& 0
      \end{array}
      \right): \right),\nonumber
  \end{align*}
  in block-operator notation. Now applying \cref{prop:SCRP_to_RIP} yields that
  $(Z, Z', Z'', \dot{Z}; \tilde{M}, \tilde{N})$ as defined in the claim is an
  $\mathbf{X}$-scRSM. 
\end{proof}
  
  \begin{remark}
Clearly, \cref{theorem:RSIW_Theorem_1} also implies a rough stochastic Itô (short:rsI) formula for functions $F\in C^{3}$. Notably,  this rsI formula comes, in contrast to related results in \cite{friz_2021} and \cite{bugini2024}, without any growth assumptions on $F$. 

 \end{remark}

  \begin{example}
      We consider an example coming from  {\em rough stochastic differential equations} (RSDEs). Let $\mathbf{Z}\in 
      \mathscr{C}_{g}^{\alpha}([0,T]; \mathbb{R}^{d})$ and consider suitable 
      $$(f,f'):[0,T]\times  \mathbb{R}^{d_{X}}\to \mathcal{L}(\mathbb{R}^{d_{Z}}; \mathbb{R}^{d_{X}})\times \mathcal{L}((\mathbb{R}^{d_{Z}})^{\otimes 2}; \mathbb{R}^{d_{X}})$$ 
      such that there is a unique solution $X^{s, x; \mathbf{Z}}$ of the RDE
  \begin{equation*}
      \begin{aligned}
  X^{s, x; \mathbf{Z}}_{t} = x +\int_{s}^{t} \Big(f_{r}(X^{s, x; \mathbf{Z}}_{r}), Df_{r}(X^{s, x; \mathbf{Z}}_{r})f_r(X^{s, x; \mathbf{Z}}_{r})+ f'_{r}(X^{s, x; \mathbf{Z}}_{r})\Big)d\mathbf{Z}_{r}, 
    \end{aligned}
    \end{equation*}
which satisfies the following assumptions:
      \begin{assumption}
      \label{assumption:RSDE_flow}
      Set $\Gamma(\cdot)\coloneqq D(\cdot) f$ and $\Gamma'(\cdot)\coloneqq D(\cdot) f'$. We assume for $\phi_{t}(x)\coloneqq X^{0,x; \mathbf{Z}}_{t}$, that $\mathbb{R}^{d_{X}}\ni x \mapsto \phi_{t}(x)\in \mathbb{R}^{d_{X}}$ is a homeomorphism with 
      \begin{equation*}
      \begin{aligned}
        \Phi&\coloneqq \left(\phi , f \circ \phi, D\phi, (\Gamma f+ f') \circ \phi, D(f\circ \phi), D^2\phi , 0\right) \in \mathscr{D}^{3\alpha}_{\mathbf{Z}}\operatorname{Lip}^{3}_{x}(\mathbb{R}^{d_{X}}; \mathbb{R}^{d_{X}}), \\ 
              \Phi^{-1}&\coloneqq \left(\phi^{-1},  - \Gamma \phi^{-1}, D\phi^{-1}, (\Gamma^{2} 
        \phi^{-1}+ \Gamma' \phi^{-1})^{\top}, -D(\Gamma \phi^{-1}), D^{2} \phi^{-1}, 0\right)\in \mathscr{D}^{3 \alpha}_{\mathbf{Z}} \operatorname{Lip}^{3}_{x}(\mathbb{R}^{d_{X}}; \mathbb{R}^{d_{X}}).
    \end{aligned}
    \end{equation*}
      \end{assumption}
    In \cref{appendix:colllection_controlled_fields}, we provide conditions on $(f, f')$ to satisfy \cref{assumption:RSDE_flow}. 
    \begin{proposition}
    \label{prop:flow_transform}
        Suppose \cref{assumption:RSDE_flow} holds. Let $S$ be a continuous semimartingale of the form 
        \begin{equation*}
            S_{t} = S_{0} + \int_{0}^{t} \dot{S}_{s} ds + M_{t},
        \end{equation*}
        where $\dot{S}\in L^{\infty}([0,T]; \mathbb{R}^{d_{Y}})$ a.s.\ and $M \in \mathcal{M}^{c, \loc, 1}(\mathbb{R}^{d_{Y}})$.
        Then the following claims are equivalent:
    \begin{enumerate}
        \item \label{ex:RSDE:item_1}$\mathcal{Y}=(Y, f(Y), \Gamma f(Y) + f'(Y), \dot{S}; M , Df(Y) \bullet M) $ is  a $\mathbf{Z}$-scRSM. In particular, $Y$ solves the RSDE
        \begin{equation*}
        Y_{t}= Y_{0} + \int_{0}^{t} \left(f_{s}(Y_{s}), (\Gamma_{s} f_{s}(Y_{s})+ f'_{s}(Y_{s})\right) d\mathbf{Z}_{s}+ S_{t}.
        \end{equation*}
        \item \label{ex:RSDE:item_2} Define $\tilde{Y}_{t} \coloneqq \phi^{-1}_{t}(Y_{t})$. Then $\tilde{Y}$ solves the SDE.
    \begin{equation}
    \label{ex:RSDE:SDE}
        d\tilde{Y}_{t} =   D \phi^{-1}_{t}(\phi_{t}( \tilde{Y}_{t})) dS_{t} + \frac{1}{2} D^{2} \phi^{-1}_{t}( \phi_{t}(
        \tilde{Y}_{t})) {\langle \dot M  \rangle}_{t} dt; \quad \tilde{Y}_{0}= Y_{0}.
    \end{equation}
    \end{enumerate}
    \end{proposition}
    \begin{proof}
       The claim \eqref{ex:RSDE:item_1} $\Rightarrow$ \eqref{ex:RSDE:item_2} follows by using \cref{theorem:RSIW_Theorem_1} to calculate $\Phi^{-1}\circ \mathcal{Y}$. The converse claim  \eqref{ex:RSDE:item_2} $\Rightarrow$ \eqref{ex:RSDE:item_1} follows by calculating $\Phi \circ \tilde{\mathcal{Y}}$ using \cref{theorem:RSIW_Theorem_1}, where $\tilde{\mathcal{Y}}$ is given by \eqref{ex:RSDE:SDE}.
    \end{proof}
    We are particularly interested in the case 
    \begin{equation*}
        dS_{t}= b_{t}(Y_{t}) dt + \sigma_{t}(Y_{t}) dW_{t}, 
    \end{equation*}
    with $W$ being a Brownian motion and $b, \sigma$ being (progressively measurable) coefficients of suitable dimensions.
    Then by \cref{prop:flow_transform}, the existence and uniqueness of solutions $Y$ to 
    \begin{equation*}
        Y_{t}= Y_{0} + \int_{0}^{t} b_{s}(Y_{s}) ds + \int_{0}^{t} \sigma_{s}(Y_{s}) dW_{s}+\int_{0}^{t} (f, \Gamma f+ f')_{s}(Y_{s}) d\mathbf{Z}_{s}
    \end{equation*}
    (in the sense of \cref{def:Rough_Ito_process}) is equivalent to the existence and uniqueness of solution to the SDE
    \begin{equation*}
        d\tilde{Y}_{t}= \tilde{b}_{t}(\tilde{Y}_{t}) dt + \tilde{\sigma}_{t}(\tilde{Y}_{t})dW_{t}; \quad \tilde{Y}_{0}=Y_{0},
    \end{equation*}
    with coefficients
    \begin{equation*}
        \begin{aligned}
            \tilde{b}_{t}(y)&\coloneqq D \phi^{-1}_{t}\left(\phi_{t}(y) \right) b_{t}(y) + \frac{1}{2} D^{2} \phi^{-1}_{t}\left(\phi_{t}(y)\right)( \sigma_{t}(y), \sigma_{t}(y)),\\
            \tilde{\sigma}_{t}(y)&\coloneqq D \phi^{-1}_{t}(\phi_{t}(\tilde{Y}_{t})) \sigma_{t}(y).
        \end{aligned}
    \end{equation*}
Assuming that $b_{t}(\omega,\cdot)$ and $\sigma_{t}(\omega,\cdot)$ are Lipschitz continuous for a.e.\ $(t,\omega)\in[0,T]\times\Omega$, the well-posedness of the above SDE can be readily established by exploiting the fact that $\Lip^1$ is an algebra, together with the spatial regularity of $\phi$ and $\phi^{-1}$ as specified in \cref{assumption:RSDE_flow}.

In \cite{CDFO13}, the authors constructed solutions to RSDEs with autonomous coefficients by combining the above flow-transform approach with a limiting procedure for RDEs driven by geometric rough paths. In contrast, our argument is intrinsic, in the sense that it does not rely on any limiting procedure. This feature also enables us to treat non-autonomous vector fields in the above RSDE in a straightforward manner, provided they are controlled with respect to the fixed reference path $\mathbf{Z}$; see \cref{appendix:section:On RDE-flows for non-autonomous vector fields} for details.
\end{example}

    \subsubsection{The rough stochastic Itô-Wentzell formula for martingale functionals }
    \label{subsubsection:RSIW_2}

We now additionally establish a composition rule for parameter-dependent scRSMs evaluated at scRSMs. Recall that the proof of the rsIW formula    \cref{theorem:RSIW_Theorem_1}   for controlled fields relied on applying the deterministic space–time controlled calculus from    \cref{section:Introduction to controlled fields}  to a suitably lifted stochastic process. In contrast, the proof of the corresponding composition rule for martingale functionals is fundamentally different. It is based on the (discrete-time) Burkholder–Davis–Gundy inequality, as we must carefully track the correlation between the martingale functional and the scRSM at which it is evaluated.

  \begin{lemma}
  [Exercise 3.1.5 in \cite{kunita_1997}]
  \label{lemma:regularity_martingale_funct}
      Let $(W_{t})_{t\in [0,T]}$ be a $d_{W}$-dimensional Brownian Motion. For $k \in \mathbb{N}_{0}$ and $\delta>0$,  consider  a continuous $(\mathfrak{F}_{t})$-adapted process
      \begin{equation*}
          \beta : [0,T]\times \Omega \to \operatorname{Lip}^{k+\delta}(\mathbb{R}^{d}; \mathbb{R}^{d_{G} \times d_{W}}),
      \end{equation*}
      satisfying
      \begin{equation*}
          \mathbb{E} \left [  \int_{0}^{T} \Vert \beta_{r} \Vert_{\operatorname{Lip}^{k+\delta}}^{2} dr \right]< \infty.
      \end{equation*}
Then, the random field $$G_{t}(x)\coloneqq \int_{0}^{t} \beta_{r}(x) dW_{r}$$
has a $\operatorname{Lip}^{k+\epsilon}(\mathbb{R}^{d}; \mathbb{R}^{d_{G} \times d_{W}})$-valued modification for any $\epsilon \in (0, \delta)$, which is still denoted by $G_t(x)$. Moreover,  for any $l\in \mathbb{N}^{d}_{0}$ with $|l|\leq k$ it holds almost surely
      \begin{equation*}
          D^{l} G_{t}(x)=\int_{0}^{t} D^{l} \beta_{s}(x) dW_{s}.
      \end{equation*}
  \end{lemma}
  \begin{theorem}
  \label{theorem:RSIW_II}
  Let $\beta$ and $G$ be as in \cref{lemma:regularity_martingale_funct} with $k=3$  and $\delta>0$. Further,  let $\mathcal{Y}=(Y, \partial_{X}Y, \partial_{X}^{2}Y, \dot{Y}; 
  N, M)$ be an $\mathbf{X}$-scRSM.
  Then $G \circ Y$ satisfies
  \begin{eqnarray*}
    \label{thm:eq:RSIW_Martingale}
    G_t (Y_t) - G_s (Y_s) & = &  \int_s^t \beta_r (Y_r) dW_{r}+ \int_{s}^{t} DG_{r} (Y_{r}) \dot{Y}_{r} dr  +\int_{s}^{t} DG_{r}(Y_{r}) dM_{r} \nonumber \\
    &&+ \int_s^t (D G_r (Y_r) \partial_{X}Y_{r}, D^2 G_r (Y_r) (\partial_{X}Y_{r}, \partial_{X}Y_{r}) + D
    G_r (Y_r) \partial_{X}^{2}Y_{r}) d\mathbf{X}_{r} \nonumber \\
    &&+ \left \langle \int_{0}^{\cdot} D \beta_r (Y_r) dW_{r}, M \right \rangle_{s, t} + \frac{1}{2} \int_{s}^{t} D^2 G_{r}( Y_r) d \langle M \rangle_{r} \nonumber \\
    &  & + \frac{1}{2} \int_s^t D^2 G_{r}(Y_r) \left(\partial_{X}Y_{r}, \partial_{X}Y_{r}\right)
    d [\mathbf{X}]_r, \nonumber
    \end{eqnarray*}
  where the $d\mathbf{X}$-integral  is the limit in probability of corresponding corrected Riemann sums, i.e.,  it holds for $Z : = D G  (Y ) \partial_{X} Y $, $Z' : = D^2 G
  (Y ) (\partial_{X}Y , \partial_{X}Y) + D G  (Y ) \partial_{X}^{2} Y $, any $(s, t) \in
  \Delta_T$ and any sequence of deterministic partitions $(\mathcal{P}^{n})_{n \in \mathbb{N}}$ of $[s, t]$ of vanishing mesh-size:
  \begin{equation}
  \label{generalized_RSI}
     \sum_{[u, v] \in \mathcal{P}^{n}} Z_u \delta X_{u, v} + Z_u' \mathbb{X}_{u,
     v} \stackrel{\mathbb{P}}{\longrightarrow} \int_s^t (Z, Z')_r d \mathbf{X}_{r}
  \end{equation}
  as $n \to \infty$. 
     \end{theorem}
     \begin{proof} We follow a similar argument as in the proof of Theorem 3.3.1 in \cite{kunita_1997}.     Without loss of generality, we take $s=0, t=1$ and assume $M, N$ being $L^{2}$-martingales (otherwise we proceed with a suitable stopping procedure as in the proof of \cref{prop:SCRP_to_RIP}, noting the consistency of  RSI-integrals established  in \cref{lemma:Stopped_RSI}). Denote $V=\int_{0}^{\cdot} \dot{Y}_{t} dt$, and  let $\mathcal{P}^{N}=(t_{i}^N)_{i=0}^{2^{N}}$  be the dyadic partition  of $[0,1]$.  We decompose
     \begin{equation*}
     \begin{aligned}
         G_{1}(Y_{1})- G_{0}(Y_{0})&= \sum_{i=0}^{2^N-1} G_{t_{i+1}^{N}}(Y_{t_{i+1}^{N}})- G_{t_{i+1}^{N}}(Y_{t_{i}^{N}})+ \sum_{i=0}^{2^N-1} G_{t_{i+1}^{N}}(Y_{t_{i}^{N}})- G_{t_{i}^{N}}(Y_{t_{i}^{N}})\\
         &\eqqcolon I^{N}+ II^{N}. 
     \end{aligned}
     \end{equation*}
     By Section 3.2 in \cite{kunita_1997} , we see that $II^{N}$ converges in probability to $\int_{0}^{1} \beta_{r}(Y_{r}) dW_{r}$. 
     
     Now we deal with  $I^{N}$.   By \cref{lemma:regularity_martingale_funct}, we have
     \begin{eqnarray*}
         I^{N}&\coloneqq&  \sum_{i=0}^{2^{N}-1} DG_{t_{i+1}^{N}}(Y_{t_{i}^{N}}) \delta Y_{t_{i}^{N}, t_{i+1}^{N}} +\frac{1}{2} D^{2} G_{t_{i+1}^{N}}(Y_{t_{i}^{N}}) (\delta Y_{t_{i}^{N}, t_{i+1}^{N}})^{\otimes 2} +\operatorname{Rem.} \nonumber \\
         &=&\left(\sum_{i=0}^{2^{N}-1} DG_{t_{i}^{N}}(Y_{t_{i}^{N}}) \delta ( Y-M- V)_{t_{i}^{N}, t_{i+1}^{N}} +\frac{1}{2} D^{2} G_{t_{i}^{N}}(Y_{t_{i}^{N}}) (\delta Y_{t_{i}^{N}, t_{i+1}^{N}})^{\otimes 2}\right) \nonumber \\
         && + \left(\sum_{i=0}^{2^{N}-1} \delta (DG_{\cdot}(Y_{t_{i}^{N}}))_{t_{i}^{N}, t_{i+1}^{N} }\delta Y_{t_{i}^{N}, t_{i+1}^{N}} \right)+ \left(\sum_{i=0}^{2^{N}-1} DG_{t_{i}^{N}}(Y_{t_{i}^{N}}) \delta M_{t_{i}^{N}, t_{i+1}^{N}} \right)\nonumber\\
         &&+ \left(\sum_{i=0}^{2^{N}-1} DG_{t_{i}^{N}}(Y_{t_{i}^{N}}) \delta V_{t_{i}^{N}, t_{i+1}^{N}} \right)+\operatorname{Rem.} \nonumber \\
         &\eqqcolon& III^{N}+ IV^{N} + V^{N} + VI^{N} + \operatorname{Rem.}, \nonumber
     \end{eqnarray*}
     where $|\operatorname{Rem.}|=\mathcal{O}(2^{N(1-3\alpha)})$ almost surely.   As $N\to\infty$, clearly  $V^{N}$ converges in probability to the Itô integral $\int_{0}^{1} DG_{r}(Y_{r}) dM_{r}$, and $VI^{N}$ converges a.s. to the Riemann-Stieltjes integral $\int_{0}^{1} DG_{r}(Y_{r}) dV_{r}$. 
 
 For the  terms $III^N$ and $IV^n$, we  further decompose them.  First, note that a.s.
     \begin{equation*}
     \begin{aligned}
         \delta(Y-M-V)_{s,t}= \int_{s}^{t} (\partial_{X} Y, \partial^{2}_{X} Y)_{r} d\mathbf{X}_{r} &\stackrel{3\alpha}{=} (\partial_{X} Y-N)_{s} \delta X_{s,t} + \partial_{X}^{2} Y_{s} \mathbb{X}_{s,t} + \int_{s}^{t} N_{r} dX_{r}\\
         &= \partial_{X} Y_{s} \delta X_{s,t} + \partial^{2}_{X} Y_{s} \mathbb{X}_{s,t} + \Pi(N;X)_{s,t}\\
         &= \partial_{X} Y_{s} \delta X_{s,t} + \partial^{2}_{X} Y_{s} \mathbb{X}_{s,t} + \delta N_{s,t} \delta X_{s,t} - \Pi(X; N)_{s,t} \\&\stackrel{2\alpha}{=}\partial_{X} Y_{s} \delta X_{s,t}.
     \end{aligned}    
     \end{equation*}
     Then  $IV^{N}$ is  further decomposed as
     \begin{equation*}
     \begin{aligned}
     IV^{N}&=\left(\sum_{i=0}^{2^{N}-1} \delta (DG_{\cdot}(Y_{t_{i}^{N}}))_{t_{i}^{N}, t_{i+1}^{N} } \partial_{X} Y_{t_{i}^{N}} \delta X_{t_{i}^{N}, t_{i+1}^{N}} \right)+\left(\sum_{i=0}^{2^{N}-1} \delta (DG_{\cdot}(Y_{t_{i}^{N}}))_{t_{i}^{N}, t_{i+1}^{N} }\delta M_{t_{i}^{N}, t_{i+1}^{N}} \right)\\
     &\quad +\left(\sum_{i=0}^{2^{N}-1} \delta (DG_{\cdot}(Y_{t_{i}^{N}}))_{t_{i}^{N}, t_{i+1}^{N} }\delta V_{t_{i}^{N}, t_{i+1}^{N}} \right)+\operatorname{Rem.}, 
     \end{aligned}
     \end{equation*}
     where again $|\operatorname{Rem.}|=\mathcal{O}(2^{N(1-3\alpha)})$ almost surely. The first sum vanishes in probability by the discrete Burkholder–Davis–Gundy inequality, and the third sum vanishes a.s. by standard arguments for Riemann-Stieltjes integrals. For the second sum, by \cite[Corollary 2.2.19]{kunita_1997}  we get, as $N\to\infty$,
     \begin{equation*}
      \sum_{i=0}^{2^{N}-1} \delta (DG_{\cdot}(Y_{t_{i}^{N}}))_{t_{i}^{N}, t_{i+1}^{N} }\delta M_{t_{i}^{N}, t_{i+1}^{N}} \stackrel{\mathbb{P}}{\longrightarrow} \left \langle \int_{0}^{\cdot} D\beta_{r}(Y_{r}) dW_{r} , M \right \rangle_{0, 1}.
     \end{equation*}

     Now, it remains to  consider $III^{N}$. Denote
     \begin{equation*}
         A_{s,t} \coloneqq DG_{s}(Y_{s}) \delta(Y-M-V)_{s,t}+  \frac{1}{2}D^{2} G_{s}(Y_{s})( \delta Y_{s,t})^{\otimes 2}
     \end{equation*}
     and
     \begin{equation*}
     \tilde{A}_{s,t}\coloneqq DG_{s}(Y_{s}) (\partial_{X} Y_{s} \delta X_{s,t}+ \partial_{X}^{2} Y_{s} \mathbb{X}_{s,t})+  D^{2} G_{s}(Y_{s})( \partial_{X} Y_{s}, \partial_{X} Y_{s}) \mathbb{X}_{s,t}.
     \end{equation*}
     Then, by the discrete Burkholder–Davis–Gundy inequality, we get
     \begin{equation*}
         \sum_{i=0}^{2^{N}-1}\Big|A_{t_{i}^{N}, t_{i+1}^{N}}- \tilde{A}_{t_{i}^{N}, t_{i+1}^{N}}- \frac{1}{2}D^2 G_{t_{i}^{N}}(Y_{t_{i}^{N}}) 
         \left((\partial_{X}Y_{t_{i}^{N}}, \partial_{X}Y_{t_{i}^{N}}) \delta [\mathbf{X}]_{t_{i}^{N}, t_{i+1}^{N}}+ (\delta M_{t_{i}^{N}, t_{i+1}^{N}})^{\otimes 2}\right) \Big| \stackrel{\mathbb{P}}{\longrightarrow}
         0,
     \end{equation*}
     as $N \to \infty$.
Moreover,  for any $s,u,t \in [0,T]$ with $s< u <t$,  it holds a.s.

     \begin{eqnarray*}
     \delta \tilde{A}_{s,u,t}&=&- \delta(DG_{\cdot} (Y_{\cdot})\partial_{X} Y_{\cdot})_{s,u} \delta X_{u,t} - \delta (DG_{\cdot}(Y_{\cdot}) \partial_{X}^{2} Y_{\cdot})_{s,u} \mathbb{X}_{u,t} \nonumber \\
     &&+ D G_{s}(Y_{s}) \partial_{X}^{2} Y_{s} \delta X_{s,u} \otimes \delta X_{u,t} - \delta (D^{2} G_{\cdot} (Y_{\cdot})(\partial_{X} Y_{\cdot}, \partial_{X} Y_{\cdot}))_{s,u} \mathbb{X}_{u,t} \nonumber \\
     &&+ D^{2} G_{s} (Y_{s})(\partial_{X} Y_{s}, \partial_{X} Y_{s})\delta X_{s,u} \otimes \delta X_{u,t} \nonumber \\
     &\stackrel{3\alpha}{=}&- \delta(DG_{\cdot} (Y_{\cdot})\partial_{X} Y_{\cdot})_{s,u} \delta X_{u,t} \nonumber + D G_{s}(Y_{s}) \partial_{X}^{2} Y_{s} \delta X_{s,u} \otimes \delta X_{u,t} \nonumber \\
     &&+ D^{2} G_{s} (Y_{s})(\partial_{X} Y_{s}, \partial_{X} Y_{s})\delta X_{s,u} \otimes \delta X_{u,t} \nonumber 
     \end{eqnarray*}
     Noting that a.s.
     \begin{eqnarray*}
        - \delta(DG_{\cdot} (Y_{\cdot})\partial_{X} Y_{\cdot})_{s,u} &=& (DG_{s}(Y_{s})- DG_{u}(Y_{s}))\partial_{X}Y_{s}+ (DG_{u}(Y_{s})-DG_{u}(Y_{u}))\partial_{X}Y_{s}\\
        &&- (DG_{u}(Y_{u}) \delta (\partial_{X} Y)_{s,u}\\
        &\stackrel{2\alpha}{=}&
        (DG_{s}(Y_{s})- DG_{u}(Y_{s}))\partial_{X}Y_{s}+D^{2} G_{s}(Y_{s})(\partial_{X} Y_{s}, \partial_{X} Y_{s}) \delta X_{u, s}\\
        &&+D^{2} G_{s}(Y_{s})(\operatorname{Id}, \partial_{X} Y_{s}) \delta M_{u, s}+ DG_{u}(Y_{u}) \delta N_{u,s}\\
        &&+ (DG_{u}(Y_{u}) \partial_{X}^{2} Y_{s} \delta X_{u,s}
     \end{eqnarray*}   
    So in total we get a.s.
     \begin{eqnarray*}
     \delta \tilde{A}_{s,u,t}
     &\stackrel{3\alpha}{=}&
     -D^{2} G_{s}(Y_{s}) (\operatorname{Id}, \partial_{X} Y_{s}) \delta M_{s,u} \otimes \delta X_{u,t} - (DG_{u}(Y_{s})- DG_{s}(Y_{s}))\partial_{X} Y_{s} \delta X_{u,t}\nonumber \\
     &&- DG_{s}(Y_{s}) \delta N_{s,u} \delta X_{u,t}\nonumber \\
     &\eqqcolon & B_{s,u, t}+ C_{s, u, t}+ D_{s,u,t} \nonumber \end{eqnarray*}
  
By the above equality, we see that a.s.
\begin{equation*}
    \bigg| \sum_{i=0}^{2^{N}-1} \tilde{A}_{t_{i}^{N}, t_{i+1}^{N}}- \sum_{i=0}^{2^{N+1}-1} \tilde{A}_{t_{i}^{N+1}, t_{i+1}^{N+1}} \bigg| \lesssim 2^{N(1-3\alpha)}+ \sum_{E\in \{B, C, D\}}\left|\sum_{i=0}^{2^{N}-1} \delta E_{t_{i}^{N}, t_{2i+1}^{N+1}, t_{i+1}^{N}}\right|.
\end{equation*}
Now again by the discrete Burkholder–Davis–Gundy inequality we see that for $E\in \{B, C, D \}$ 
\begin{equation*}
\sum_{N\geq 0} \left \Vert \sum_{i=0}^{2^{N}-1} \delta E_{t_{i}^{N}, t_{2i+1}^{N+1}, t_{i+1}^{N}} \right \Vert_{L^{2}} \lesssim \sum_{N \geq 0} 2^{-N\alpha}< \infty. 
\end{equation*}
In total this yields $| III^{N'}-III^{N}| \to 0$ in probability as $N, N' \to \infty$. Thus $III^N$ is a Cauchy-sequence in probability and by completeness the limit as in \eqref{generalized_RSI} exists. 
     \end{proof}
\begin{remark}
Note that \cref{theorem:RSIW_II} could be extended to general, suitably regular martingale-fields $(M_{t}(x))_{t\in [0,T], x\in \mathbb{R}^{d}}$ as considered throughout \cite{kunita_1997}. 
\end{remark}
Combining \cref{theorem:RSIW_Theorem_1} and \cref{theorem:RSIW_II} yields the following composition rule for suitable parameter-dependent scRSMs.
\begin{theorem}
    \label{theorem:total_RSIW_formula}
Let $\mathbf{X} \in \mathscr{C}^{\alpha, 1}([0,T]; \mathbb{R}^{d_{X}})$ and  $\mathcal{Y}=(Y, \partial_{X} Y, \partial_{X}^{2} Y, \dot{Y}; M, N)$ be an
  $\mathbf{X}$-scRSM.
  Let $\mathcal{F} = (F , F', \partial F, F'', \partial
  F', \partial^2 F, \dot{F })\in L^0_{\operatorname{ad}}\left (\Omega;  
\mathscr{D}^{3\alpha}_{\mathbf{X}}{\Lip}^{3}_{x}(\mathbb{R}^{d_{Y}}; \mathbb{R}^{d_{H}})\right)$ and $(W_{t})_{t\in [0,T]}$, $ \beta : [0,T]\times \Omega \times \mathbb{R}^{d_{Y}}\to \mathbb{R}^{d_{H} \times d_{W}}$ and $G$ as in \cref{lemma:regularity_martingale_funct} with $k =3$ and $\delta>0$. Then 
  for 
  \begin{equation*}
      H_{t}(x)\coloneqq F_{t}(x)+G_{t}(x)=\int_{0}^{t} \dot{F}_{s}(x) ds+ \int_{0}^{t} (F'_{s}(x), F''_{s}(x)) d\mathbf{X}_{s} + \int_{0}^{t} \beta_{s}(x) dW_{s} 
  \end{equation*}
  it holds a.s.
\begin{eqnarray*}
    H_{t}(Y_{t})&=& \int_{0}^{t} \beta_s (Y_s) dW_{s}+ \int_{0}^{t} DH_{s}(Y_{s})  \dot{Y}_{s} + \dot{F}_{s}(Y_{s}) ds + \int_0^t DH_{s} (Y_{s}) 
    dM_{s}  \nonumber \\
    &&+\int_{0}^{t} (F'_{s}(Y_{s}) + D H_{s}(Y_{s}) \partial_{X} Y_{s}, \partial_{X}^{2} Z_{s})  d
    \mathbf{X}_{s} +\frac{1}{2} \int_{0}^{t} D^{2} H_{s}(Y_{s}) d\langle M \rangle_{s} \nonumber\\
    &&+\int_{0}^{t} \left(  D F'_{s}(Y_{s})  \partial_{X}Y_{s} + \frac{1}{2} D^2
    H_{s}(Y_{s}) (\partial_{X}Y_{s}, \partial_{X} Y_{s} )  \right) d [\mathbf{X}]_{s} \nonumber\\
    &&+ \left \langle \int_{0}^{\cdot} D \beta_s (Y_s) dW_{s}, M \right \rangle_{0, t} \nonumber, 
\end{eqnarray*}
where 
\begin{equation*}
  \partial_{X}^{2} Z_t \coloneqq  D H_{t}(Y_{t})  {\partial^{2}_{X} Y_{t}}  + F''_t(Y_{t})  +
    \partial F'_t(Y_{t}) \partial_{X} Y_{t}  + (\partial F'_t(Y_{t}) \partial_{X} Y_{t})^T + D^{2} H_t(Y_{t}) (\partial_{X} Y_{t},
    \partial_{X} Y_{t}).  
\end{equation*}
\end{theorem}
\begin{remark}
One interesting example of such a rough stochastic functional considered in \cref{theorem:total_RSIW_formula} is the solution flow $H_{t}(x) \coloneqq X^{0, x}_{t}$ to an RSDE: 
    \begin{equation*}
        dX^{0,x}_{t} = b_{t}(X^{0,x}_{t}) dt + \sigma_{t}(X^{0,x}_{t}) dW_{t} + (f, f')_{t}(X_{t}^{0,x}) d\mathbf{X}_{t}; \quad X^{0,x}_{0}=x,
    \end{equation*}
    as introduced in \cite{friz_2021}. As already apparent from \cref{example:RDE_flow_transport} and classical works on stochastic flows as  \cite{kunita_1997} this would provide a powerful toolbox for the study of (S)PDEs with rough perturbations. While the authors of  \cite{BuginiFrizStannat2024}  studied the regularity of such RSDE flows in an $L^{p}$-sense (see \cref{subsection:IAG_general_rf} for an application), our \cref{theorem:total_RSIW_formula} requires $L^{0}$-regularity. However, an analysis of  $L^{0}$-regularity lies beyond the scope of the present work. 
\end{remark}

\section{Applications to stochastic analysis}
\label{section:Applications to (Rough) Stochastic Analysis}
We fix $\alpha\in (\nicefrac{1}{3}, \nicefrac{1}{2})$ throughout this section.
\subsection{Itô-Alekseev-Gröbner formula in Skorokhod form}

As applications of the rough (stochastic) calculus developed in \cref{section:Introduction to controlled fields,section:On Rough Stochastic Calculus}, we revisit the Itô-Alekseev-Gröbner (IAG) formula of \cite{HHJM24}, sharpening their moment assumptions for the Itô characteristics of the perturbation process ($Y$ below). 
  Let $\mathbf{Z} = (Z, \mathbb{Z}) \in
    \mathscr{C}^{\alpha, 1+\alpha} ([0, T] ; \mathbb{R}^{d_Z})$ be a deterministic rough path and for any $(t, x) \in [0, T] \times
    \mathbb{R}^{d_X}$ denote by $(X_s^{t, x; \mathbf{Z}})_{s
    \in [t, T]}$ the unique solution of the RDE
    \begin{equation*}
     dX^{t,x; \mathbf{Z}}_{s} =
       \mu (X_s^{t, x; \mathbf{Z}}) ds+ \sigma (X^{t, x; \mathbf{Z}}_s) d\mathbf{Z}_{s}; \quad X^{t, x; \mathbf{Z}}_{t}=x
    \end{equation*}
    with $\sigma \in \Lip^{5}(\mathbb{R}^{d_{X}}; \mathbb{R}^{d_{X}\times d_{Z}})$ and $\mu \in \Lip^{3}(\mathbb{R}^{d_{X}}; \mathbb{R}^{d_{X}})$. 
    Note that the randomization $\mathbf{Z}\rightsquigarrow \mathbf{W}^{\text{It\^o}}(\omega)=(W(\omega), \mathbb{W}^{\text{It\^o}}(\omega))$ of this RDE, denoted by $\bar X$, solves the corresponding Itô-SDE. In particular,
    $$
          \overline{F}_t (x)(\omega) := f ( \overline{ X}^{t,x}_T(\omega) )= F_{t}^{\mathbf{Z}}(x)\bigg|_{\mathbf{Z}=\mathbf{W}^{\text{It\^o}}(\omega)} \coloneqq f(X^{t, x; \mathbf{Z}}_{T})\bigg|_{\mathbf{Z}=\mathbf{W}^{\text{It\^o}}(\omega)}
    $$
    is exactly the backward-adapted random field considered in \cite{HHJM24}, see also \cite{DELMORAL2022197}. Here and below a bar indicates randomization: for any (jointly-measurable) random field with rough path dependence $G : [0,T]\times \Omega \times 
\mathscr{C }^{0,\alpha, 1} ([0,T]; \mathbb{R}^{d_{W}}) \to V$ we introduce the short-hand
notation $\bar{G} : [0,T]\times \Omega \rightarrow V$ by $\bar{G}_{t} (\omega) := G_{t}(\omega, \mathbf{W}(\omega))$ for $\mathbb{P}$-a.e. $\omega\in \Omega$, where $\mathbf W(\omega):=\mathbf{W}^{\text{It\^o}}(\omega)$ is the Itô lift of the Brownian motion $W$.
Also, throughout this section, 
    $(Y_t)_{t \in [0, T]}$ denotes a $d_{X}$-dim. It{\^o} process with dynamics
    \begin{equation}
    \label{eq:Ito_process} dY_t = b_t dt+\beta_t dW_t, 
    \end{equation}
    where $W$ is some standard $d_W$-dimensional Brownian Motion and 
    $(b,\beta)$ are progressively measurable, locally bounded processes.\footnote{This is a mild condition that covers in particular continuous adapted  $b, \beta$, without moment assumptions.
   That said, further relaxation to a.s. local integrability in time, in which case our H\"older setting is no more appropriate and one has to to resort to a $p$-variation setting.}
    As before, e.g. \cref{ex:RDE_flows}, write  $\Gamma(\cdot) \equiv  D(\cdot)\sigma$.
  \begin{proposition}
    \label{cor:RSAG_formula}
    For any $f \in C^{3}(\mathbb{R}^{d_{X}}; \mathbb{R})$ we define the rough field $F_t^{\mathbf{Z}} (x) := f (X^{t, x,
    \mathbf{Z}}_T)$. Then it holds almost surely 
    \begin{equation}
    \label{eq:RSIAG}
    \begin{aligned}
         F^{\mathbf{Z}}_t (Y _t) -
         F^{\mathbf{Z}}_s (Y_s) = & \int_s^t D
         F^{\mathbf{Z}}_r (Y_r)  (b_r - \mu (Y_r))
         \hspace{0.17em} dr+ \int_s^t D F^{\mathbf{Z}}_r (Y_r) 
         \hspace{0.17em} \beta_r dW_r \\
         &- \int_s^t (\Gamma F_{r}^{\mathbf{Z}}(Y_{r}), -(\Gamma^{2} F_{r}^{\mathbf{Z}}(Y_{r}))^{\top}) d \mathbf{Z}_r \\
         &+ \frac{1}{2} \int_s^t \hspace{0.17em} D^2
         F^{\mathbf{Z}}_r (Y_r)\left( (\beta_{r}, \beta_{r})-(\sigma(Y_{r}), \sigma(Y_{r}))\dot{[\mathbf Z]}_r)\right) \hspace{0.17em} dr.
         \end{aligned}
       \end{equation} 
    for any $(s, t) \in \Delta_T$. Moreover, all terms in the above are jointly measurable in $(\omega, \mathbf{Z})$, hence admit a (measurable) 
    randomization $\mathbf{Z}\rightsquigarrow \mathbf{W}^{\text{It\^o}}(\omega)$.
     \end{proposition}

    \begin{proof} Recall from \cref{theorem:rough_transport_existence_uniqueness} that (note, that we correct for not assuming $\mathbf{Z}$ to be geometric here)
    \begin{equation*}
        (F^\mathbf{Z}, -\Gamma F^{\mathbf{Z}}, DF^{\mathbf{Z}}, (\Gamma^{2} F^{\mathbf{Z}})^{\top}, -D(\Gamma F^{\mathbf{Z}}), D^{2} F^{\mathbf{Z}}, - (D F^{\mathbf{Z}}) \mu - \frac{1}{2} (D^{2} F^{\mathbf{Z}})(\sigma, \sigma)\dot{[\mathbf Z]})\in \mathscr{D}^{3\alpha}_{\mathbf{Z}} \operatorname{Lip}^{3}_{x}.
    \end{equation*}
    Further we see that $Y$ is a $\mathbf{Z}$-scRSM only consisting
    of the local Martingale part $M = \int_{0} \beta_{r} dW_{r} \in \mathcal{M}^{c, \loc, 1}$ and drift $V=\int_{0} b_{r} dr$. 
    Crucially this also implies that $(\partial_{X} Y, \partial_{X}^{2} Y) \equiv 0$. The first claim then follows by applying \cref{theorem:RSIW_Theorem_1}. The statement on joint measurability 
    follows from measurable selection arguments, essentially  \cite[Prop. 1]{stricker_yor_78} with the remark that all stochastic and rough stochastic integrals in the above are limits in probability of appropriate Riemann-Stieltjes sums (this also holds for the $\int ... \beta_r dW_r $ integral, by considering the local martingale integrator $\beta \bullet W$), and also that RDE solutions depend measurably on $\mathbf{Z}$.
\end{proof}
 
     Upon randomization, the left-hand side of \eqref{eq:RSIAG} is exactly the ``left-hand side'' of the IAG formula obtained by \cite{HHJM24} and that we are interested in. 
     In some pragmatic sense this already yields the IAG formula where all (Lebesgue, stochastic, rough stochastic) integrals on 
     the right-hand side understood in a $(\int ... )|_{\mathbf{Z}\rightsquigarrow \mathbf{W}^{\text{It\^o}}(\omega)}$ substitution sense, 
     similar to the {\em substitution approach} in anticipating anticipating stochastic calculus (e.g. \cite[Section~3.2.4]{nualart_2006}). Unfortunately, 
     this approach is not justifiable here, and indeed would  (if formally applied) not give the correct answer. (To wit, \cite[Theorem 3.2.8]{nualart_2006} would produce correction terms not seen in the IAG formula.) 
     The problem in identifying the randomized integrals above clearly becomes more manageable in case of independent randomization, see e.g. \cite{friz_randomisation_2025-1} and references therein.   
     The difficulty in the present situation is that the randomization is fully correlated. To appreciate what can go wrong in such  situation, consider the bracket between a deterministic path $Z$ and a standard Brownian motion $W$; clearly $[Z,W] \equiv 0$ (which earns $Z$ the property of  being a weak Dirichlet process; this fact is crucial to the above rough stochastic calculus, see \cref{lemma:ibp}). Trivially, any randomization  $([Z,W]) |_{{Z}\rightsquigarrow W(\omega)}$ is still zero, in contrast to the classical fact $[W,W]_t = \mathrm{Id}\times t$.

We now give a direct proof of the IAG formula, which implies that the randomization of the above stochastic and rough stochastic integral {\em together} coincides with a Skorokhod integral. In view of the above
discussion a direct (correlated) randomization is ill-suited to this end, and the proof relies on employing local independence properties, where one can use efficiently known facts on the connection of rough and stochastic analysis. We have
\begin{theorem}
\label{theorem:IAG}
Let  $f \in \mathrm{Lip}^3, \mu \in \mathrm{Lip}^3, \sigma \in \mathrm{Lip}^5$ (with dimensions as in \cref{cor:RSAG_formula}) Assume further
$\beta
\beta^{\perp}, b \in L^{p} ([0, T] \times \Omega, \mathrm{Leb} \otimes
\mathbb{P})$ for some $p>1$.  Then, it holds for $F_t^{\mathbf Z}(x):=f(X_T^{t,x;\mathbf Z})$ and all $(s,t) \in \Delta _T$
  \begin{eqnarray*}
    \overline{F}_{t}(Y_{t})-\overline{F}_s(Y_{s})& = & \int_s^t D \overline{F}_r
    (Y_r) (b_r - \mu (Y_r)) dr
    \nonumber
    + \int_s^t D\overline{F}_r (Y_r)
    (\beta_r - \sigma (Y_r)) \diamond dW_{r}\nonumber\\
    &  & + \frac{1}{2} \int_s^t D^{2}\overline{F_r} (Y_r)\left( ( \beta_{r}, \beta_{r}) -(\sigma(Y_{r}), \sigma(Y_{r})) \right) dr , 
  \end{eqnarray*}
  almost surely. Here ``$\diamond$'' denotes Skorokhod integration w.r.t. $W$.
  \end{theorem}
  \begin{proof}
  (The assumptions ensure that the composition of $f$ with the RDE flow is in $\mathscr{D}^{3 \alpha}_{\mathbf{W}}\mathrm{Lip}^3$). W.l.o.g. we take $s = 0$ and $t = T$.   Let $\pi=\pi^N = (t_i)_{i= 0}^N $ be a sequence of deterministic partitions of $[0, T]$ with locally vanishing mesh. We decompose and apply the (rough) It\^o formula on small intervals to get
\begin{eqnarray}\label{e:decomp}
  F_T^{\mathbf{Z}} (Y_T) - F_0^{\mathbf{Z}} (Y_0) & = & \sum_{i=0}^{N}
  F^{\mathbf{Z}}_{t_{i + 1}} (Y_{t_{i + 1}}) - F^{\mathbf{Z}}_{t_i}
  (Y_{t_i})\nonumber\\
  & = & \sum_{i=0}^{N} F^{\mathbf{Z}}_{t_{i + 1}} (Y_{t_{i + 1}}) -
  F^{\mathbf{Z}}_{t_{i + 1}} (X_{t_{i + 1}}^{t_i, Y_{t_i} ; \mathbf{Z}})\\
  &=:&  \sum_{i=0}^{N} \Delta_i^{\mathbf Z} - \tilde{\Delta}_i^{\mathbf Z},\nonumber
\end{eqnarray}
where
\begin{eqnarray*}
  \Delta_i^{\mathbf{Z}} & : = & F^{\mathbf{Z}}_{t_{i + 1}} (Y_{t_{i + 1}}) -
  F^{\mathbf{Z}}_{t_{i + 1}} (Y_{t_i})\\
  & = &  \int_{t_i}^{t_{i + 1}} D
  F^{\mathbf{Z}}_{t_{i + 1}} (Y_r)  \hspace{0.17em} \beta_r dW_r + \int_{t_i}^{t_{i+1}}  L^{\beta, b}_r F^{\mathbf{Z}}_{t_{i + 1}} (Y_r) 
   dr\\
   &= :& \Delta_{i, 1}^{\mathbf{Z}} + \Delta_{i, 2}^{\mathbf{Z}}
   \end{eqnarray*}
   and
   \begin{eqnarray*}
  \tilde{\Delta}^{\mathbf{Z}}_i & : = & F^{\mathbf{Z}}_{t_{i + 1}} (\tilde{Y}_{t_{i +
  1}}^{\mathbf{Z}}) - F^{\mathbf{Z}}_{t_{i + 1}} (Y_{t_i}) \quad
  (\text{with }\tilde{Y}_r^{\mathbf{Z}}   := X_r^{t_i, Y_{t_i} ; \mathbf{Z}})\\
  & = & \int_{t_i}^{t_{i + 1}} D F^{\mathbf{Z}}_{t_{i + 1}} (\tilde{Y}_r) 
  \hspace{0.17em} \sigma (\tilde{Y}_r^{\mathbf{Z}}) d \mathbf{Z}_r +
  \int_{t_i}^{t_{i+1}}  L^{\sigma, \mu}_r F^{\mathbf{Z}}_{t_{i + 1}}
  (\tilde{Y}_r^{\mathbf{Z}})dr\\
  &=: & \tilde{\Delta}^{\mathbf{Z}}_{i, 1} + \tilde{\Delta}^{\mathbf{Z}}_{i, 2}.
\end{eqnarray*}
In the above, 
\begin{equation*}
\begin{aligned}
L_r^{\beta, b}\varphi (\cdot) &\coloneqq  D \varphi (\cdot) b_r+\frac{1}{2}D^{2} \varphi (\cdot) (\beta_r, \beta_r); \quad L^{\sigma, \mu}_{r} \varphi (\cdot)\coloneqq D\varphi(\cdot) \mu(\cdot)+ \frac{1}{2} D^{2}\varphi(\cdot) (\sigma(\cdot), \sigma(\cdot)) \dot{[\mathbf{Z}]}_{r}
\end{aligned}
\end{equation*}
for functions $\varphi \in \Lip^{3}$.
We discuss randomization of $\Delta^{\mathbf{Z}}_{i, 1},  \tilde{\Delta}^{\mathbf{Z}}_{i, 1}$
only, the Lebesgue integral terms are left to the readers. First, we deal with $\Delta^{\mathbf{Z}}_{i, 1}$. We denote by 
$\mathfrak{C}^T_t$ the collection of  Borel sets of $\mathscr{C}^{0, \alpha, 1}([t, T]; \mathbb{R}^{d_{W}})$. By
\cref{lem:randomization} below, we  have a $(\mathfrak{C}^T_{t_{i + 1}} \otimes \mathfrak{F} _{t_{i +
1}})$-measurable version of
\[ (\mathbf{Z}, \omega) \mapsto \int_{t_i}^{t_{i + 1}} D
   F^{\mathbf{Z}}_{t_{i + 1}} (Y_r)  \hspace{0.17em} \beta_r dW_r \]
such that the Brownian randomization (on $[t_{i + 1}, T]$) gives
\[ \overline{\int_{t_i}^{t_{i + 1}} D F^{\mathbf{Z}}_{t_{i + 1}} (Y_r) 
   \hspace{0.17em} \beta_r dW_r} = \int_{t_i}^{t_{i + 1}} D\overline{ F_{t_{i
   + 1}}} (Y_r)  \hspace{0.17em} \beta_r dW_r, \]
noting that $\{W_r, r\in[t_i, t_{i+1}]\}$ is an $(\mathfrak{F}_r  \vee \mathfrak{F}_{t_{i +
1}}^T)$-Brownian motion, while the Brownian
randomization of $\mathbf{Z}$ (on $[t_{i + 1}, T]$) is $\mathfrak{F}_{t_{i +
1}}^T$-measurable, and that  $\mathfrak{F}_r  \perp \mathfrak{F}_{t_{i +
1}}^T$ for $r \le t_{i+1}$. By \cref{lemma:measurable_select_1} this gives that a.s. 
\begin{equation*}
    \overline{\Delta_{i}}=   \int_{t_i}^{t_{i + 1}} D
  \overline{F}_{t_{i + 1}} (Y_r)  \hspace{0.17em} \beta_r dW_r + \int_{t_i}^{t_{i+1}}  L^{\beta, b}_r \overline{F}_{t_{i + 1}} (Y_r)dr:=\overline{\Delta_{i,1}}+\overline{\Delta_{i,2}}.
\end{equation*}
Second, we treat $\tilde{\Delta}_{i, 1}$. We have to deal with the randomization of 
\[ \int_{t_i}^{t_{i + 1}} D F^{\mathbf{Z}}_{t_{i + 1}} (\tilde{Y}_r^{\mathbf{Z}}) 
    \sigma (\tilde{Y}_r^{\mathbf{Z}}) d \mathbf{Z}_r =
   \int_{t_i}^{t_{i + 1}} \Gamma F_{t_{i + 1}}^{\mathbf{Z}} \left( X_r^{t_i,
   Y_{t_i} ; \mathbf{Z}} \right) d \mathbf{Z}_r . \]
Different from before, we now have dependence of $\mathbf{Z}$ on $[t_i,
T]$. The idea is a sort of ``tower property'' for (Brownian) randomization. We first perform the It\^o Brownian randomization of $\mathbf{Z}$ only on $[t_i, t_{i + 1}]$, so that
$\Gamma F_{t_{i + 1}}^{\mathbf{Z}} (x)$ stays deterministic. In this
setting, it is standard to obtain
\[ \overline{\left( \int_{t_i}^{t_{i + 1}} \Gamma F_{t_{i + 1}}^{\mathbf{Z}}
   \left( X_r^{t_i, Y_{t_i} ; \mathbf{Z}} \right) d \mathbf{Z}_r 
   \right)}^{\mathbf{Z} ; t_{i + 1}, T} = \int_{t_i}^{t_{i + 1}}\Gamma
   \overline{F _{t_{i + 1}}}^{\mathbf{Z} ; t_{i + 1}, T}
   (\bar{X}_r^{t_i, Y_{t_i}})  \hspace{0.17em} dW_r,\]
where the notation $\overline{G}^{\mathbf Z;t_{i+1},T}$ means that we only randomize $G$ up to time $t_{i + 1}$,
thus maintaining dependence on $\mathbf{Z}$ on $[t_{i + 1}, T]$. The
integrand on the right-hand side is easily seen to be $(\mathfrak{C}^T_{t_{i +
1}} \otimes \mathfrak{F} _r)$-predictable, for $r \in [t_i, t_{i + 1}]$, we
can now conclude as before by applying \cref{lem:randomization} to see that the full randomization satisfies 
   \[ \overline{\left( \int_{t_i}^{t_{i + 1}} \Gamma F_{t_{i + 1}}^{\mathbf{Z}}
   \left( X_r^{t_i, Y_{t_i} ; \mathbf{Z}} \right) d \mathbf{Z}_r  \right)}
   = \int_{t_i}^{t_{i + 1}} \Gamma\overline{F _{t_{i + 1}}} 
   (\bar{X}_r^{t_i, Y_{t_i}})dW_r. 
   \]
Similar to $\overline{\Delta_i}$, we also get by \cref{lemma:measurable_select_1}
\begin{equation*}
    \overline{\tilde\Delta_{i}}=   \int_{t_i}^{t_{i + 1}} D
  \overline{F}_{t_{i + 1}} (\bar{X}_r^{t_i, Y_{t_i}})\sigma(\bar{X}_r^{t_i, Y_{t_i}}) dW_r + \int_{t_i}^{t_{i+1}}  L^{\sigma, \mu}_r \overline{F}_{t_{i + 1}} (\bar{X}_r^{t_i, Y_{t_i}})dr=\overline{\tilde \Delta_{i,1}}+\overline{\tilde \Delta_{i,2}}.
\end{equation*}

\

By \cref{prop:MC}, we have,  letting $r^{\pi}$ denote the smallest $t_{i + 1} \geqslant r$ in the partition
$\pi= \{0=t_{0}< t_1< \dots< t_N=T\}$,

\[ \sum_{i=0}^{N} \overline{\Delta_{i, 1}} = \int_0^T D\overline{F_{r^{\pi}}} (Y_r)
   \beta_r (\omega) \diamond dW_r , \quad  \sum_{i=0}^{N} \overline{\tilde{\Delta}_{i, 1}} = \int_0^T D\overline{F_{r^{\pi}}}
   (\bar{X}_r^{r^{\pi}, Y_{r^{\pi}}}) \sigma \left( \bar{X}_r^{r^{\pi},
   Y_{r^{\pi}}} \right) \diamond dW_r , \]
   and similarly for the Lebesgue integrals.   By the decomposition~\eqref{e:decomp}, 
we have 
\begin{equation}\label{e:S-pi}
 S^{\pi}_{0, T} := \sum_i (\overline{\Delta_{i, 1}} -
   \overline{\tilde{\Delta}_{i, 1}}) = \bar{F}_T (Y_T) - \bar{F}_0  (Y_0) -
   L^{\pi}_{0, T},
   \end{equation}
where 
\[ L^{\pi}_{0, T} := \int_0^T \left\{ L^{\sigma, \mu}_r \overline{
   F_{r^{\pi}}} (\bar{X}_r^{r^{\pi}, Y_{r^{\pi}}}) - L^{\beta, b}_r 
   \overline{F_{r^{\pi}}} (Y_r) \right\} dr. \]
Noting $\mu \in \mathrm{Lip}^3$ and  $\sigma \in \mathrm{Lip}^5$, we get that
\begin{equation*}
   \frac{\partial}{\partial x} \bar X_T^{r,x}\Big|_{x=Y_r}, \frac{\partial}{\partial x} \bar X_T^{r,x}\Big|_{x= \bar X_r^{r^\pi, Y_{r^\pi}} }, \frac{\partial^2}{\partial x^2} \bar X_T^{r,x}\Big|_{x=Y_r},  \frac{\partial^2}{\partial x^2} \bar X_T^{r,x}\Big|_{x=\bar X_r^{r^\pi, Y_{r^\pi}} } \in L^{q} ([0, T] \times \Omega, \mathrm{Leb}
\otimes \mathbb{P})
\end{equation*}
for all $q>1$.  Then, by the assumption $\beta \beta^{\perp}, b \in L^{p} ([0, T] \times \Omega, \mathrm{Leb}
\otimes \mathbb{P})$ and the boundedness of  the functions $f , D f, D^2 f, \sigma, \mu$,   we can apply  the dominated convergence theorem and get that $L^{\pi}_{0, T}
(\omega) \rightarrow L _{0, T} (\omega)$ in $L^{p'} (\Omega, \mathbb{P})$ for any $p'\in(1, p)$, with
obvious limit obtained from \ $F_{r^{\pi}} \to F_r$
and $\bar{X}_r^{r_{\pi}, Y_{r_{\pi}}} \to Y_r$ as $|\pi|\to 0$. 

Thus,  by \eqref{e:S-pi} we get  that 
$S^{\pi}_{0, T} = \bd (u^{\pi})$ converges in $L^{p'} (\Omega,
\mathbb{P})$, where
\[ u^{\pi}_{r} := D \overline{F_{r^{\pi}}} (Y_r) \beta_r -
   D\overline{F_{r^{\pi}}} (\bar{X}_r^{r^{\pi}, Y_{r^{\pi}}}) \sigma \left(
   \bar{X}_r^{r^{\pi}, Y_{r^{\pi}}} \right) . \]
Moreover,  we can also show by the dominated convergence theorem that $u^{\pi}_{r}(\omega) \rightarrow u_{r} (\omega): =  \left(D\overline{F_r} (Y_r) \beta_r \right)(\omega) -
   \left(D\overline{F_r} (Y_r) \sigma \left(Y_r \right)\right)(\omega)$ in $L^{p'} ([0, T]
\times \Omega, \mathrm{Leb} \otimes \mathbb{P})$.  By the closedness of the divergence operator $\bd$, we get $u\in \mathrm{Dom}_{p'}(\bd)$ and 
$S^{\pi}_{0, T} \rightarrow \bd (u )$  in $L^{p'} (\Omega,
\mathbb{P})$.
\end{proof}
We used the following facts on measurable selection from \cite{stricker_yor_78} and \cite{flz24}. Let $U$ denote some Polish space and $\mathfrak{U}$ it's Borel-$\sigma$-algebra. Further let $(\Omega, \mathfrak{F}, (\mathfrak{F}_{t})_{t \in [0,T]}, \mathbb{P})$ be a filtered probability space satisfying the usual assumptions. 
\begin{lemma}
\label{lemma:measurable_select_1}
It holds 
\begin{enumerate}
    \item Let $(X_{n})_{n \in \mathbb{N}}$ be a sequence of $\mathfrak{F}\otimes \mathfrak{U}$-measurable functions on $\Omega \times U$. Suppose that for all $u \in U$ the sequence $X_{n}(\cdot, u)$ converges in $\mathbb{P}$. Then there exists an $\mathfrak{F}\otimes \mathfrak{U}$-measurable function $X$ s.t. for any $u \in U$ $X(\cdot, u)=\lim_{n\to \infty} X_{n}(\cdot, u)$ in $\mathbb{P}$. 
    \item Let $Q$ be a probability measure on $(U, \mathfrak{U})$ and  $X, Y$ both $\mathfrak{F}\otimes \mathfrak{U}$-measurable functions on $\Omega \times U$ such that $X(\cdot, u)=Y(\cdot, u)$ $\mathbb{P}$-almost surely. Then $X=Y$ $\mathbb{P}\otimes Q$-almost surely on $\Omega \times U$.
\end{enumerate}
\end{lemma}
\begin{lemma}\label{lem:randomization}
\label{lemma:measurable_selection_2}
  Let $M$ be a local $(\mathfrak{F}_t)$-martingale, $\gamma = \gamma_t (\omega, u)$ an  ($\mathfrak{F}_t\otimes \mathfrak{U}$)-predictable
  process on $[0, T]\times \Omega \times U .$ Consider a $U$-valued random variable  $A$ which is independent of
  $\mathfrak{F}_T$. Then \ $I (u, \omega) : = \int_0^T \gamma_t (u,
  \omega) d M_t$ admits a (jointly) measurable version, denoted by the same
  expression, with randomization
  \[ \bar{I} (\omega) := I (A (\omega), \omega) = \int_0^T \bar{\gamma}_t
     (\omega) d M_t, \quad \bar{\gamma}_t (\omega) := \gamma_t (A
     (\omega) , \omega) \]
  Here the right-hand side is a well-defined stochastic integral since $M$ is
  also a local $(\mathfrak{F}_r \vee \mathfrak{U} : 0 \leqslant r \leqslant
  T)$-martingale, thanks to independence, and $\bar{\gamma}_t$ is plainly
  $(\mathfrak{F}_r \vee \mathfrak{U} : 0 \leqslant r \leqslant T)$-predictable. 
\end{lemma}
We further used the following fact from Malliavin Calculus, also used in \cite{HHJM24}:
\begin{proposition} \label{prop:MC}
  Let $W$ be a standard Brownian motion, on standard Wiener space $C ([0, T],
  \mathbb{R}^d) .$ For $t \in (0, T]$, define a filtration $\mathfrak{F}^{t,
  T}_r$ on $[0, t]$ given by
  \[ \mathfrak{F}^{t, T}_r := \sigma \left( W_v - W_u :  (u, v) \in \Delta_{0,
     r} \cup \Delta_{t, T} \right ) \vee \mathcal{N}, \quad 0 \leqslant r \leqslant
     t, \]
     where $\mathcal{N}$ denotes the collection of nullsets. \\
  Then (i) Fix $t \in (0, T]$. Then $(W_r : 0 \leqslant r \leqslant t)$ is a \
  $\mathfrak{F}^{t, T}_r$-Brownian motion.
  (ii) Fix $(s, t) \in \Delta_T .$ Let $(X_r : s \leqslant r \leqslant t)$ be
  $\mathfrak{F}^{t, T}_r$-predictable $L^2$-process (i.e. $\mathbb{E} \int_s^t
  | X_r |^2 d r < \infty$). Then
\[
\big( \mathbbm{1}_{[s,t]} X_r \big)(t,\omega)
:=
\begin{cases}
X_r(\omega), & r \in [s,t],\\[6pt]
0, & \text{otherwise}.
\end{cases}
\]
  is Skorokhod integrable, on $[0, T]$, and (with well-defined It\^{o} integral on
  right-hand side)
  \[ \int_0^T \mathbbm 1_{[s, t]} (r) X_r \diamond d W_r = \int_s^t X_r d W_r . \]
\end{proposition}

\begin{proof} 
(i) is elementary. For (ii) we note that it suffices to treat the case $s =
 0$, the general case essentially follows by taking differences. As noted in
\cite{HHJM24}, this statement a variant of well-known facts in Malliavin calculus. 
\end{proof}
\begin{remark} \label{subsection:IAG_general_rf}
Looking at the proof of \cref{theorem:IAG} reveals that the precise structure $F^{\mathbf{Z }}_s
(y) = f (X_T^{\mathbf{Z} ; s, y})$ is not crucial. What one needs is a
family $\textrm{} \mathrm{G}^{\mathbf{Z}} = (G, G')  = (G, G') (y ; s, t)$
of $Z$-controlled rough paths on $[s, t]$ such that
\begin{equation}
\label{eq:two_sided_causality}
F_t^{\mathbf{Z}} (y) - F_s^{\mathbf{Z}} (y) = \int_s^t
   {\mathrm{G}_r^{\mathbf{Z}}}  (y ; s, t) d {\mathbf{Z}_r}  + \int_s^t
   L^{\mathbf{Z}}_r (y ; s, t) d r 
\end{equation}
with a sort of {\em 2-sided batch causality} w.r.t. $\mathbf{Z}$
meaning that ${\mathrm{G}_r^{\mathbf{Z}}}  (y ; s, t)$ only depends on
$\mathbf{Z}$ (increments) on $[s, r] \cup [t, T]$. 
In the above case this is ensured
by the RDE flow property
\[ F_s^{\mathbf{Z}} (y) = f (X_T^{s, y;\mathbf{Z}}) = f
   \left(X_T^{t,X_t^{s, y; \mathbf{Z}}; \mathbf{Z}} \right) =
   F_t^{\mathbf{Z}} (X_t^{s, y;\mathbf{Z}}),  \]
together with the (deterministic) rough It\^{o} formula, applied to $F_t^{\mathbf{Z}} (.)$ for fixed $t$, which gives
\begin{equation*}
\begin{aligned}
{\mathrm{G}_r^{\mathbf{Z}}}  (y ; s, t) &= \left( -\Gamma 
   F_t^{\mathbf{Z}} (X_r^{\mathbf{Z} ; s, y}) {, -\Gamma^2} 
   F_t^{\mathbf{Z}} (X_r^{\mathbf{Z} ; s, y}) \right); \\
   L^{\mathbf{Z}}_r (y ; s, t) &= -\left(D F_t^{\mathbf{Z}}b\right)
   (X_r^{\mathbf{Z} ; s, y})-\frac{1}{2}\left(D^{2} F^{\mathbf{Z}}_{t}(\sigma, \sigma)\right)(X_{r}^{s, y; \mathbf{Z}}) .
\end{aligned}
\end{equation*}

There is another class of interesting examples based on rough SDEs, rather than RDEs. That is, consider a family of rough semimartingales $(R^{s, y; \mathbf{Z}},\sigma(R^{s, y; \mathbf{Z}}))$ to be the unique solution of the rough SDE (RSDE)
\[ dR^{s, y; \mathbf{Z}}_{t} = \mu (R^{s, y; \mathbf{Z}}_{t}) dt + \sigma (R^{s, y; \mathbf{Z}}_{t}) d \mathbf{Z}_t + \sigma_0 (R^{s, y; \mathbf{Z}}_{t}) dB_t (\omega); \quad R^{s, y; \mathbf{Z}}_{s}=y, \]
where $B$ denotes a Brownian motion. 
We shall here simply assume that for test functions $f\in \Lip^{3}$, the field 
$$(s,y) \mapsto \left(F_s^{\mathbf{Z}} (y) := \mathbb{E}\left[f (R_T^{s, y; \mathbf{Z}})\right], \dots \right) \in 
\mathscr{D}^{3\alpha}_{\mathbf{Z}}{\Lip}^3_x,
$$
noting that such results are highly non-trivial and were obtained in \cite{friz_2021,BuginiFrizStannat2024}.
We now consider another Brownian motion $W$ independent of $B$ to randomize the parameter dependence on $\mathbf{Z}$. 
As is well-known in \cite{fzk23, friz_randomisation_2025-1}, (independent Brownian) randomization of such an 
RSDE yields a solution $\bar{R} = R^{\mathbf{Z}} |_{\mathbf{Z} = \mathbf{W} (\omega)}$ 
to a ``doubly'' SDE.
By the 
Markov property,
\[ F_s^{\mathbf{Z}} (y) =\mathbb{E}\left[f (R_T^{s, y; \mathbf{Z}})\right] =\mathbb{E}\left[f
   \left(R_T^{t, R_t^{s, y;\mathbf{Z}} ;\mathbf{Z}} \right)\right] =\mathbb{E}\left[F_t^{\mathbf{Z}}
   (R_t^{s, y; \mathbf{Z}}) \right].\]
By \cite{BuginiFrizStannat2024} this defines a $\mathbf{Z}$-controlled field. Apply rough stochastic It\^{o} formula to $F_t^{\mathbf{Z}}
(\cdot)$ to see that, writing $\Gamma (\cdot) = D (\cdot) \sigma, \, \Gamma_0 (\cdot) = D (\cdot) \sigma_{0}$,
\begin{eqnarray*}
  F_t^{\mathbf{Z}} (X_t^{s, y; \mathbf{Z} }) & = & F_t^{\mathbf{Z}} (y) 
  + \int_s^t \Gamma F_t^{\mathbf{Z}} (R_r^{s, y;\mathbf{Z}}) d
  \mathbf{Z}_r + \int_s^t \Gamma_0 F_t^{\mathbf{Z}} (R_r^{s, y;\mathbf{Z}}) d B (\omega)+  \int_s^t \mathcal{L}_r F_t^{\mathbf{Z}} (R_r^{s, y; \mathbf{Z}}) d
  r
\end{eqnarray*}
where $\mathcal{L}_r$ contains contributions from drift $\mu$, and the (rough
resp. stochastic) bracket terms associated to $[\mathbf{Z}]$ and $\langle \sigma_0
\bullet B\rangle$. Taking expectations, using rough Fubini \cite{friz_2020} and suitable
bounds on $\sigma_0$ to ensure martingality of the $d B$-integral, gives
\[ F_t^{\mathbf{Z}} (y)-F_s^{\mathbf{Z}} (y) =   -\int_s^t \mathbb{E}
    \left[\Gamma F_t^{\mathbf{Z}} (R_r^{s, y;\mathbf{Z}}) \right]d \mathbf{Z}_r -
   \int_s^t \mathbb{E} \left[ \mathcal{L}_r F_t^{\mathbf{Z}} (R_r^{s, y; \mathbf{Z}})\right] d r \]
where the integrands are given, on $[s, t]$ and in some more details, by
\[ (G, G')^{\mathbf{Z}} (y ; s, t) =\left(-\mathbb{E} \left[\Gamma F_t^{\mathbf{Z}}
   (R_r^{s, y;\mathbf{Z}})\right], -\mathbb{E} \left[ \Gamma^2 F_t^{\mathbf{Z}}
   (R_r^{s, y;\mathbf{Z}})\right]\right), L^{\mathbf{Z}}_r (y ; s, t) =\mathbb{E}
   (\mathcal{L}_r F_t^{\mathbf{Z}} (R_r^{s, y;\mathbf{Z}})) . \]
Clearly, they enjoy the 2-sided causality condition \eqref{eq:two_sided_causality}.

The resulting IAG formula
is not very different from \cref{theorem:IAG}. Since, as $t \downarrow s$, $G^{\mathbf{Z}} (y)$
approaches $- \Gamma F_s^{\mathbf{Z}} (y)$ so that the term
``$-D\overline{F}_r (Y_r) \sigma (Y_r) \equiv -\Gamma \overline{F_r} (Y_r)$''
in the Skorokhod integral is unchanged. Similarly, $L^{\mathbf{Z}}_r (y) \sim
\mathcal{L}_r F_t^{\mathbf{Z}} (y)$. The resulting IAG formula then reads \begin{eqnarray*}
  \overline{F}_t (Y_t) - \overline{F}_s (Y_s) & = & \int_s^t D\overline{F}_r (Y_r)  (b_r
  - \mu (Y_r)) dr + \int_s^t D\overline{F}_r (Y_r)  (\beta_r - \sigma (Y_r))
  \diamond dW_r \nonumber\\
  &  & + \frac{1}{2}  \int_s^t D^{2}\overline{F_r} (Y_r)  ((\beta_r, \beta_r)
  - (\sigma (Y_r), \sigma (Y_r)) - (\sigma_0 (Y_r), \sigma_0 (Y_r))) dr,  \nonumber
\end{eqnarray*}
where by \cite{friz_randomisation_2025-1} and \cite{fzk23}
\[ \overline{F_s } (y) =  \mathbb{E}\left[f (R_T^{s, y;\mathbf{Z}})\right]
   \bigg|_{\mathbf{Z} \rightsquigarrow \mathbf{W}^{\text{It\^{o}}} (\omega)} =\mathbb{E} \left [f
   (\bar{R}_T^{\textbf{} s, y}) | \sigma (W (t) : s \leqslant t \leqslant T)
   \right ] \]
   can be seen as (Feynman-Kac) solution to a (terminal value ``backward'') SPDE.
   
 \end{remark}   
\subsection{A note on the stochastic interpolation formula}
\label{subsection:stochastic_interpolation_formula}
Let $\mathbf{Z} = (Z, \mathbb{Z}) \in \mathscr{C}^{0,\alpha}_g ([0, T] ;
\mathbb{R}^{d_Z})$ be a deterministic geometric rough path. For  $(t, x)
\in [0, T] \times \mathbb{R}^{d_X}$,  denote by $(X_s^{t, x; \mathbf{Z}})_{s \in
[t, T]}$ the unique solution of the RDE$(\mu, \sigma)$,
\[ dX^{t, x; \mathbf{Z}}_s = \mu (X_s^{t, x; \mathbf{Z}}) ds + \sigma (X^{t,
   x; \mathbf{Z}}_s) d \mathbf{Z}_s ; \quad X^{t, x; \mathbf{Z}}_t = x. \]
Assume proper conditions on $\mu$ and $\sigma$  such that for $F^{\mathbf{Z}}_t (x) \coloneqq f (X^{t, x; \mathbf{Z}}_T)$ it holds
\[ (F^{\mathbf{Z}}, \dots) \in \mathscr{D}^{3
   \alpha}_{\mathbf{X}} {\Lip}^3_{x, \loc} . \]
(Classic sufficient conditions are $\sigma \in \Lip^5 (\mathbb{R}^{d_X} ;
\mathbb{R}^{d_X \times d_Z})$ and $\mu \in \Lip^3 (\mathbb{R}^{d_X} ;
\mathbb{R}^{d_X})$ see \cref{ex:RDE_flows}, but this can be relaxed to local regularity, as long as
one can rule out explosion. See  \cref{rem:no_explosion}.) 
Let  $Y_{\bullet}=\widehat{X
}^{\mathbf{Z}}_{\bullet \leftarrow s} (x) \coloneqq \hat{X}^{s, x; \mathbf{Z}}_{\bullet}$  be the  solution to the RDE$(\hat{\mu},
\hat{\sigma})$ under similar well-posedness assumptions on $\hat{\mu}, \hat{\sigma}$. By \cref{ex:RDE_flows} we see, $(Y, \hat{\sigma} (Y), ((D \hat{\sigma})
\hat{\sigma}) (Y), \hat{\mu} (Y)) \in \mathscr{D}^{3 \alpha}_{\mathbf{X}} ([0,
T] ; \mathbb{R}^{d_X})$.
By applying \cref{lemma:rAG} for the test-function $f=\operatorname{Id}$ we obtain a {\em rough forward-backward interpolation formula}:

\begin{equation}\label{e:rfbif}
\begin{aligned}
  \widehat{X }^{\mathbf{Z}}_{t \leftarrow s} (x) - X^{\mathbf{Z}}_{t
  \leftarrow s} (x) & =  \int_s^t 
  D_x X^{\mathbf{Z}}_{t \leftarrow u}
  (\widehat{X }^{\mathbf{Z}}_{u \leftarrow s} (x))  (\hat{\mu} - \mu)
  (\widehat{X }^{\mathbf{Z}}_{u \leftarrow s} (x))  \hspace{0.17em} d u\\
  &   + \int_s^t D_x X^{\mathbf{Z}}_{t \leftarrow u} (\widehat{X
  }^{\mathbf{Z}}_{u \leftarrow s} (x)) (\hat{\sigma} - \sigma) (\widehat{X
  }^{\mathbf{Z}}_{u \leftarrow s} (x)) \circ d \mathbf{Z}_u .
\end{aligned}
\end{equation}
Every geometric rough path $\mathbf{Z}$ (over $\mathbb{R}^{d_Z}$) is, by definition, the limit of some canonically lifted, smooth sequence $(Z^n)$ in rough path space. In \cite{Coutin2007},  the
sequence  $(Z^n)$  is called {\em good}, if $(\mathbf{Z} ; Z^n)$ converges to $(\mathbf{Z} ;
\mathbf{Z})$, the (canonically defined) doubled rough path over
$\mathbb{R}^{2 d_Z}$, in the $\alpha$-H\"older rough path space. Writing
$\varphi^{s, t, \mathbf{Z}}_u$ for the (controlled) integrand of the $\circ
d \mathbf{Z}$-integral, we then have, by rough path stability results for
RDE flows, their Jacobians and rough integrals,
\begin{equation}\label{e:good-approx}
 \int_s^t \varphi^{s, t, \mathbf{Z}}_u \circ d \mathbf{Z}_u = \lim_{n \to \infty}
   \int_s^t \varphi^{s, t, \mathbf{Z}}_u \dot{Z}^n_u d u. 
   \end{equation}
It was seen in \cite{Coutin2007} that for any (deterministic) sequence of partitions with
mesh-size $| \pi^n | \downarrow 0$, piecewise linear approximation to the  Brownian
motion $W$ is good (in probability) for  any $\alpha < 1 / 2$. Applying \eqref{e:good-approx}
with $\mathbf{Z} \rightsquigarrow \left( W, \int \delta W \otimes \circ
\mathrm{dW} \right) (\omega) \equiv \mathbf{W}^{\mathrm{Strato}} (\omega)$, we get that, a.s.
\[ {\left\{ \int_s^t \varphi^{\mathbf{Z}}_u (\omega) \circ d \mathbf{Z}_u
   \right\}_{\mathbf{Z} \rightsquigarrow \mathbf{W}^{\mathrm{Strato}}}}  =
   \int_s^t \bar{\varphi}^{s, t}_u \circ d W_u ,\]
   with  
\[ \bar{\varphi}^{s, t}_u = D_x X _{t \leftarrow u} (\widehat{X }_{u \leftarrow
   s} (x)) (\hat{\sigma} - \sigma) (\widehat{X }_{u \leftarrow s} (x)), \]
where we write $X,\hat{X}$ for the randomized RDEs, which satisfies classical (Stratonovich) SDEs.   
The integral on the
right-hand side is then an anticipating Stratonovich integral, by the very definition of
this integral \cite[Definition 3.1.1]{nualart_2006}. 
Assuming suitable Malliavin regularity conditions, which are
certainly satisfied when dealing with reasonable SDE solutions, one has the
Stratonovich-Skorokhod conversion formula \cite[Theorem 3.1.1]{nualart_2006}
\[ \int_s^t \bar{\varphi}^{s, t}_u \circ d W_u = \int_s^t \bar{\varphi}^{s,
   t}_u \diamond d W_u + \frac{1}{2} \int_s^t (\mathbf D_u^+ + \mathbf D_u^-)
   \bar{\varphi}^{s, t}_u d u, \]
where  $\mathbf D^{\pm}$ are left- and right-sided  Malliavin derivatives defined by   $\mathbf D_u^{\pm} F_u:=\lim_{s\to u^{\pm}}\mathbf D_u F_s$. 
Thus, the randomization of  \eqref{e:rfbif} yields
\begin{equation}\label{e:Itofbif}
\begin{aligned}
  \widehat{X }_{t \leftarrow s} (x) - X_{t
  \leftarrow s} (x) & =  \int_s^t 
  D_x X_{t \leftarrow u}
  (\widehat{X }_{u \leftarrow s} (x))  (\hat{\mu} - \mu)
  (\widehat{X }_{u \leftarrow s} (x))  \hspace{0.17em} d u\\
  & \quad  + \frac{1}{2} \int_s^t (\mathbf D_u^+ + \mathbf D_u^-)
   \bar{\varphi}^{s, t}_u d u+ \int_s^t \bar{\varphi}^{s,
   t}_u \diamond d W_u.
\end{aligned}
\end{equation}

 By the results of  \cite{DELMORAL2022197} or \cite{HHJM24} as exposed in the \cref{subsection:forward_backward_stochastic_analysis}, we should also have, noting the It\^o-Stratonovich correction terms for the SDEs of $X$ and $\widehat X$, 
\begin{equation}\label{e:Itofbif'}
\begin{aligned}
  \widehat{X }_{t \leftarrow s} (x) - X_{t
  \leftarrow s} (x) & =  \int_s^t 
  D_x X_{t \leftarrow u}
  (\widehat{X }_{u \leftarrow s} (x))  (\hat{\mu} - \mu)
  (\widehat{X }_{u \leftarrow s} (x))   d u\\
  &\quad +\frac{1}{2}  \int_s^t D_x^2 X_{t \leftarrow u} (\hat{X}_{u \leftarrow s} (x)) 
((\hat{\sigma} \hat{\sigma}^T - \sigma \sigma^T) (\hat{X}_{u \leftarrow s}
(x))) du\\
&\quad + \frac{1}{2} \int_s^t D_x X_{t \leftarrow u} (\widehat{X}_{u \leftarrow s}
   (x))  ((D_x \hat{\sigma}) \hat{\sigma} - (D_x \sigma) \sigma) (\widehat{X }_{u
   \leftarrow s} (x))   d u \\
  & \quad + \int_s^t \bar{\varphi}^{s,
   t}_u \diamond d W_u .
\end{aligned}
\end{equation}

Now, we show that \eqref{e:Itofbif}  and \eqref{e:Itofbif'}  do agree with each other.  Let $\{F(x), x\in\mathbb R^{d_X}\}$ be a random process which has  a.s. continuously differentiable paths  such that $F(x)\in \mathcal D^{1,2}$ for each $x\in \mathbb R^{d_X}$. Let $X$ be a  random variable belonging to $\mathcal D^{1,2}$. Then,   by \cite[Lemma 2.3]{OconePardoux1989} (see also \cite[Exercise 1.3.6]{nualart_2006}),  under property conditions, we have the following Wentzell-type chain rule: 
\[\mathbf D F(X) = (\mathbf D F(x))|_{x=X}+ D_x F(X) \mathbf D X = :(\mathbf DF)(X) + D_x F(X) \mathbf DX. \]
Applying the above formula and noting that $\mathbf D^{-}_u (\widehat X_{u\leftarrow s})=0, \mathbf D^{+}_u (\widehat X_{u\leftarrow s})=\hat \sigma(\widehat X_{u\leftarrow s})$ a.s., we get 
\begin{align*}
    &(\mathbf D_u^+ + \mathbf D_u^-)
   \bar{\varphi}^{s, t}_u= (\mathbf D_u^+ + \mathbf D_u^-) D_x X _{t \leftarrow u} (\widehat{X }_{u \leftarrow
   s} (x)) (\hat{\sigma} - \sigma) (\widehat{X }_{u \leftarrow s} (x))\\
   &=\left[\left( \mathbf D_u^+ D_x X _{t \leftarrow u}\right) (\widehat{X }_{u \leftarrow   s} (x))+ D^2_x X _{t \leftarrow u} (\widehat{X }_{u \leftarrow   s} (x)) \hat\sigma(\widehat{X }_{u \leftarrow   s}) \right] (\hat{\sigma} - \sigma) (\widehat{X }_{u \leftarrow s} (x))\\
   &\quad + \left[\left( \mathbf D_u^- D_x X _{t \leftarrow u}\right) (\widehat{X }_{u \leftarrow   s} (x)) \right] (\hat{\sigma} - \sigma) (\widehat{X }_{u \leftarrow s} (x))\\
&\quad +  D_x X _{t \leftarrow u} (\widehat{X }_{u \leftarrow
   s} (x)) ((D_x\hat{\sigma})\hat{\sigma} -(D_x \sigma )\hat{\sigma}) (\widehat{X }_{u \leftarrow s} (x)).
\end{align*}
In order to prove the consistence between \eqref{e:Itofbif}  and \eqref{e:Itofbif'}, it suffices to show that the above coincides with 
\[ D_x^2 X_{t \leftarrow u} (\hat{X}_{u \leftarrow s} (x)) 
((\hat{\sigma} \hat{\sigma}^T - \sigma \sigma^T) (\hat{X}_{u \leftarrow s}
(x)))+  D_x X_{t \leftarrow u} (\widehat{X}_{u \leftarrow s}
   (x))  ((D_x \hat{\sigma}) \hat{\sigma} - (D_x \sigma) \sigma) (\widehat{X }_{u
   \leftarrow s} (x)).  \]
   This reduces to prove the following identity:
\begin{align*}
    & \left(\mathbf D_u^+ D_x X _{t \leftarrow u}\right) (\widehat{X }_{u \leftarrow   s} (x))+\left(\mathbf D_u^- D_x X _{t \leftarrow u}\right) (\widehat{X }_{u \leftarrow   s} (x))\\
    &= D_x^2 X_{t\leftarrow u}(\widehat X_{u\leftarrow s}(x))\sigma(\widehat X_{u\leftarrow s}(x))+  D_xX_{t\leftarrow u}(\widehat X_{u\leftarrow s}(x))D_x\sigma(\widehat X_{u\leftarrow s}(x)),
\end{align*}
or equivalently, noting that $\mathbf D_u F=0$ for $F\in \mathfrak F_{u',T}^W:=\sigma(W_t-W_{u'}, t\in[u',T])$ with $u'>u$ and hence $\left(\mathbf D_u^+ D_x X _{t \leftarrow u}\right)(x)$  vanishes, 
\begin{equation}\label{e:linear-eqs}
     \mathbf D_u^- D_x X _{t \leftarrow u} (x)=  D^2_xX_{t\leftarrow u}(x)\sigma(x)+ D_xX_{t\leftarrow u}(x)D_x\sigma(x).
\end{equation}

The Eq. \eqref{e:linear-eqs}
 can be verified by comparing the linear SDEs satisfied by $\mathbf D_u^- D_x X _{t \leftarrow u} (x)$, $ D^2_xX_{t\leftarrow u}(x)$ and $ D_xX_{t\leftarrow u}(x)$ as follows.  Recall that $X_{t\leftarrow u}(x) =X_t^{u,x}$  satisfies 
 \[ X^{u, x}_t =x+ \int_u^t \mu (X_r^{u, x}) dr + \int_u^t \sigma (X^{u,
   x}_r) \circ d W_r . \]
Then,   $ A_t:=  D_xX^{u,x}_t (x)$ and $ B_t:= D_x^2X^{u,x}_t (x)$ satisfy the following linear SDEs respectively:
\[  A_t =1+\int_u^t  D_x\mu(X_r^{u,x}) A_r dr + \int_u^t D_x\sigma(X_r^{u,x}) A_r\circ dW_r,\]
and 
\[  B_t =\int_u^t \left[ D^2_x\mu(X_r^{u,x}) A_r^2+D_x \mu(X_r^{u,x})   B_r  \right]dr +\int_u^t  \left[D^2_x\sigma(X_r^{u,x}) A_r^2+D_x \sigma(X_r^{u,x})   B_r \right]\circ dW_r. \]
For the term $C_t:=  \mathbf D_u^- D_x X^{u,x}_t (x)=D_u^- A_t$,  noting that $\mathbf D^-_u X_\cdot^{u,x}$ satisfies the same   linear stochastic  differential equation  as $D_x X_\cdot^{u,x}=A_\cdot$, but with  a different initial condition $\mathbf D^-_u X_u^{u,x}=\sigma(x)$, and hence $\mathbf D^-_u X_r^{u,x}=\sigma(x) D_xX_r^{u,x}=\sigma(x) A_r$, we get 
\begin{align*}
C_t &=D_x\sigma(x)+  \int_u^t  \left[\sigma(x)D^2_x\mu(X_r^{u,x}) A_r^2+ D_x \mu(X_r^{u,x})  C_r\right]  dr\\
&\quad +\int_u^t\left[  \sigma(x)D^2_x\sigma(X_r^{u,x}) A_r^2+D_x \sigma(X_r^{u,x})   C_r\right] \circ dW_r.
\end{align*}
Combining the equations of $A, B$ and $C$, it is straightforward to verify $C=D_x\sigma(x) A+\sigma(x) B$, which confirms \eqref{e:linear-eqs}.  This proves the consistency between  \eqref{e:Itofbif}  and \eqref{e:Itofbif'}.

\section{Comments and comparisons with existing literature} \label{sec:comments}
\addtocontents{toc}{\protect\setcounter{tocdepth}{1}}

\subsection{Rough (stochastic) analysis (\cref{section:Introduction to controlled fields,section:On Rough Stochastic Calculus} and \cref{appendix:section:On RDE-flows for non-autonomous vector fields})}
\label{subsection:Comparison_rIW}

Lyons’ seminal work~\cite{Lyons1998} introduced rough differential equations and established their well-posedness. Gubinelli~\cite{GUBINELLI2004} later introduced the fundamental notion of controlled rough paths, which provides a flexible analytic framework and is also employed in the present work. The Itô formula for rough paths, interpreted as a composition rule for jets, appeared in (the 2014 edition) of ~\cite{friz_2020}. 

The first link between rough path theory and (Stratonovich, but not Skorokhod) anticipating stochastic calculus is due to Coutin et al.~\cite{Coutin2007}. It may help to recall
\begin{equation}
  \int_0^T W_T \circ d W_t = W_T^2, 
  \qquad 
  \int_0^T W_T \diamond d W_t  = W_T^2 - T.  
\end{equation}
Naturaly (rough and/or stochastic) integrals which are based on the convergence of sums of the form $
\sum (\dots) (W_t - W_s)$, possibly helped by higher order terms of the form
$(\dots) \mathbb{W}$ to deal with fluctuations of the integrand, will pick the 
left (Stratonovich) integral.

Many authors have investigated rough path objects with spatial dependence, a perspective necessary for rough flows, rough transport equations and rough partial differential equations in general. A non-exhaustive list of works adopting such space-time viewpoints includes Caruana~\cite{caruana2009partial} (first work on rough PDEs, including rough transport), the monograph \cite{Friz_Victoir_2010}, which contains 
(amongst many other things) a first systematic study of rough flows, Bailleul’s flow approach~\cite{bailleul2015flows}, Keller and Zhang~\cite{keller_zhang_2016}, who provide a rough It\^o-Wentzell formula under a true roughness condition. In \cite{gubinelli_controlled_2014} the authors introduce space-time expansions of fields, similar to the 3rd order expansions in our controlled fields, to study the well-posedness of nonlinear RPDEs in viscosity sense. We emphasize however that controlled fields consist of a \emph{cascade} of such expansions and this is a crucial structural property used throughout \cref{section:Introduction to controlled fields}. We further mention works by  Bailleul, Gubinelli and Riedel ~\cite{bailleul2017,bailleul2019rough}, Bellingeri et al for higher order rough transport  \cite{bellingeri2021transport}, and the rough It\^{o}-Wentzell formula of Castrequini, Catuogno and Machado~\cite{castrequini_2025}.

The appearance of the stochastic sewing lemma~\cite{Le2020} opened the door to a new field of {\em rough stochastic analysis}. Foundational developments include \cite{friz_2021}, which dealt with well-posedness of rough SDEs, and the parallel rough semimartingale development ~\cite{fzk23} rooted in harmonic analysis. Applications to pathwise stochastic control were explored in~\cite{flz24}; there and in ~\cite{friz_randomisation_2025-1} the role of measurable selection was highlighted when randomizing rough SDEs to connect them to classical problems. Further applications to McKean–Vlasov equations with common noise appear in~\cite{friz_mckean-vlasov_2025}, see also~\cite{bugini_rough_2025} and~\cite{Bugini2025Nonlinear}. 
Extensions to jump processes are investigated in~\cite{allan_rough_2025}, while~\cite{bbfp25} presents applications to financial mathematics and to nonlinear stochastic PDEs.
 
In the context of Gaussian rough paths, Stratonovich–Skorokhod integral formulas were investigated in \cite{cass2019stratonovich,song_tindel_2022}. More precisely, let $X$ be a Gaussian process admitting a geometric rough path lift $\mathbf{X} = (X, \mathbb{X})$. The relation between the Skorokhod integral $\int y \diamond dX$ and the Stratonovich rough integral $\int y d\mathbf{X}$ was established in \cite{cass2019stratonovich} when $y$ is a solution to a rough differential equation driven by $\mathbf{X}$, and in \cite{song_tindel_2022} when $y$ is an adapted rough path controlled by $X$. We emphasize, however, that in both works the integrands $y$ are adapted processes. Consequently, these results do not apply to the (in general non-adapted) setting considered in \cref{section:Applications to (Rough) Stochastic Analysis}.

\subsection{Forward-backward stochastic analysis (\cref{section:Applications to (Rough) Stochastic Analysis})}
\label{subsection:forward_backward_stochastic_analysis}
In \cite{DELMORAL2022197} the authors consider a forward semimartingale and backward random field of the form \
\begin{equation}
  \label{eq:1.4} \left\{ \begin{aligned}
    Y_{s, t} & = y + \int_s^t B_{s, u}  \hspace{0.17em} du + \int_s^t
    \Sigma_{s, u}  \hspace{0.17em} dW_u,\\
    F_{s, t} (x) & = F (x) + \int_s^t G_{u, t} (x) \hspace{0.17em} du +
    \int_s^t H_{u, t} (x) \hspace{0.17em} dW_u,
  \end{aligned} \right.
\end{equation}
and show that, under suitable conditions (including Malliavin
differentiability of $\Sigma$) the following {\em backward It{\^o}--Wentzell formula}
\begin{align}
  F_{v, t} (Y_{s, v}) - F_{u, t} (Y_{s, u}) & = \int_u^v \left( D F_{r, t}
  (Y_{s, r}) B_{s, r} + \frac{1}{2} D^2 F_{r, t}
  (Y_{s, r})(\Sigma_{s, r} , \Sigma_{s, r}) - G_{r, t} (Y_{s, r})
  \right)   dr \nonumber\\
  & \quad + \int_u^v \Big(  DF_{r, t} (Y_{s, r})\Sigma_{s, r} -
  H_{r, t} (Y_{s, r}) \Big) \diamond d W_r .  \label{equ:Del1}
\end{align}
holds where the final $\diamond d W$ integral is understood as a Skorokhod stochastic integral. As an application, the authors consider two strong  solutions $X$ and $\hat{X}$ of the SDE $ dX_t = b (X_t)  dt + \sigma
(X_t)  dW_t$, with  $\hat{X}$ defined analogously. Write $X_{t \leftarrow s} (x) \equiv$ $X^{s,
x}_t$ if started at $X_s = x$. Under suitable assumptions one has a stochastic
flow, with well-defined Jacobian and Hessian $D X_{t \leftarrow s}, D^2 X_{t
\leftarrow s}$. 
Setting $\Delta g =   \hat{g} - g, \,  g \in \{ a,b,\sigma \}$, where $a = \sigma \sigma^{\top}$, the authors of \cite{DELMORAL2022197} find the {\em forward–backward stochastic interpolation formula}
\begin{equation}
\label{equ:Del2}
\begin{aligned}
  \hat{X}^{s,x}_t - X^{s,x}_t  \equiv &  \hat{X}_{t \leftarrow s} (x) - X_{t\leftarrow s} (x) \\
   = & \int_s^t  D X_{t \leftarrow u}
  (\hat{X}_{u \leftarrow s} (x))   (\Delta \mu (\hat{X}_{u \leftarrow s} (x))) du \\
  &  +  \frac{1}{2}  \int_s^t D^2 X_{t \leftarrow u} (\hat{X}_{u \leftarrow
  s} (x))  (\Delta a (\hat{X}_{u \leftarrow s} (x)))  du \\
  & + \int_s^t D X_{t \leftarrow u} (\hat{X}_{u \leftarrow s} (x))  (\Delta
  \sigma (\hat{X}_{u \leftarrow s} (x))) \diamond dW_u. 
\end{aligned}
\end{equation}

As noted in \cite{DELMORAL2022197}, this interpolation formula can be seen as an extension of the Alekseev–Gröbner lemma \cite{alekseev_1961,grobner_1960}, as well as a generalization of the classical variation-of-constants formula and related Gronwall-type lemmas to diffusion processes (but differs from the stochastic Gronwall lemma presented in \cite{scheutzow_2013}). 
It can also be seen as an extension of Theorem~6.1 in \cite{pardoux_protter_1987}  
on two-sided stochastic integrals to diffusion flows, and also interpreted as a backward version of the generalized It\^{o}--Wentzell formula presented in \cite{OconePardoux1989}~, also Theorem~3.2.11 in \cite{nualart_2006}. 
They also point to numerous origins in the literature, including Chapters 7-10 in \cite{DelMoral2004}, and references therein. 
Similar forward-backward interpolation formulas for stochastic matrix Riccati diffusion flows arising in data assimilation theory (cf., for example, \cite[Theorem~1.3]{bishop_2020}) were studied in a series of papers by Bishop et al.; cf.\ \cite{bishop_2019,bishop_2020} and references therein.

The forward-backward perturbation methodology has also been used in \cite{Arnaudon03092019, Arnaudon_Moral_2020} in the context of nonlinear diffusions and their mean field type
interacting particle interpretations, see for instance Section~2.3 in~\cite{Arnaudon_Moral_2020}. 

The work of \cite{DELMORAL2022197} can be seen as a natural extension of the second order perturbation methodology developed in the above referenced articles to diffusion type perturbed processes when $\sigma\neq\bar\sigma$.

The first article considering the case $\sigma\neq\bar\sigma$ with
$\sigma\neq 0$ and $\bar\sigma\neq 0$ was \cite{HHJM24}.
In this article, the authors discuss an It\^{o}--Alekseev--Gr\"obner formula for abstract diffusion perturbation models given by an It\^{o} process of dynamics $dY_t = b_t dt+ \beta_t dW_t$. Essentially, $\bar X$ is replaced by $Y$ and one considers the backward field
\[ F_t (x) = f (X^{t, x}_T) \equiv f (X_{T \leftarrow t} (x)) \]
parametrized by test functions $f$, similar to It\^o's formula (and natural to study the semigroup upon taking expectations). In \cite{HHJM24} it was then seen that
\begin{equation} \label{equ:HudIAG}
\begin{aligned}
  F_t (Y_t) - F_s (Y_s)  = & \int_s^t DF_r (Y_r)  (b_r - \mu (Y_r)) dr +
  \int_s^t DF_r (Y_r)  (\beta_r - \sigma (Y_r)) \diamond dW_r \\
  &  + \frac{1}{2}  \int_s^t D^2 F (Y_r)  ((\beta_r, \beta_r) - (\sigma
  (Y_r), \sigma (Y_r))) dr .
\end{aligned}
\end{equation}
We note that \eqref{equ:HudIAG} sits between \eqref{equ:Del1} and \eqref{equ:Del2}. 
We recover this in \cref{theorem:IAG} with an elegant rough path proof, also reducing integrability assumptions on the It\^o characteristics of $Y$, as was conjectured in \cite[Remark 3.2]{HHJM24}. Our proof also yields a reasonable understanding what structural features of $F_t$ are responsible for the validity of \eqref{equ:HudIAG}, with a view towards more general backward random fields, in the spirit of \eqref{equ:Del1}, but without Malliavin conditions on the It\^{o} integrand of $Y$.

\subsection{Comments on stochastic numerics} \label{sec:csn}

It is a quite challenging issue in stochastic numerics to establish strong or weak convergence rates
for numerical approximations of stochastic evolution equations
(SDEs and SPDEs of the evolutionary type) with non-globally monotone coefficients (cf.\ \cref{eq:monotonicity} below).
The difficulty of this issue is also illustrated by the fact that
there are several \emph{counterexample SDEs} with bounded and smooth
but non-globally monotone coefficient functions so that the SDE solution
can \emph{not} be approximated by any numerical approximation based on observations
of the driving noise with a rate of strong (weak) convergence in the literature
\cite{MR3305998,MR3538358,MR3609372}.
In general, it remains a fundamental open problem to provide sufficient (and necessary) conditions
on the coefficient functions of the SDE so that
the SDE solution can be approximated by implementable time-discrete numerical approximations
with a strictly positive polynomial of strong convergence.

In the situation of an SDE with globally monotone coefficient functions
(in which case the drift coefficient function satisfies a global one-sided Lipschitz condition;
see \cref{eq:monotonicity} below)
strong convergence rates for numerical approximations can be established by applying
It\^{o}'s formula to the square distance $ \| x - y \|^2 $, $ ( x, y ) \in \R^d \times \R^d $,
as the test function, by exploiting the global monotonicity condition in the sense that there
exists $ c \in \R $ such that for all $ x, y \in \R^d $, 
it holds that
\begin{equation}
\label{eq:monotonicity}
\textstyle
  \left< x - y, \mu(x) - \mu(y) \right>
  +
  \frac{ 1 }{ 2 } \left\| \sigma( x ) - \sigma( y ) \right\|_{ HS( \R^m, \R^d ) }^2
  \leq
  c \left\| x - y \right\|^2 ,
\end{equation}
and, thereafter, by applying
\emph{Gronwall's lemma} to the quantity $ \E\bigl[ \| X_t - Y_t \|^2 \bigr] $, $ t \in [0,\infty) $,
where $ ( X_t )_{ t \geq 0 } $ is the solution process of the SDE under consideration
and where $ ( Y_t )_{ t \geq 0 } $ is a time-continuous version
of a numerical approximation of the SDE solution.

In the situation of an SDE without globally monotone coefficient functions (where \cref{eq:monotonicity} is not fulfilled),
this chain of arguments is not working as the prerequisites of Gronwall's lemma are not satisfied anymore.
Specifically, this approach is not working anymore in the situation of SDEs with non-monotone coefficient functions
as there appear additional unbounded random quantities depending on $ X_t $ and $ Y_t $
inside the expectation in this case
so that Gronwall's lemma can not be applied to $ \E\bigl[ \| X_t - Y_t \|^2 \bigr] $, $ t \in [0,\infty) $, anymore.

This is precisely the situation where the \emph{It\^{o}--Alekseev--Gr\"{o}bner formula} \cite{HHJM24} can be brought into play
(cf.\ \cite[Section~4]{Huddeetal2020v1arXiv} and \cite{Hutzenthaleretal2020arXiv}).
We also refer to \cite{DELMORAL2022197} for a closely related variant of the IAG
and we refer to \cite{MaurerZurcher2025arXiv} for an extension of the IAG to the situation of Poisson noise.

In the situation of deterministic ODEs, the Alekseev-Gr\"{o}bner formula can also be used to estimate the
errors of numerical approximations of solutions of ODEs. However, in the deterministic case the Alekseev--Gr\"{o}bner formula
is not really powerful as the above outlined standard Gronwall argument also works in the situation of a non-globally monontone
but locally Lipschitz continuous coefficient function.

A related alternative approach is to use the perbutation approach in \cite{MR4079431} to establish
strong convergence rates for numerical approximations of SDEs with non-globally monotone coefficients.
This approach is based on multiplying with a negative exponential as an integrating factor.
However, with this approach -- loosely speaking -- stronger exponential integrability properties for the numerical approximations
and the SDE solution need to be established \cite{MR3766391} than in the situation of the IAG where
such strong exponential integrability properties basically only need to be established
for the solution of the SDE (cf., for example, \cite[Lemma~4.10]{MR2259251}, \cite[Lemma~2.5 and Remark~2.6]{MR3020951},
and \cite[Section~2.2]{MR4749000})
but not for the numerical approximation \cite{MR4190823}.

It should also be pointed out that the above sketched approaches are applicable to SDEs with possibly highly
degenerate noise coefficients (for example, additive noise that vanishes in certain directions).
In the case of non-degenerate noise strong convergence rates for numerical approximations of SDEs
have been be established in \cite{MR4583671} in a very wide generality. In particular, the above mentioned
counterexample SDEs \cite{MR3305998,MR3538358,MR3609372} all fall in the regime of degenerate noise.

It should also be pointed out that -- beyond the linear Black--Scholes model -- many of the popular SDE models
in the literature 
(cf., for example, \cite[Chapter~4]{MR3364862})
do have non-globally monotone coefficient functions, such as 
stochastic Lotka-Volterra models, 
the Heston stochastic volatility model, 
stochastic oscillator models, 
the Kardar–Parisi–Zhang (KPZ) equation, 
and 
stochastic Navier--Stokes equations, 
just to name a few. So, it is highly relevant to develop stochastic analysis techniques
that allow to establish optimal strong convergence speeds also in the situation
of SDEs with non-globally monotone coefficient functions.

\appendix

\section{Higher-order Kolmogorov criterion}\label{sec:besov-spaces}
\begin{theorem}
  \label{appendix:theorem:Kolmogorov}
  
  Let $V^{1, 1}, V^{1, 2}, V^2$ be (finite-dimensional) Banach spaces. Suppose there is a
  c\`{a}dl\`{a}g two-parameter process $A : \Delta_T \times
  \Omega \rightarrow (V^{1, 1} \otimes V^{1, 2}) \otimes V^2$ such that we
  have a two-parameter processes $A^1 : \Delta_T \times \Omega
  \rightarrow V^{1, 1} \otimes V^{1, 2}$ and three paths 
  $A^2 : [0, T] \times \Omega \rightarrow V^2, A^{1, i} : [0, T] \times \Omega
  \rightarrow V^{1, i}$ all also c\`{a}dl\`{a}g, such that for any
  $s, u, t\in [0,T]$ with $s<u<t$ it holds
  \[ \delta A_{s, u, t} = A^1_{s, u} \otimes \delta A^2_{u, t} ; \quad
     \delta A^1_{s, u, t} = \delta A^{1, 1}_{s, u} \otimes \delta A^{1, 2}_{u,
     t} \]
  Assume there is $q \geq 3$ and $\beta > \nicefrac{1}{q}$ such that
  \begin{equation}
\label{appendix:Kolmogorov:multiest}
\left.
  \begin{aligned}
     \| \delta A^{1, 1}_{s, t} \|_{L^q} \vee \| \delta A^{1, 2}_{s, t} \|_{L^q} \vee \| \delta A^2_{s, t}
    \|_{L^q} & \lesssim | t - s |^{\beta }; \\
     \| A^1_{s, t} \|_{L^{\frac{q}{2}}} & \lesssim | t - s |^{2 \beta};\\
     \| A_{s, t} \|_{L^{\frac{q}{3}}} & \lesssim | t - s |^{3 \beta}.
  \end{aligned}
  \right\}
  \end{equation}
  Then for all $\alpha \in [ 0, \beta - \nicefrac{1}{q} )$ there are
  modifications of $(A, A^1, A^{1, 1}, A^{1, 2}, A^2)$ denoted by $(\tilde{A},
  \tilde{A}^1  \tilde{A}^{1, 1}, \tilde{A}^{1, 2}, \tilde{A}^2)$ and
  $\tilde{K}^{(l)}_{\alpha} \in L^{\frac{q}{l}} $for $l = 1, 2, 3$ such that
  \[ | \tilde{A}_{s, t} |  \leq \widetilde{K }^{(3)}_{\alpha} | t - s |^{3
     \alpha} ; \qquad | \tilde{A}^1_{s, t} | \leq \tilde{K}^{(2)}_{\alpha} | t
     - s |^{2 \alpha} ; \qquad | \delta \tilde{A}^2_{s, t} | \vee \max_{j = 1,
     2} | \delta \tilde{A}_{s, t}^{1, j} | \leq \tilde{K}_{\alpha}^{(1)} | t -
     s |^{\alpha} \]
  for any $(s, t) \in \Delta_T$ almost surely.
  \end{theorem}
  \begin{proof}The proof is a straightforward generalization of the proof given in \cite[Theorem~3.1]{friz_2020}. We give it regardless for the convenience of
  the reader. W.l.o.g. take $T = 1$ and let $D_n $denote the set of integer
  multiples of $2^{- n}$ in $[0, 1)$. As in the usual criterion, it suffices
  to consider $s, t \in \bigcup_n D_n$ with the values at the remaining times
  filled in by continuity. Note that $| D_n |^{- 1} = 2^n$. Set
  \begin{eqnarray*}
    K_n^{(1)} & := & \sup_{s \in D_n } \left( | \delta A^{1, 1}_{s,
    s + 2^{- n}} | \vee | \delta A^{1, 2}_{s,
    s + 2^{- n}} | \vee | \delta A^2_{s, s + 2^{- n}} |\right); \nonumber\\
    K_n^{(2)} & := & \sup_{s \in D_n} | A^1_{s, s+ 2^{-n}} |;  \nonumber\\
    K^{(3) }_n & := & \sup_{s \in D_n} | A_{s, s+2^{-n}} |. \nonumber
  \end{eqnarray*}
  It follows from \eqref{appendix:Kolmogorov:multiest}, that
  \begin{eqnarray*}
    \| K^{(1)}_{n} \|_{L^q} \leq \left( \sum_{s \in D_n} \left \Vert  | \delta
    A_{s, s + 2^{- n}}^{1, 1}|^{q} \right \Vert_{L^1} + \left \Vert  | \delta
    A_{s, s + 2^{- n}}^{1, 2}|^{q} \right \Vert_{L^1} +\left\| \left| {\delta A^2_{s, s +
    2^{- n}}}  \right|^q \right\|_{L^1} \right)^{\frac{1}{q}} & \lesssim & |
    D_n |^{\beta  - \frac{1}{q}}; \nonumber\\
    \| K^{(2)}_n \|_{L^{\frac{q}{2}}} \leq \left( \sum_{s \in D_n} \left\| \left|
    {A^1_{s, s + 2^{- n}}}  \right|^{\frac{q}{2}} \right\|_{L^1}
    \right)^{\frac{2}{q}} & \lesssim & | D_n |^{\beta - \frac{2}{q}};
    \nonumber\\
    \| K^{(3)}_n \|_{L^{\frac{q}{3}}} \leq \left( \sum_{s \in D_n} \left\| | A_{s,
    s + 2^{- n}}  |^{\frac{q}{3}} \right\|_{L^1} \right)^{\frac{3}{q}} &
    \lesssim & | D_n |^{\beta  - \frac{3}{q}}. \nonumber
  \end{eqnarray*}
  Now fix $s, t \in \bigcup_n D_n$ with $s<t$ and choose $m$ such that $| D_{m + 1} | <
  t - s \leq | D_m |$. Now note that the intervall $[s, t)$ can be expressed
  as the finite disjoint union of the form $[u, v) \in D_n$ with $n = m + 1$
  and where no three intervals have the same length. We denote this partition
  by $ s = t_0 < \ldots . < t_N = t $.
  Then it holds
  \begin{equation}
  \label{appendix:eq:Kolmogorov_1}
  \max_{j = 1, 2} | \delta A^{1, j}_{s, t^{}} | \vee | \delta A^2_{s, t} |
     \leq \max_{0 \leq i \leq N} \max_{j = 1, 2} | \delta A^{1, j}_{s, t_i^{}}
     | \vee | \delta A^2_{s, t_i} | \leq 2 \sum_{n \geq m + 1} K_n^{(1)}. 
  \end{equation}  
  Further we have
  \begin{equation}
  \label{appendix:eq:Kolmogorov_2}
  \begin{aligned}
  | A^1_{s, t} | \leq \max_{0 \leq i \leq N} | A^1_{s, t_i} | &\leq |
     A^1_{s, t_1} | + \sum_{i = 1}^{N - 1} | A^1_{t_i, t_{i + 1}}+ \delta A^1_{s, t_i, t_{i + 1}}   | \\
     &\leq 2 \sum_{n \geq m + 1} K_n^{(2)} + \sum_{i =
     1}^{N - 1} | \delta A^1_{s, t_{i}, t_{i + 1}} |,  
     \end{aligned}
     \end{equation}
     noting that the last term may be bounded by
  \begin{equation}
  \label{appendix:eq:Kolmogorov_3}
  \begin{aligned}
  \sum_{i = 1}^{N - 2} | \delta A^1_{s, t_{i + 1}, t_{i + 2}} | &= \sum_{i =
     1}^{N - 2} \left| \delta A^{1, 1}_{s, t_i} \otimes \delta A^{1, 2}_{t_{i,
     t_{i + 1}}} \right| \leq \sum_{i = 1}^{N - 2} \left(  \sum_{k = 0}^{i - 1} | \delta A^{1,
     1}_{t_k, t_{k + 1} } | \right) \left| \delta A^{1,
     2}_{t_{i, t_{i + 1}}} \right| \\
     &\leq \left( \sum_{i=0}^{N-2} |\delta A_{t_{k}, t_{k+1}}^{1,1}| \right) \left(\sum_{i = 1}^{N - 2} \left| \delta A^{1,
     2}_{t_{i, t_{i + 1}}} \right| \right) \leq 4 \left( \sum_{n \geq m + 1}
     K_n^{(1)} \right)^2.  
  \end{aligned}
  \end{equation} 
  Analogously we have by combining \eqref{appendix:eq:Kolmogorov_1}, \eqref{appendix:eq:Kolmogorov_2} and \eqref{appendix:eq:Kolmogorov_3}
  \begin{equation} 
  \label{appendix:eq:Kolmogorov_4}
  \begin{aligned}
    | A_{s, t} | & =  \left| \sum_{i = 1}^{N - 1} A_{t_i, t_{i + 1}} + A^1_{s
    , t_i} \otimes \delta A^2_{t_i, t_{i + 1}} \right| \\
    & \leq  \sum_{i = 1}^{N - 1} | A_{t_i, t_{i + 1}} | + \max_{j=1, \dots,  N-1} | A^1_{s,
    t_j} | \sum_{i=0}^{N-1} | \delta A^2_{t_i, t_{i + 1}} | \\
    & \leq  2 \sum_{n \geq m + 1} K^{(3)}_n +  \left( 2 \sum_{n \geq m + 1} K_n^{(2)} + 4 \left(
    \sum_{n \geq m + 1} K_n^{(1)} \right)^2 \right) \left( 2 \sum_{n \geq m + 1}
    K_n^{(1)} \right) \\
    & =  2 \sum_{n \geq m + 1} K^{(3)}_n + 4 \left(  \sum_{n \geq m + 1}
    K_n^{(1)} \right) \left( \sum_{n \geq m + 1} K_n^{(2)} \right) + 8 \left(
     \sum_{n \geq m + 1} K_n^{(1)} \right)^3.  
  \end{aligned}
  \end{equation}
  Now note that for any $\alpha \in [ 0, \beta - \nicefrac{1}{q} )$ it
  holds by \eqref{appendix:eq:Kolmogorov_1}:
  \[ \frac{| \delta A^{1, 1}_{s, t}| \vee | \delta A^{1, 2}_{s, t}| \vee | \delta A^2_{s,
     t} |}{| t - s |^{\alpha}} \leq 2 \sum_{n \geq m + 1} \frac{K_n^{(1)}}{|
     D_{m + 1} |^{\alpha}} \leq 2 \sum_{n \geq m + 1} \frac{K_n^{(1)}}{| D_n
     |^{\alpha}} \leq  2 \sum_{n \geq 0} \frac{K_n^{(1)}}{| D_n
     |^{\alpha}}\eqqcolon K^{(1)}_{\alpha} \]
  with $K^{(1)}_{\alpha} \in L^q$ as
  \[ \| K^{(1) }_{\alpha} \|_{L^q} \leq 2 \sum_{n \geq 0} | D_n |^{\beta -
     \frac{1}{q} - \alpha} < \infty. \]
  Analogously by \eqref{appendix:eq:Kolmogorov_2} we have
  \[ \frac{| A^1_{s, t} |}{| t - s |^{2 \alpha}} \leq 4 \left( \sum_{n \geq m
     + 1} \frac{K_n^{(1)}}{| D_n |^{\alpha}} \right)^2 + 2 \sum_{n \geq m + 1}
     \frac{K_n^{(2)}}{| D_n |^{2 \alpha}} \leq 4 \left( \sum_{n \geq 0} \frac{K_n^{(1)}}{| D_n |^{\alpha}} \right)^2 + 2 \sum_{n \geq 0}
     \frac{K_n^{(2)}}{| D_n |^{2 \alpha}} \eqqcolon K^{(2) }_{\alpha}. \]
with $K^{(2)}\in L^{\frac{q}{2}}$ as 
\[ \| K^{(2)}_{\alpha} \|_{L^{\frac{q}{2}}} \lesssim \left( \left\| \sum_{n
     \geq 0} \frac{K_n^{(1)}}{| D_n |^{\alpha}} \right\|_{L^q} \right)^2 +
     | D_n |^{2 \beta - \frac{2}{q} - 2 \alpha} \lesssim \sum_{n \geq 0}  |
     D_n |^{2 \beta - \frac{2}{q} - 2 \alpha} < \infty. \]
  At last by \eqref{appendix:eq:Kolmogorov_4} we obtain 
  \begin{equation*} 
  \begin{aligned}
  \frac{| A_{s, t} |}{| t - s |^{3 \alpha}} &\leq 2 \sum_{n \geq m + 1}
     \frac{K^{(3)}_n}{| D_n |^{3 \alpha}} + \left( 2 \sum_{n \geq m + 1}
     \frac{K_n^{(1)}}{| D_n |^{\alpha}} \right) \left( 2 \sum_{n \geq m + 1}
     \frac{K_n^{(2)}}{| D_n |^{2 \alpha}} \right) + \left( 2 \sum_{n \geq m +
     1} \frac{K_n^{(1)}}{| D_n |^{\alpha}} \right)^3 \\
     &\leq 2 \sum_{n \geq 0}
     \frac{K^{(3)}_n}{| D_n |^{3 \alpha}} + \left( 2 \sum_{n \geq 0}
     \frac{K_n^{(1)}}{| D_n |^{\alpha}} \right) \left( 2 \sum_{n \geq 0}
     \frac{K_n^{(2)}}{| D_n |^{2 \alpha}} \right) + \left( 2 \sum_{n \geq 0} \frac{K_n^{(1)}}{| D_n |^{\alpha}} \right)^3 
     =: K^{(3)}_{\alpha}.  
   \end{aligned}
   \end{equation*}
  with $K^{(3)}_{\alpha} \in L^{\frac{q}{3}}$ as by applying the Hölder-Inequality we get
  \begin{equation*}  
         \| K^{(3)}_{\alpha} \|_{L^{\frac{q}{3}}} \lesssim  \|
       K^{(1)}_{\alpha} \|_{L^q}^3 + \| K^{(1)}_{\alpha} \|_{L^q} \|
       K^{(2)}_{\alpha} \|_{L^{\frac{q}{2}}} + \left\| \sum_{n \geq 0}
       \frac{K^{(3)}_n}{| D_n |^{3 \alpha}} \right\|_{L \frac{q}{3}} \lesssim
       \sum_{n \geq 0} | D_n |^{3 \beta - \frac{3}{q} - 3 \alpha} < \infty.
  \end{equation*}   
  concluding the proof.
  \end{proof}
  \begin{remark}
  \label{appendix:Kolmogorov_transpose_remark}
    Note that an analogous claim for the case
    \[ \delta B_{s, u, t} = \delta B^1_{s, u} \otimes B^2_{u, t} ; \quad B^2_{s, u, t} = \delta B^{2, 1}_{s, u} \otimes \delta
       B^{2, 2}_{u, t} \]
    with $B^1 : [0, T] \times \Omega \rightarrow V^1$ and $B^2 :
    \Delta_T \times \Omega \rightarrow V^{2, 1} \otimes V^{2, 2}$ holds
    as well, by considering the transpose $A_{s,t}\coloneqq B_{s,t}^{\top}$ in \cref{appendix:theorem:Kolmogorov}. 
\end{remark}
The higher-order Kolmogorov criterion is used in the proof of \cref{prop:eq:SCRPtoRIP} to obtain suitable path-regularity of certain iterated integrals, which we also consider in the following example.
\begin{example} \label{Iterated_Martingale_Inequality} Suppose $\mathbf{X} = (X,
  \mathbb{X}) \in \mathscr{C}^{\alpha}([0,T]; \mathbb{R}^{d_{X}})$ with $\beta \in \left( \nicefrac{1}{3},
  \nicefrac{1}{2} \right)$ and for some $q \geq 3$ such that $3
  \beta - \nicefrac{1}{q} > 1$ consider $(Y, Y') \in \mathscr{D}^{2 \alpha}_X\left([0,T]; L^q(\mathcal{L}(\mathbb{R}^{d_{X}}; \mathbb{R}^{d_{Y}})\right)$ i.e.
  \[ \sup_{s < t} \frac{\| \delta Y_{s, t} - Y_s' \delta X_{s, t} \|_{L^q}}{|
     t - s |^{2 \alpha}} < \infty ; \hspace{3em} \sup_{s < t} \frac{\| \delta 
     Y_{s, t} \|_{L^{q}} \vee \| \delta Y'_{s, t} \|_{L^{q}}}{| t - s |^{\alpha}} < \infty \]
  Now let $M \in \mathcal{M}^{q, c, 1}$ be a
  continuous, $q$-integrable martingale with Lipschitz-continuous bracket
  process. Thus the Itô integral
  $ \Pi (Y ; M)_{s, t} = \int_s^t \delta Y_{s, r} dM_{r}$
  is well-defined and in $L^{q}$.  We further define for $R^Y_{s, t} :=
  \delta Y_{s, t} - Y_s' \delta X_{s, t}$, the iterated integral
  \[ \Pi (R ; M)_{s, t} := \int_s^t R_{s, r}^Y
    \otimes  dM_r,  \]
  which is defined in the sense of the IBP-integrals from \cref{lemma:ibp} by noting that $[s,T]\ni t\mapsto R_{s,t}^{Y}$ is an $(\mathfrak{F}_{t})_{t\in [s,T]}$- local martingale.    
  Immediately one can verifiy the following algebraic relations for any $s,u, t\in [0,T]$ with $s< u<t$ almost surely
  \[ \delta \Pi (R ; M)_{s, u, t} = - (\delta Y_{s, u}) \otimes \Pi (X ;
     M)_{u, t} + R^Y_{s, u} \otimes \delta M_{u, t} \]
  Hence by applying \cref{appendix:theorem:Kolmogorov} and \cref{appendix:Kolmogorov_transpose_remark}, we obtain
  almost surely
  \[ \left| \int_s^t Y_r \ensuremath{\operatorname{dM}}_r - Y_s' \Pi (X ;
     M)_{s, t} \right| = \left| \int_s^t (\delta Y_{s, r} - Y_s' \delta X_{s,
     r}) dM_r \right|=\big| \Pi(R; M)_{s,t} \big| \approx O (| t - s |^{3
     \alpha}) \]
  for any $\alpha \in (0, 3 \beta - \nicefrac{1}{q})$.
\end{example}
\section{On solution flows for RDEs along non-autonomous vector fields}
\label{appendix:section:On RDE-flows for non-autonomous vector fields}
In this section, we investigate the well-posedness and smoothness of RDEs along non-autonomous, controlled vector fields. Let $V, W, U$ denote some finite-dimensional Banach spaces. Here,  $\alpha \in (\nicefrac{1}{3}, \nicefrac{1}{2}]$ and $\mathbf{X}=(X, \mathbb{X})\in \mathscr{C}^{\alpha}([0,T]; V)$. Recall (see, e.g., \cite[Chapter~4]{friz_2020}) that $(\mathscr{D}^{2\alpha}_{X}([0,T]; W), \Vert \cdot \Vert_{\mathscr{D}^{2\alpha}([0,T]; W)})$ is a Banach space with
\begin{equation*}
    \Vert Y; Y' \Vert_{\mathscr{D}^{2\alpha}([0,T]; W)}\coloneqq |Y_{0}|+ |Y'_{0}| + [Y, Y']_{X; 2},
\end{equation*}
where $[Y,Y']_{X;2}$ is given in \eqref{e:semi-norm}. 
\begin{definition}
  Let $\mathfrak{K}\subset W$ be closed. We call the $3$-tuple $\mathcal{F}=(f, f', \partial f)$ of functions
  \begin{equation*} 
  \begin{aligned}
      \mathcal{F}: [0,T] \times \mathfrak{K} &\to U \times \mathcal{L} (V ; U)
     \times \mathcal{L} (W ; U)\\
     (t,x)  &\mapsto (f_{t}(x) , f'_{t}(x), \partial f_{t}(x))  
  \end{aligned}
  \end{equation*}
  an $X$-{\em controlled} 
  $\Lip^2$ 
  {\em field}, on $[0,T]\times \mathfrak{K}$, if it holds on $[0,T]\times \mathfrak{K}$,
  \begin{eqnarray*}
  f_{t_{1}}(x_{1})&\stackrel{2}{=}& f_{t_{0}}(x_{0}) + f'_{t_{0}}(x_{0}) \delta X_{t_{0}, t_{1}}+  \partial f_{t_{0}}(x_{0}) (x_{1}-x_{0})\\
  G_{t_{1}}(x_{1}) &\stackrel{1}{=}& G_{t_{0}}(x_{0}) \quad G\in  \{f', \partial f \},
  \end{eqnarray*}
  where ``$\stackrel{k}{=}$'' is defined in~\eqref{e:k=}.    We denote the space of such $X$-controlled fields  by $\mathscr{D}^{2 \alpha}_{X}\Lip^{2}_{x}(\mathfrak{K}; U)$. If $\mathcal{F}\in \mathscr{D}^{2\alpha}_{X}\Lip^{2}_{x}(\mathfrak{K}; U)$ for any compact subset $\mathfrak{K}\subset W$, we say $\mathcal F\in \mathscr{D}^{2 \alpha}_{X}\Lip^{2}_{x, \loc}(W; U)$.
\end{definition}
Naturally, this can be extended to higher orders of spatial generality. 
\begin{definition}
  Let $\mathfrak{K}\subset W$ be closed. We call the $(2k-1)$-tuple \\
  $\mathcal{F}=(f, f', \partial f, \partial f', \dots, \partial^{k-2} f, \partial^{k-2} f', \partial^{k-1} f)$ of functions
  \begin{equation*} 
  \begin{aligned}
      (\partial^{j} f , \partial^{j} f', \partial^{j+1} f): [0,T] \times \mathfrak{K} &\to \mathcal{L} (W^{\otimes j} ; U)
     \times \mathcal{L} (W^{\otimes j}\otimes V ; U)\times \mathcal{L}(W^{\otimes j+1}; U)\\
     (t,x)  &\mapsto (\partial^{j} f_{t}(x) , \partial^{j} f'_{t}(x), \partial^{j+1} \partial f_{t}(x))  
  \end{aligned}
  \end{equation*}
  for any $j=0, \dots, k-2$,  an $X$-{\em controlled} 
  $\Lip^k$ 
  {\em field}, on $[0,T]\times \mathfrak{K}$ if 
  \begin{enumerate}
  \item It holds
  \begin{equation*}
      \sup_{t \in [0,T]}[f_{t}, \partial f_{t},\dots, \partial^{k-1} f_{t}]_{\Lip^{k}}< \infty; \quad \sup_{t\in [0,T]} [f'_{t}, \partial f'_{t}, \dots, \partial^{k-2} f'_{t}]_{\Lip^{k-1}} < 
  \infty
  \end{equation*}
      \item $(\partial^{j}f , (\partial^{j} f')^{\top}, \partial^{j+1} f)\in \mathscr{D}^{2\alpha}\Lip^{2}_{x}(\mathfrak{K}; \mathcal{L}(W^{\otimes j}; U))$ for any $j=0, \dots, k-2$.
  \end{enumerate}
   We denote the space of such $X$-controlled fields  by $\mathscr{D}^{2 \alpha}_{X}\Lip^{k}_{x}(\mathfrak{K}; U)$. If $\mathcal{F}\in \mathscr{D}^{2\alpha}_{X}\Lip^{k}_{x}(\mathfrak{K}; U)$ for any compact subset $\mathfrak{K}\subset W$, we say $\mathcal F\in \mathscr{D}^{2 \alpha}_{X}\Lip^{k}_{x, \loc}(W; U)$.
\end{definition}
Analogously to \cref{cor:RIW_formula}, such controlled fields satisfy a natural composition rule with controlled rough paths. The smoothness of this composition operator can be further quantified.
\begin{lemma}
\label{appendix:Frechet_lemma}
    Let $\mathcal{F}=(f, f', \partial f, \partial f', \dots, \partial^{k+1} f)\in \mathscr{D}^{2\alpha}_{X}\Lip^{k+2}_{x}(W; U)$. Then 
    \begin{equation*}
    \begin{aligned}
        A^{\mathcal{F}}: \mathscr{D}^{2\alpha}_{X}([0,T]; W)&\to \mathscr{D}^{2\alpha}_{X}([0,T]; U)\\
        (Y, Y') &\mapsto (f(Y), \partial f(Y) Y'+ f'(Y))
    \end{aligned}
    \end{equation*}
    is a well-defined operator, which is $k$-times Fréchet differentiable with derivatives given by, for $(Z^{i}, (Z^{i})')\in \mathscr{D}^{2\alpha}_{X}([0,T]; W)$ with $i=1,\dots, j$, 
    \begin{equation*}
    \begin{aligned}
        &D^{j}A^{\mathcal{F}}(Y, Y')[(Z^{1}, (Z^{1})')\dots, (Z^{j}, (Z^{j})')]
        \\
        &= \Big(\partial^{j} f(Y) (Z^{1},\dots ,Z^{j}), ~ \partial^{j+1} f(Y) (Y', Z^1, \dots, Z^{j}) + \left(\partial^{j} f'(Y) (Z^1, \dots, Z^{j})\right)^{\top} \\
        &\qquad + \partial^{j} f(Y) \left( Z^{1} , \dots, Z^{j}\right)'\Big),
    \end{aligned}
    \end{equation*}
    for each $j=1, \dots, k$, where we used the short-hand notation
    \begin{equation*}
        \left( Z^{1} , \dots, Z^{j}\right)'\coloneqq \sum_{i=1}^{j}\left(Z^1,   \dots, Z^{i-1}, (Z^i)', Z^{i+1}, \dots, Z^j\right). 
    \end{equation*}
\end{lemma}
\begin{proof}
Let $j=1, \dots, k$ and for any $i=1, \dots, j$, $(Z^{i}, (Z^{i})')\in \mathscr{D}^{2\alpha}_{X}([0,T]; W)$.  Then it holds that 
\begin{equation*}
    \begin{aligned}
         &\Big[ \partial^{j} f(Y+Z^{j})(Y'+ (Z^{j})', Z^{1}, \dots, Z^{j-1})- \partial^{j} f(Y)(Y', Z^{1}, \dots, Z^{j-1}) \\
         & \quad +(\partial^{j-1} f'(Y+Z^{j})(Z^{1}, \dots, Z^{j-1}))^{\top}- (\partial^{j-1} f'(Y)(Z^{1}, \dots, Z^{j-1}))^{\top} \\
         &\quad  + \partial^{j-1} f(Y+ Z^{j})(Z^{1}, \dots, Z^{j-1})' - \partial^{j-1} f(Y)(Z^{1}, \dots, Z^{j-1})'  \\
         &\quad  - \left(\partial^{j+1} f(Y) (Y', Z^1, \dots, Z^{j}) + \left(\partial^{j} f'(Y) (Z^1, \dots, Z^{j})\right)^{\top} \right.  \\
        &\quad\qquad +  \left.  \partial^{j} f(Y) \left( Z^{1} , \dots, Z^{j}\right)' \right)\Big]_{1} \lesssim [Z^{j}]_{1}^{2} \lesssim \Vert Z^{j}; (Z^{j})'\Vert_{\mathscr{D}^{2\alpha}}^{2}.
         \end{aligned}
\end{equation*}
The $2\alpha$-estimates for the remainder terms follow analogously. Recall that $\partial^{j} f(y)$ is $j$-symmetric (see \cref{def:Stein_Lip_def}) and therefore we have shown 
\begin{equation*}
\begin{aligned}
    &\Big \Vert D^{j-1} A^{\mathcal{F}}\left( (Y, Y')+ (Z^{j}, (Z^{j})')\right)[(Z^{1}, (Z^{1})'), \dots, (Z^{j-1}, (Z^{j-1})')]\\
    &\quad - D^{j-1} A^{\mathcal{F}}(Y, Y')[(Z^{1}, (Z^{1})'), \dots, (Z^{j-1}, (Z^{j-1})')] \\
    &\quad -  D^{j}A^{\mathcal{F}}(Y, Y')[(Z^{1}, (Z^{1})'), \dots, (Z^{j}, (Z^{j})')]\Big \Vert_{\mathscr{D}^{2\alpha}} \lesssim \Vert Z^{j}; (Z^{j})' \Vert_{\mathscr{D}^{2\alpha}}^{2}.
\end{aligned}
\end{equation*}
Hence, $D^{j}A^{\mathcal{F}}$ as described above indeed are the claimed Fréchet derivatives.
 At last, note that
\begin{equation*}
   D^{j}A^{\mathcal{F}}(Y, Y')[(Z^{1}, (Z^{1})'),\dots, (Z^{j}, (Z^{j})')] \in \mathscr{D}^{2\alpha}_{X}([0,T]; U)
\end{equation*}
    follows directly from the classical composition rules for controlled rough paths; see \cite[Chapter 7]{friz_2020}.
\end{proof}
For $\alpha \in (\nicefrac{1}{3}, \nicefrac{1}{2}]$, we define the mapping
\begin{equation*}
\begin{aligned}
    \mathcal{I}: \mathscr{D}^{2\alpha}_{X}([0,T]; \mathcal{L}(V; W)) &\to \mathscr{D}_X^{2\alpha}([0,T]; W)\\
    \quad (Y, Y')&\mapsto \left (\int_{0}^{\cdot} (Y, Y')_{s} d\mathbf{X}_{s}, Y \right). 
    \end{aligned}
\end{equation*}
By, e.g., \cite[Theorem 4.10]{friz_2020}, $\mathcal{I}$ is a bounded linear operator. 
Equipped with this result, we are now in a position to establish the following claim.
\begin{theorem}
\label{thm:appendix_smoothness_flow}
    Let $\alpha\in (\nicefrac{1}{3}, \nicefrac{1}{2}), k \in \mathbb{N}, \mathbf{X} \in \mathscr{C}^{\alpha}([0,T]; V)$ and $\mathcal{F}=(f, f', \partial f, \partial f', \dots, \partial^{k+1} f)\in \mathscr{D}_{X}^{2\alpha}\Lip^{k+2}_{x}(W; \mathcal{L}(V; W))$. Then for any $\xi\in W$ there exists $0< T_{0}\leq T$ and a unique $(Y^{\xi}, (Y^{\xi})')\in \mathscr{D}^{2\alpha}_{X}([0,T_{0}]; W)$ with $(Y^{\xi})'= f(Y^{\xi})$ and 
    \begin{equation*}
        Y_{t}^{\xi}= \xi+ \int_{0}^{t} \left(f_{s}(Y_{s}^{\xi}), \partial f_{s}(Y_{s}^{\xi}) f(Y_{s}^{\xi}) + f'_{s}(Y_{s}^{\xi})\right) d\mathbf{X}_{s} =\xi+ \int_{0}^{t} \left(A^{\mathcal F}(Y^\xi, (Y^\xi)')\right)_s d\mathbf{X}_{s},
    \end{equation*}
   such that for the solution flow 
\begin{equation}\label{e:solution-flow}
    \phi_{t}(\xi)\coloneqq (Y^{\xi}_{t}, f_{t}(Y_{t}^{\xi})),
\end{equation}
    it holds $\phi \in C^{k}(W; \mathscr{D}_X^{2\alpha}([0,T_{0}]; W))$. In particular  $D\phi_{t}(\xi)= (D Y^{\xi}_{t}, D(f(Y^{\xi}_{t})))$ is the unique solution to 
   \begin{equation*}
   D \phi_t(\xi) = \left( \operatorname{Id}  + \int_{0}^{t} D\left(A^{\mathcal{F}}(\phi(\xi))\right)_{s}d\mathbf{X}_{s}, Df(Y_t^{\xi}) DY_t^{\xi}\right).
   \end{equation*}
   
   If further one of the following conditions holds:
   \begin{enumerate}
       \item for any $t\in [0,T]$, $f_{t}\in \mathcal{C}_{b}(W; \mathcal{L}(V; W))$ and $f'_{t} \in \mathcal{C}_{b}(W; \mathcal{L}(V^{\otimes 2}; W))$;
       \item for any $t\in [0,T]$, $(f_{t}, f'_{t})\in \mathcal{L}(W; \mathcal{L}(V; W)\times \mathcal{L}(V^{\otimes 2}; W))$ (linear rough vector fields);
   \end{enumerate}
   then one can take $T_{0}=T$ in the above.
\end{theorem}
\begin{proof}  The existence and uniqueness of the solution $(Y^{\xi}, (Y^{\xi})')$ on $[0,T_{0}]$ with $T_{0}$ as described above is an easy generalization of the proof of \cite[Theorem 8.3]{friz_2020}, which we leave to the reader. From now let $k \geq 2$. To see the smoothness of the flow, we define the map
    \begin{equation*}
    \begin{aligned}
    \Psi: \mathscr{D}^{2\alpha}_{X}([0,T_{0}]; W) \times W &\to \mathscr{D}_X^{2\alpha}([0,T_{0}], W)\\ 
    \left((Z, Z'), \xi\right) &\mapsto \left( Z- \xi- \int_{0}^{\cdot} (A^{\mathcal{F}}(Z, Z'))_{s} d\mathbf{X}_{s}, Z'- f(Z)\right)\\
    &\qquad = (Z-\xi, Z')- (\mathcal{I} \circ A^{\mathcal{F}}) (Z, Z').
    \end{aligned}
    \end{equation*}
    By the uniqueness of the solution, it holds that $\Psi((Z, Z'), \xi)=0$ iff $(Z, Z')=(Y^{\xi}, f(Y^{\xi}))$. Now by 
    \cref{appendix:Frechet_lemma}, we know that the operator $\mathcal{I}\circ A^{\mathcal{F}}$ is $k$-times Fréchet differentiable with derivatives $D^{j} (\mathcal{I}\circ A^{\mathcal{F}})[H, H']=\mathcal{I}(D^{j}A^{\mathcal{F}}[H, H'])$ for any $j=1, \dots, k$ and $(H, H')\in \mathscr{D}^{2\alpha}_{X}([0,T_{0}]; W)$. Denoting by $\mathbf{D}$  the Fréchet derivative of $\Psi$ restricted on $\mathscr{D}^{2\alpha}_{X}$, we see 
    \begin{equation*}
        \mathbf{D}\Psi((Z, Z'), \xi)[H, H']= (H, H')- \mathcal{I}\left(DA^{\mathcal{F}}(Z, Z')[H, H']\right).
    \end{equation*}
    Note that for each fixed $(Z,Z')$, the operator $\mathbf D \Psi$ is invertible iff for any $(K, K')\in \mathscr{D}^{2\alpha}_{X}([0,T_{0}]; W)$, there is a unique $(H, H')\in \mathscr{D}^{2\alpha}_{X}([0,T_{0}]; W)$ such that 
    \begin{equation}\label{e:S}
        \mathcal{S}(H, H')\coloneqq (K,K') + \mathcal{I}\left(DA^{\mathcal{F}}(Z, Z')[H, H']\right)=(H, H').
    \end{equation}
 Clearly, every fixed point of \eqref{e:S} is  an element of the affine subspace
    \begin{equation*}
    \mathscr{D}^{2\alpha}_{X}([0,T_{0}]; (K_{0}, K'_{0}))\coloneqq \left \{(Y, Y')\in \mathscr{D}^{2\alpha}_{X}([0,T_{0}])\big| Y_{0}=K_{0}; \;  Y'_{0}=K'_{0}+ \partial f_0(Z_0)K_{0} \right \}, 
    \end{equation*}
which is a complete metric space invariant under $\mathcal{S}$.  Given $\theta\in(0,1)$, \cite[Theorem 4.10]{friz_2020} ensures the existence of  a sufficiently small $\tilde T_0>0$,  independent of the initial condition $(K_{0}, K'_{0})$, such that  for any $(H_{1}, H'_{1}), (H_{2}, H'_{2})\in \mathscr{D}^{2\alpha}_{X}([0,T_{0}]; (K_{0}, K'_{0}))$,
    \begin{equation*}
    \begin{aligned}
       &\left \Vert \mathcal{S}(H_{1}, H_{1}')- \mathcal{S}(H_{2}, H_{2}') \right \Vert_{\mathscr{D}^{2}_{X}([0,\tilde{T_{0}}])}= \left \Vert \mathcal{S}(H_{1}, H_{1}')- \mathcal{S}(H_{2}, H_{2}') \right \Vert_{2; X; [0,\tilde{T}_{0}]} \\
       &=\left \Vert \mathcal{I}(DA^{\mathcal{F}}(Z, Z')[H_{1}-H_{2}, H'_{1}- H'_{2}]); \partial f(Z)(H_{1}-H_{2}) \right \Vert_{2; X; [0,\tilde{T}_{0}]}\\
       &< \theta \left \Vert H_{1}- H_{2}; H_{1}'- H_{2}'\right\Vert_{2; X; [0,\tilde{T}_{0}]} = \left \Vert H_{1}- H_{2}; H_{1}'- H_{2}'\right\Vert_{\mathscr{D}^{2\alpha}_{X}([0,\tilde{T_{0}}])},
    \end{aligned}
    \end{equation*}
    where $[\cdot, \cdot]_{2;X;[0,\tilde T_0]}$ is as defined in \eqref{e:semi-norm} but on the sub-intervall $[0,\tilde T_{0}]$.
    Therefore, $\mathcal{S}$ is a contraction  on  $\mathscr{D}^{2\alpha}_{X}([0,\tilde{T}_{0}]; (K_{0}, K'_{0}))$, yielding the existence and uniqueness of the fixed point of \eqref{e:S}.  By iterating this argument over $[(n \tilde{T}_{0})\wedge T_{0}, ((n+1)\tilde{T}_{0})\wedge T_{0}]$ for $n\in \mathbb{N}$, we obtain the existence of a unique fixed point $(H, H')\in \mathscr{D}^{2\alpha}_{X}([0,T_{0}]; W)$ such that $\mathcal{S}(H, H')=(H, H')$. Thus, $\mathbf{D}\Psi((Z, Z'), \xi)$ is invertible,  and then  by the \textit{implicit function theorem} (see \cite[Theorem 19.28]{Driver2003}) we see that $\mathbb{R}^{d_{Y}}\ni \xi \mapsto \phi(\xi)=(Y^\xi, (Y^\xi)')\in \mathscr{D}^{2\alpha}_{X}([0,T_{0}]; W)$ is $k$-times 
    continuously Fréchet differentiable with 
    \begin{equation*}
        D\phi(\xi)= \left(\operatorname{Id} + \mathcal{I}(DA^{\mathcal{F}}(\phi(\xi))[D\phi(\xi)]), \partial f(Y^{\xi})DY^{\xi}\right).
    \end{equation*}
\end{proof}
We now establish the homeomorphism property for the solution flow.  The injectivity follows directly from the following elementary property of strongly controlled rough paths.
\begin{lemma}
\label{lemma:appendix:backwards_davie}
    Let $\alpha \in (\nicefrac{1}{3}, \nicefrac{1}{2}]$ and $\mathbf{X}\in \mathscr{C}^{\alpha}([0,T]; V)$. Let $(Y, Y', Y'', 
    \dot{Y}), (\tilde{Y}, \tilde{Y}', \tilde{Y}'', \dot{\tilde{Y}})\in \mathscr{D}^{3\alpha}_{\mathbf{X}}([0,T]; W)$ be such that there is a constant $L>0$ satisfying,  for all $t \in [0,T]$,
    \begin{equation*}
        |Y'_t- \tilde{Y}'_{t}|+ |Y''_{t}- \tilde{Y}''_{t}|+ |\dot{Y}_{t}- \dot{\tilde{Y}}_{t}|\leq L |Y_{t}- \tilde{Y}_{t}| . 
    \end{equation*}
    Assume moreover that  there exists $t_{0}\in [0,T]$ such that $Y_{t_{0}}= \tilde{Y}_{t_{0}}$. Then $(Y, Y', Y'', 
    \dot{Y})=(\tilde{Y}, \tilde{Y}', \tilde{Y}'', \dot{\tilde{Y}})$ on $[0,T]$.
\end{lemma}
\begin{proof}
    Let $T>1=t_{0}$ w.l.o.g.\ and denote by $\mathcal{P}^{n}=(t_{i}^{n})_{i=0}^{2^{n}}$ the $n$th-level dyadic partition of $[0,1]$. Then it holds (see \cite[Proposition 5.12]{friz_2020}) that for $(s,t)\in \Delta_{T}$
    \begin{equation*}
        \delta (Y-\tilde{Y})_{s,t}\stackrel{3}{=} (Y'-\tilde{Y}')_{t} \delta X_{s,t}+ (Y''-\tilde{Y}'')_{t}(\mathbb{X}_{s,t}- \delta X_{s,t}\otimes \delta X_{s,t})+ (\dot{Y}-\dot{\tilde{Y}})_{t} (t-s)
    \end{equation*}
    Therefore for any $s \in [t^{n}_{2^{n}-1}, 1]$ it holds
    \begin{equation*}
        |Y'_s- \tilde{Y}'_{s}|+ |Y''_{s}- \tilde{Y}''_{s}|+ |\dot{Y}_{s}- \dot{\tilde{Y}}_{s}|\lesssim |Y_{s}- \tilde{Y}_{s}|=\mathcal{O}(2^{-3\alpha n}).
    \end{equation*}
    Iteratively we see that for any $i=0, \dots, 2^{n}-1$ it holds
    \begin{equation*}
        |\delta(Y- \tilde{Y})_{t_{i}^{n}, t_{i+1}^{n}}|+ |\delta(Y'- \tilde{Y}')_{t_{i}^{n}, t_{i+1}^{n}}|+ |\delta(Y''- \tilde{Y}'')_{t_{i}^{n}, t_{i+1}^{n}}|+ |\delta(\dot{Y}- \dot{\tilde{Y}})_{t_{i}^{n}, t_{i+1}^{n}}|= \mathcal{O}(2^{-3\alpha n}).
    \end{equation*}
    For any $s\in [0,1]$ it then follows
    \begin{equation*}
       |Y'_s- \tilde{Y}'_{s}|+ |Y''_{s}- \tilde{Y}''_{s}|+ |\dot{Y}_{s}- \dot{\tilde{Y}}_{s}|\lesssim  |Y_{s}- \tilde{Y}_{s}|\leq \sum_{i=0}^{2^{n}} |\delta (Y-\tilde{Y})_{s\vee t_{i}^{n}, s \vee t_{i+1}^{n}}|\leq \mathcal{O}(2^{(1-3\alpha)n})\stackrel{n \to \infty}{\longrightarrow}0.
    \end{equation*}
    The argument on $(1, T]$ follows analogously. 
\end{proof}
\begin{theorem}
\label{thm:appendix_sectionB_main}
    Let $\alpha\in (\nicefrac{1}{3}, \nicefrac{1}{2}), k \in \mathbb{N}, \mathbf{X} \in \mathscr{C}^{\alpha}([0,T]; \mathbb{R}^{d_{X}})$ and 
    $$\mathcal{F}=(f, f', \partial f, \partial f', \dots, \partial^{k+1} f)\in \mathscr{D}_{X}^{2\alpha}\mathrm{Lip}^{k+2}_{x}(\mathbb{R}^{d_{Y}}; \mathcal{L}(\mathbb{R}^{d_{X}}; \mathbb{R}^{d_{Y}})).$$ 
    Assume further that one of the following conditions holds: 
    \begin{enumerate}
       \item \label{thm:appendix_sectionB_main:item1}  for any $t\in [0,T]$, $f_{t}\in \mathcal{C}_{b}(W; \mathcal{L}(V; W))$ and $f'_{t} \in \mathcal{C}_{b}W; \mathcal{L}(V^{\otimes 2}; W))$;
       \item \label{thm:appendix_sectionB_main:item2} for any $t \in [0,T]$, $(f_{t}, f'_{t})\in \mathcal{L}(\mathbb{R}^{d_{Y}}; \mathcal{L}(\mathbb{R}^{d_{X}}; \mathbb{R}^{d_{Y}}))\times \mathcal{L}((\mathbb{R}^{d_{X}})^{\otimes 2}; \mathbb{R}^{d_{Y}})$ (linear rough vector fields).
   \end{enumerate}
   Then, for any $\xi\in \mathbb{R}^{d_{Y}}$ there exists a unique $(Y^{s,\xi}, (Y^{s,\xi})')\in \mathscr{D}^{2\alpha}_{X}([0,T]; \mathbb{R}^{d_{Y}})$ such that $(Y^{s,\xi})'= f(Y^{s, \xi})$ and 
    \begin{equation}
    \label{eq:nonautonomous_RDE}
        Y_{t}^{s,\xi}= \xi+ \int_{s}^{t} \left(f_{r}(Y_{r}^{s, \xi}), \partial f_{r}(Y_{r}^{s, \xi}) f(Y_{r}^{s, \xi}) + f'_{r}(Y_{r}^{s,\xi})\right) d\mathbf{X}_{r}. 
    \end{equation}
    Furthermore, for any $s \in [0,T]$, 
    \begin{equation*}
        \mathbb{R}^{d_{Y}}\times [s, T] \ni (\xi, t) \mapsto \phi(s, t; \xi) \coloneqq Y^{s, \xi}_{t}\in \mathbb{R}^{d_{Y}}
    \end{equation*}
    is a flow of  $C^{k}$-diffeomorphisms. 
    \end{theorem}
\begin{proof}
By \cref{thm:appendix_smoothness_flow} the $C^{k}$ property of $\phi(s,t, \cdot)$ follows by evaluation. We are left to show the homeomorphism property of the flow $\mathbb{R}^{d_{Y}}\ni \xi \mapsto \phi(s,t; \xi)\in \mathbb{R}^{d_{Y}}$ and the smoothness of the inverse flow $\phi(s, t; \cdot)^{-1}$. The injectivity follows by considering for fixed $s \in [0,T]$ the $\mathbf{X}$-controlled rough paths $\mathcal{Y}^{i}\coloneqq (Y^{s, \xi_{i}}, f(Y^{s, \xi_{i}}), (Df f)(Y^{s, \xi_{i}})+ f'(Y^{s, \xi_{i}}), 0)$ starting at initial values $\xi_{1}, \xi_{2} \in \mathbb{R}^{d_{Y}}$. Further denote by $\phi(s,t; \xi^{i})$ their respective flows. Then, if for any $t \in [s, T]$ and any $\xi_{1}, \xi_{2}$ it holds $\phi(s,t; \xi_{1})=\phi(s, t; \xi_{2})$, \cref{lemma:appendix:backwards_davie} implies $\xi_{1}=\xi_{2}$.   When condition \eqref{thm:appendix_sectionB_main:item2} is satisfied, i.e., in the case of linear RDEs, the surjectivity follows automatically from the fact that $\xi \mapsto \phi(s,t; \xi)$ is linear and injective. Under \eqref{thm:appendix_sectionB_main:item1} the surjectivity follows by an  argument analogous to that of  \cite[Theorem 4.5.1]{kunita_1997}: Note that, by extending classical \emph{a priori} estimates for RDEs \cite[Proposition 8.2]{friz_2020} to the present situation with non-autonomous vector fields one obtains 
    \begin{equation*}
      \frac{1}{|\phi(s, t; \xi)- \xi|}=\frac{1}{|\delta (\phi(s, \cdot; \xi))_{s,t}|}\geq C((f, f'), \mathbf{X})>0  
    \end{equation*}
    where $C((f, f'), \mathbf{X})>0$ is independent of $\xi$. Letting $\xi \to \infty$ yields that $\phi(s, t; \xi) \to \infty$. Denoting the one-point compactification $\overline{\mathbb{R}^{d_{Y}}}\coloneqq \mathbb{R}^{d_{Y}}\cup \{\infty\}$, we can continuously extend $\phi(s,t; \cdot)$ onto $\overline{\mathbb{R}^{d_{Y}}}$. We denote this extension by $\overline{\phi}(s,t; \cdot) \colon \overline{\mathbb{R}^{d_{Y}}} \to \overline{\mathbb{R}^{d_{Y}}}$. Further note that $\overline{\mathbb{R}^{d_{Y}}}$ is homeomorphic to the $d_{Y}$-sphere $\mathbb{S}^{d_{Y}}\subset \mathbb{R}^{d_{Y}+1}$. Now by the \textit{invariance of domain theorem} from topology we know that every continuous, injective map $ \varphi \colon \mathbb{S}^{d_{Y}}\to \mathbb{S}^{d_{Y}}$ is also surjective and thus is $\overline{\phi}(s,t; \cdot)$ on $\overline{\mathbb{R}^{d_{Y}}}$. Since $\bar \phi(s, t; \infty)=\infty$, the restriction $\phi(s,t; \cdot)$ is surjective on $\mathbb{R}^{d_{Y}}$. Therefore, the flow is bijective. Naturally  for any $\xi \in \mathbb{R}^{d_{Y}}$ the linear operator $D\phi(s,t; \xi)\in \mathcal{L}(\mathbb{R}^{d_{Y}}; \mathbb{R}^{d_{Y}})$ is also invertible, as it is the solution flow to a linear RDE. The smoothness of the inverse flow $\phi(s,t; \cdot)^{-1}$ then follows in combination with \cref{thm:appendix_smoothness_flow} by the inverse function theorem (see \cite[Theorem 19.30]{Driver2003}). 
\end{proof}
\begin{corollary}
\label{appendix:colllection_controlled_fields}
    Let $\alpha\in (\nicefrac{1}{3}, \nicefrac{1}{2}) $, $ k \in \mathbb{N} $, $ \mathbf{X} \in \mathscr{C}^{\alpha, 1+\alpha}([0,T]; \mathbb{R}^{d_{X}})$ and 
    $$\mathcal{F}=(f, f', \partial f, \partial f', \dots, \partial^{k+1} f)\in \mathscr{D}_{X}^{2\alpha}\operatorname{Lip}^{k+2}_{x}(\mathbb{R}^{d_{Y}}; \mathcal{L}(\mathbb{R}^{d_{X}}; \mathbb{R}^{d_{Y}})),$$ such that the assumptions of \cref{thm:appendix_sectionB_main} are satisfied. Then denoting by $\phi(s,t; \xi)$ the solution to \eqref{eq:nonautonomous_RDE} and $\phi_{t}(\xi)\coloneqq \phi(0, t; \xi), \overleftarrow{\phi}_{t}(\xi)\coloneqq \phi(t, T; \xi)$ and $\phi_{t}^{-1}(\xi)\coloneqq \phi^{-1}(0, t; \xi)$ we define
    \begin{equation*}
        \begin{aligned}
         \Phi&\coloneqq (\phi , f \circ \phi, D\phi, (\Gamma f+ f') \circ \phi, D(f\circ \phi), D^2\phi , 0), \\
         \overleftarrow{\Phi}&\coloneqq \left(\overleftarrow{\phi},  - \Gamma \overleftarrow{\phi}, D\overleftarrow{\phi}, (\Gamma^{2} 
        \overleftarrow{\phi}+ \Gamma' \overleftarrow{\phi})^{\top}, -D(\Gamma \overleftarrow{\phi}), D^{2} \overleftarrow{\phi}, -\frac{1}{2} (D^{2} \overleftarrow{\phi})(f, f)\dot{[\mathbf X]}\right),\\
        \Phi^{-1}&\coloneqq \left(\phi^{-1},  - \Gamma \phi^{-1}, D\phi^{-1}, (\Gamma^{2} 
        \phi^{-1}+ \Gamma' \phi^{-1})^{\top}, -D(\Gamma \phi^{-1}), D^{2} \phi^{-1}, -\frac{1}{2} (D^{2} \phi^{-1})(f, f)\dot{[\mathbf X]}\right),
        \end{aligned}
    \end{equation*}
    where $\Gamma (\cdot) \coloneqq D(\cdot) f$, $\Gamma'(\cdot) \coloneqq D(\cdot) f'$. Then it holds $\Phi, \overleftarrow{\Phi}, \Phi^{-1} \in \mathscr{D}^{3 \alpha}_{\mathbf{X}} \operatorname{Lip}^{3}_{x}(\mathbb{R}^{d_{Y}}; \mathbb{R}^{d_{Y}})$.
\end{corollary}
\begin{proof}
    The claim for $\Phi$ follows from \cref{thm:appendix_smoothness_flow} and \cref{lemma:jet_criterion}. The corresponding result  for $\overleftarrow{\Phi}$ is obtained by an argument analogous to the one used in \cref{ex:RDE_flows}. Finally, the claim for $\Phi^{-1}$ follows by noting that  $\phi^{-1}_{t}= \phi^{-1}(0, T; \cdot)\circ \overleftarrow{\phi}_{t}$ and by applying \cref{thm:appendix_sectionB_main} together with \cref{cor:composition_rule}.
\end{proof}
\printbibliography
\end{document}